\documentclass[preprint,3p]{amsart}
\usepackage[margin=1.2in]{geometry}
\usepackage{wrapfig}
\usepackage{amssymb}

\usepackage{float}

\usepackage{stmaryrd}
\usepackage{amsmath,amsfonts,amsthm,amstext}
\usepackage[latin1]{inputenc}
\usepackage{fancybox}
\usepackage{niceframe}
\usepackage{array}
\usepackage{newlfont}
\usepackage{verbatim}
\usepackage[dvips]{psfrag}
\usepackage{color}
\usepackage{float}
\usepackage[scriptsize]{caption}
\usepackage{indentfirst}
\usepackage[normalem]{ulem}
\usepackage{pst-text}
\usepackage{graphicx}
\usepackage{graphics}
\usepackage{overpic}
\usepackage{hyperref}
\usepackage{enumerate}
\usepackage{pst-node}
\usepackage{tikz-cd} 

\graphicspath{{./Figures/}}
\usepackage{wrapfig}
\newcommand{\R}{\ensuremath{\mathbb{R}}}
\newcommand{\C}{\ensuremath{\mathbb{C}}}
\newcommand{\N}{\ensuremath{\mathbb{N}}}

\newcommand{\Z}{\ensuremath{\mathbb{Z}}}

\newcommand{\s}{\Sigma}

\newcommand{\de}{\delta}
\newcommand{\e}{\varepsilon}

\newcommand{\rn}[1]{\mathbb{R}^{#1}}

\newcommand{\er}{\mathcal{O}}

\newcommand{\ag}{\alpha}
\newcommand{\bg}{\beta}

\newcommand{\dg}{\delta}

\newtheorem {theorem} {Theorem} [section]
\newtheorem {prop}[theorem]{Proposition}

\newtheorem {lemma} [theorem] {Lemma}

\newtheorem {remark} [theorem] {Remark}

\makeatletter
\DeclareFontFamily{U}{tipa}{}
\DeclareFontShape{U}{tipa}{m}{n}{<->tipa10}{}
\newcommand{\arc@char}{{\usefont{U}{tipa}{m}{n}\symbol{62}}}%

\newcommand{\arc}[1]{\mathpalette\arc@arc{#1}}

\newcommand{\arc@arc}[2]{%
	\sbox0{$\m@th#1#2$}%
	\vbox{
		\hbox{\resizebox{\wd0}{\height}{\arc@char}}
		\nointerlineskip
		\box0
	}%
}
\makeatother

\DeclareMathOperator{\sech}{sech}

\DeclareMathOperator{\arcsinh}{arcsinh}

\DeclareMathOperator{\Rp}{Re}
\DeclareMathOperator{\Ip}{Im}  
\definecolor{verde}{rgb}{0.0,0.5,0.0}
\definecolor{azul}{rgb}{0,0,128}
\definecolor{roxo}{rgb}{0.44,0.16,0.39}
\definecolor{vinho}{rgb}{0.5,0.0,0.13}
\definecolor{lilas1}{rgb}{0.6,0.33,0.73}
\definecolor{rosa}{rgb}{0.84,0.04,0.33}
\definecolor{mostarda}{rgb}{0.91,0.41,0.17}
\definecolor{mostarda2}{rgb}{1.0,0.66,0.07}

\newtheorem {mtheorem} {Theorem} 

\begin{document}


\title[Critical velocity in kink-defect interaction models]{Critical velocity in kink-defect interaction models: Rigorous results}

\author[O. M. L. Gomide]{Ot\'avio M. L. Gomide}
\address[OMLG]{Department of Mathematics, Unicamp, IMECC\\ Campinas-SP, 13083-970, Brazil}
\email{otaviomleandro@hotmail.com}

\author[M. Guardia]{Marcel Guardia}
\address[MG]{ Departament de Matemàtica Aplicada I, Universitat Politècnica de Catalunya, Diagonal 647, 08028 Barcelona, Spain}
\email{marcel.guardia@upc.edu}

\author[T. M. Seara]{Tere M. Seara}
\address[TS]{Departament de Matemàtica Aplicada I, Universitat Politècnica de Catalunya, Diagonal 647, 08028 Barcelona, Spain }
\email{tere.m-seara@upc.edu }


\maketitle

\begin{abstract}
In this work we study a model of interaction of kinks of the sine-Gordon equation with a weak defect. 
We obtain rigorous results concerning the so-called critical velocity derived in \cite{GH04} by a geometric approach. 
More specifically, we prove that a heteroclinic orbit in the energy level $0$ of a $2$-dof Hamiltonian $H_\e$ is destroyed giving rise to heteroclinic 
connections between certain elements (at infinity) for exponentially small (in $\e$) energy levels. 
In this setting Melnikov theory does not apply because there are exponentially small phenomena.
\end{abstract}

\section{Introduction}


Given an evolutionary partial differential equation, a traveling wave is a solution which travels with constant speed and shape. 
There are several types of traveling waves which are important in modeling physical phenomena. 
In particular, we give special attention to kinks, also referred as solitons.
A soliton is a spatially localized traveling wave which usually appears as a result of a balance between a nonlinearity and dispersion. 
In fact, kinks are traveling waves which travel from one asymptotic state to another. 
In the last years, solitons have attracted the focus of researchers due to their significant role in many scientific fields as optical fibers, 
fluid dynamics, plasma physics and others (see \cite{GSW02,KRZ86, W10} and references therein).

In this work, we study a model of interaction between kinks (traveling waves) of the sine-Gordon equation and a weak defect. 
The defect is modeled as a small perturbation given by a Dirac delta function. Such interaction has also been studied for the nonlinear Schrödinger equation in \cite{HMZ07,HMZ072}.

We consider the finite-dimensional reduction of the equation given by a 2-degrees of freedom Hamiltonian $H$ proposed in \cite{FKV92,GH04}. Following a geometric approach, we  give conditions on the energy of the system to admit kink-like solutions. 


\subsection{The model}\label{motivation}


%


The sine-Gordon equation is a nonlinear hyperbolic partial differential equation given by
\begin{equation}\label{sineGordonn}
\partial_t^2u-\partial_x^2u+\sin(u)=0,
\end{equation}
which presents a family of kinks $u_\mathtt{k}(x,t)$ given by
\begin{equation}\label{kinks}
u_\mathtt{k}(x,t)=4 \arctan\left(\exp\left(\dfrac{x-vt-x_0}{\sqrt{1-v^2}}\right)\right),
\end{equation}
where the parameter $v$ represents the velocity of the kink. 
%

In this work, we perturb this equation by a localized nonlinear defect at the origin
\begin{equation}\label{sineGordon}
\partial_t^2u-\partial_x^2u+\sin(u)=\e \dg(x)\sin(u),
\end{equation}
where $\dg(x)$ is the Dirac delta function.
This equation was studied in \cite{FKV92,GH04} 
where the authors consider finite-dimensional reductions of it to understand the kink-like dynamics.
As a first step, they consider solutions $u$ of small amplitude of \eqref{sineGordon}, which can be approximated by solutions of the linear partial differential equation
\begin{equation}\label{sineGordonsmall}
\partial_t^2u-\partial_x^2u+u=\e \dg(x)u,
\end{equation}
which has a family of wave solutions $u_\mathtt{im}(x,t)$ given by
\begin{equation}\label{im}
u_\mathtt{im}(x,t)=a(t)e^{-\e|x|/2},
\end{equation}
where $a(t) = a_0 \cos (\Omega t + \theta_0)$, $\Omega=\sqrt{1-\e^2/4}$ and $\mathtt{im}$ stands for impurity. 
The solution $u_\mathtt{im}$ is not a traveling wave, but it is spatially localized at $x=0$.

In order to study the interaction of kinks of the sine-Gordon equation with the defect considered in \eqref{sineGordon}, 
\cite{FKV92,GH04} use variational approximation techniques to obtain the equations which describe the evolution of the kink position $X$ and the defect mode amplitude $a$. 
To derive such equations, they consider the ansatz
%
%
%
%
%
\begin{equation}
\label{ansatz}
u(x,t)=4\arctan(\exp(x- X(t))) +a(t)e^{-\e|x|/2} .
\end{equation}
Notice that \eqref{ansatz} combines the traveling property of the family of kinks \eqref{kinks} with the localized shape of \eqref{im}. 
If 
\begin{equation}
X(t)=\dfrac{vt-x_0}{\sqrt{1-v^2}}\textrm{ and }a(t)\equiv 0,
\end{equation}
then \eqref{ansatz} becomes the original family of kinks \eqref{kinks} of \eqref{sineGordon} for $\e=0$.


Using the ansatz \eqref{ansatz} in \eqref{sineGordon} and considering terms up to order $2$ in $\e$, \cite{FKV92,GH04} obtain
the system of Euler-Lagrange equations
\begin{equation}\label{system-ns}
\left\{
\begin{array}{l}\vspace{0.3cm}
8\ddot{X}+\e U'(X)+\e a F'(X)=0,\\
\ddot{a}+\Omega^2a+\frac{1}{2}\e^2F(X)=0,
\end{array}
\right.
\end{equation}
 where
\begin{equation}\label{exprUF}
U(X)=-2 \sech^2(X),\  F(X)=-2\tanh(X)\sech(X) \textrm{ and } \Omega=\sqrt{1-\dfrac{\e^2}{4}},
\end{equation}
which describes approximately the evolution of the kink position $X$ and the defect mode amplitude $a$.
More details of this approach and its applications can be found in \cite{FKV92,GH04,M02}. It is worth to mention that the finite dimensional reduction of PDE problems to ODE systems via an adequate ansatz and variational methods has been considered in an extensive range of works (see \cite{EGKJ15, GMHK11, GH05,GH052,GHW04,W14,ZY07}).
	
It remains as an open problem to prove that the solutions of the reduced system rigorously approximate the PDE solutions. Nevertheless there are numerical evidences ensuring this reasoning (see \cite{PZ06,PZ07}). In particular, in \cite{WJB14}, the authors analyze numerically the simulations done in \cite{GH04}  for the perturbed sine-Gordon equation \eqref{sineGordon}. 


From \eqref{ansatz}, if $X(t)$ and $a(t)$ satisfy $X(t)\rightarrow \pm\infty$, $\dot{X}(t)\rightarrow C^{\pm}$ and  $a(t)\rightarrow 0$ as $t\rightarrow\pm\infty$, 
then $u(x,t)$ can be seen as an approximation for a kink of \eqref{sineGordon}, 
since it transitions from an asymptotic state  to another when $x-X(t)\rightarrow \pm\infty$. 
In this case, we say that $(X(t),a(t))$ is a \textbf{kink-like solution}, or simply a kink, of \eqref{system-ns}, 
and we say that $v_i= C^-$ and $v_f=C^+$ are the \textbf{initial velocity} and \textbf{final velocity} of the kink.

If $X(t)$ satisfy $X(t)\rightarrow \pm\infty$, $\dot{X}(t)\rightarrow C^{\pm}$ and  $a(t)$ is asymptotic to a periodic function
with small amplitude when $t\rightarrow+\infty$ of $t\rightarrow -\infty$, then $u(x,t)$ can be seen as an approximation for a kink of \eqref{sineGordon} with asymptotically periodic oscillations. 
In this case, we say that $(X(t),a(t))$ is an \textbf{oscillating kink-like solution}, or simply an oscillating  kink, of \eqref{system-ns}, 
and their initial and final velocities are defined in the same way. 
In addition, if $(X(t),a(t))$ is an oscillating kink such that $a(t)\rightarrow 0$ as $t\rightarrow-\infty$ and $a(t)$ is asymptotically periodic as $t\rightarrow+\infty$, 
then it is said to be a \textbf{quasi kink-like solution}, or quasi kink.


In this paper we perform a rigorous study of such solutions of the finite-dimensional reduction \eqref{system-ns} of the  partial differential equation \eqref{sineGordon}.


\subsection{The reduced model}\label{model}

Consider the change of variables $(X,\dot{X},a,\dot{a})\rightarrow (X,Z,b,B),$ where
\begin{equation}
X=X, Z=\dfrac{8\dot{X}}{\sqrt{\e}},\ b=\sqrt{\dfrac{2\Omega}{\e}}\e^{-1/4}a,  B=\sqrt{\dfrac{\e}{2\Omega}} \e^{-1/4}\dfrac{2}{\e}\dot{a},
\end{equation}
and the time rescaling $\tau=\sqrt{\e}t$. 
Then, denoting $'=d/d\tau$, the evolution equations of  \eqref{system-ns} are equivalent to:
\begin{equation}\label{rescaled_aut_system}
\left\{
\begin{array}{l}
X'=\vspace{0.2cm} \dfrac{Z}{8},\\
Z'=\vspace{0.2cm} - U'(X) -\dfrac{\e^{3/4}}{\sqrt{2\Omega}} F'(X) b,\\
b'=\vspace{0.2cm}\dfrac{\Omega}{\sqrt{\e}} B,\\
B'=\vspace{0.2cm}-\dfrac{\Omega}{\sqrt{\e}}b -\dfrac{\e^{3/4}}{\sqrt{2\Omega}} F(X),
\end{array}
\right. \textrm{ with }\Omega=\sqrt{1-\dfrac{\e^2}{4}}.
\end{equation}
Notice that \eqref{rescaled_aut_system} is a Hamiltonian system with respect to
\begin{equation}\label{Hamiltonian}
H(X,Z,b,B;\e)=\dfrac{Z^2}{16}+U(X)+ \dfrac{\Omega}{2\sqrt{\e}}(B^2+b^2)+\dfrac{\e^{3/4}}{\sqrt{2\Omega}} F(X) b,
\end{equation}
which can be split as
$H=H_{\mathrm{p}}+ H_{\mathrm{osc}}+ R,$
where
\begin{equation}
\left\{\begin{array}{l}
H_{\mathrm{p}}(X,Z)= \dfrac{Z^2}{16}+U(X),\vspace{0.2cm}\\ 
H_{\mathrm{osc}}(b,B)= H_{\mathrm{osc}}(b,B;\e)= \dfrac{\Omega}{2\sqrt{\e}}(B^2+b^2),\vspace{0.2cm}\\ 
R(X,b)=R(X,b;\e)=\dfrac{\e^{3/4}}{\sqrt{2\Omega}} F(X) b.
\end{array}\right.
\end{equation}
Thus the Hamiltonian $H$ is the sum of a pendulum-like Hamiltonian $H_{\mathrm{p}}$ with an oscillator $H_{\mathrm{osc}}$ coupled by the term $R$.

\begin{remark}\label{parabolic_rem}
	
	Applying the change of variables $Y=4\arctan(e^X)$, the Hamiltonian system \eqref{rescaled_aut_system} is brought into 
	$$
	\left\{
	\begin{array}{lcl}
	\dot{Y}=2\sin(Y/2)Z/8,\vspace{0.2cm}\\
	\dot{Z}=2\sin(Y/2)\left(\sin(Y)-\dfrac{\e^{3/4}}{\sqrt{2\Omega}}\cos(Y)b\right),\vspace{0.2cm}\\
	\dot{b}=\vspace{0.2cm}\dfrac{\Omega}{\sqrt{\e}}B,\\
	\dot{B}=\vspace{0.2cm}-\dfrac{\Omega}{\sqrt{\e}} b -\dfrac{\e^{3/4}}{\sqrt{2\Omega}} \sin(Y).	
	\end{array}
	\right.
	$$
	
	When  $Y=0$ and $Y=2\pi$, this system has parabolic critical points and periodic orbits which have invariant manifolds.
	
	The hyperplanes $Y=0$ and $Y=2\pi$ correspond to $X=-\infty$ and $X=+\infty$ of \eqref{rescaled_aut_system} respectively.
	For this reason, even if they are not solutions of the system, they can be seen as asymptotic solutions at infinity.
	Thus, abusing notation, we denote $f(\pm\infty)$ as $\displaystyle\lim_{X\rightarrow \pm\infty} f(X)$ when it is well defined.

	

\end{remark}


System \eqref{rescaled_aut_system} inherits many properties of the sine-Gordon equation. 
In fact, the functions $U$ and $F$ have exponential decay when $|X|\rightarrow +\infty$, 
therefore, for large values of $X$ the system becomes decoupled. 
Nevertheless, when $X=\er(1)$, the equations are coupled and the Hamiltonians $H_\mathrm{p}$ and $H_{\mathrm{osc}}$ may exchange energy, 
and this will result in interesting global phenomena.

If $F=0$ (i.e. $R=0$), then each energy level $H=h\geq0$ of system \eqref{system-ns} contains a unique kink solution and all the other 
solutions will be oscillating kinks (with the same oscillation in both tails). 
In this paper, we prove that the kink solution in $H=h$ breaks down for low energies (see Theorem \ref{splitting_thmA}) and we obtain a 
\textbf{critical energy } $h_c$ (with associated critical initial velocity $v_c= 4\sqrt{h_c}$) such that the energy level $H=h$ ($h$ small) 
contains a quasi kink (continuation of an unperturbed kink) if and only if $h\geq h_c$. 
In addition we give an asymptotic formula for $h_c$ (see Theorem \ref{perpun_thm}) which happens to be exponentially small in the parameter $\e$. 
We also find an  energy  $0<h_s<h_c$ such that the energy level $H=h$ ($h$ small) has oscillating kinks if and only if $h\geq h_s$ (see Theorem \ref{perper_thm}).

In \cite{GH04}, the authors present numerical and formal arguments for the existence of the critical velocity $v_c$ and they conjecture 
that the final velocity $v_f$ of a quasi kink lying in an energy level $h\geq h_c$ ($h$ small) is given by $v_f\approx (v_i-v_c)^{1/2}$, 
where $v_i\geq v_c$ is its initial velocity. 
Our results prove the validity of the asymptotic formula for $v_c$ and the conjecture for $v_f$
(see Theorem \ref{perpunh_thm}). 

We emphasize that the rigorous approach presented in this work is necessary to validate the conclusions obtained in \cite{GH04}. In fact, their results rely on the computation of a Melnikov integral as a first order for the total loss of energy  $\Delta E$ over the separatrix of \eqref{rescaled_aut_system} with $\dg=0$ (or more precisely of the transfer of energy from the separatrix to the oscillator). Nevertheless, Melnikov theory cannot be applied in this case due the exponentially smallness in the parameter $\e$ of the Melnikov function. In this paper we prove that it is  indeed a first order of  $\Delta E$. Note that this is not always the case:  in general problems presenting exponentially small phenomena, often the Melnikov integral is not the dominant part of the total loss of energy over a separatrix of a Hamiltonian system (see \cite{BFGS12}).    


In this paper, we relate the loss of energy $\Delta E$, and thus the existence of kinks, quasi kinks and oscillating kinks,  with  the exponentially small transversal intersection of the invariant manifolds $W^{u,s}$ of certain objects (critical points and periodic orbits) at infinity.

\section{Mathematical Formulation and Main Goal}


\subsection{The unperturbed Problem}\label{geometric}

Consider system \eqref{rescaled_aut_system} for $F=0$.
Then $H=H_\mathrm{p}+ H_{\mathrm{osc}}$ is just two uncoupled integrable systems.


In the $XZ$-plane, the solutions are contained in the level curves $H_{\mathrm{p}}(X,Z)=\kappa$. 
This system can be transformed into a degenerate (parabolic) pendulum by a change of coordinates (see Remark \ref{parabolic_rem}).
For $\kappa<0$, $H_{\mathrm{p}}=\kappa$ is diffeomorphic to a circle. 
For $\kappa\geq0$, $H_{\mathrm{p}}=\kappa$ contains the points $q_{\kappa}^{\pm}=(\pm\infty, 4\sqrt{\kappa})$ 
which behave as ``fixed points" and are connected by a heteroclinic orbit $\Upsilon_{\kappa}$ given by the graph of
\begin{equation}
\label{Zk+}
Z_{\kappa}(X)=4\sqrt{\kappa-U(X)}=4\sqrt{\kappa+\dfrac{2}{\cosh^2(X)}},\ \  X\in\R.
\end{equation}
Notice that $\varUpsilon_0$ is a separatrix.
Analogously, $(\pm\infty,-4\sqrt{\kappa})\in \{H_{\mathrm{p}}=\kappa\}$ are fixed points at infinity connected by the heteroclinic orbit 
given by the graph of $-Z_{\kappa}(X)$. See Figure \ref{projXZ_fig}. 
%
\begin{figure}[!]	
	\centering
	\bigskip
	\begin{overpic}[width=8cm]{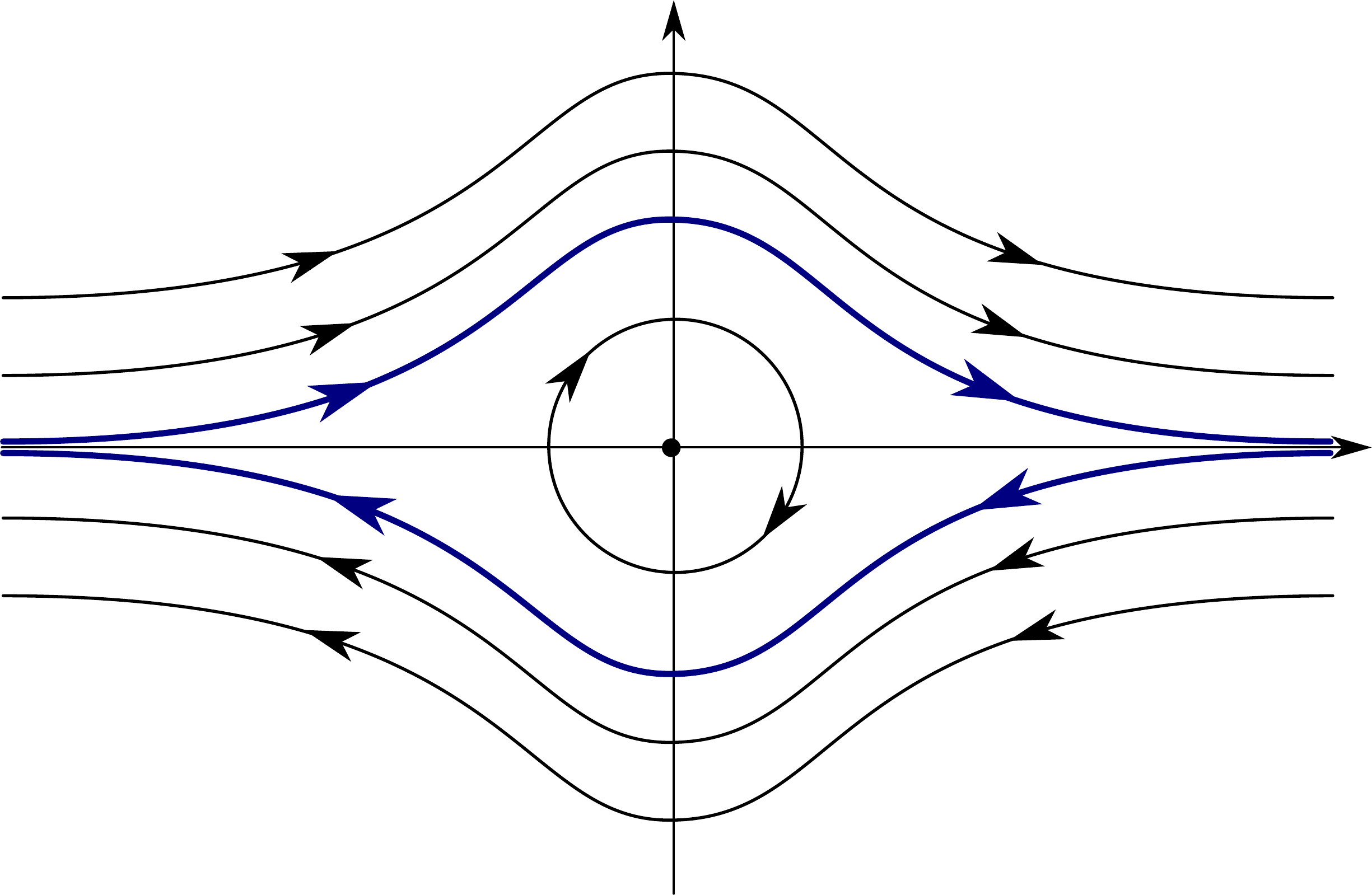}
	\put(51,62){{\footnotesize $Z$}}	
	\put(101,31){{\footnotesize $X$}}		
	\end{overpic}
	\bigskip
	\caption{ Projection of the phase space of the unperturbed system in the $XZ$-plane. }	
		\label{projXZ_fig}
\end{figure} 
From now on, we focus on the heteroclinic orbits contained in $Z>0$, since all the results of this paper can be obtained for the orbits in $Z<0$ in an analogous way.



In the $bB$-plane, the solutions of \eqref{rescaled_aut_system} for $F=0$ are 
\begin{equation}\label{pk}
P_\kappa=\{ H_{\mathrm{osc}}=\kappa \}=\left\{(b, B);\ b^2+B^2=2\kappa\sqrt{\e}/\Omega\right\} \textrm{ (see Figure \ref{bB_fig}).}
\end{equation}

\begin{figure}[!]
	\centering
	\bigskip
	\begin{overpic}[width=5cm]{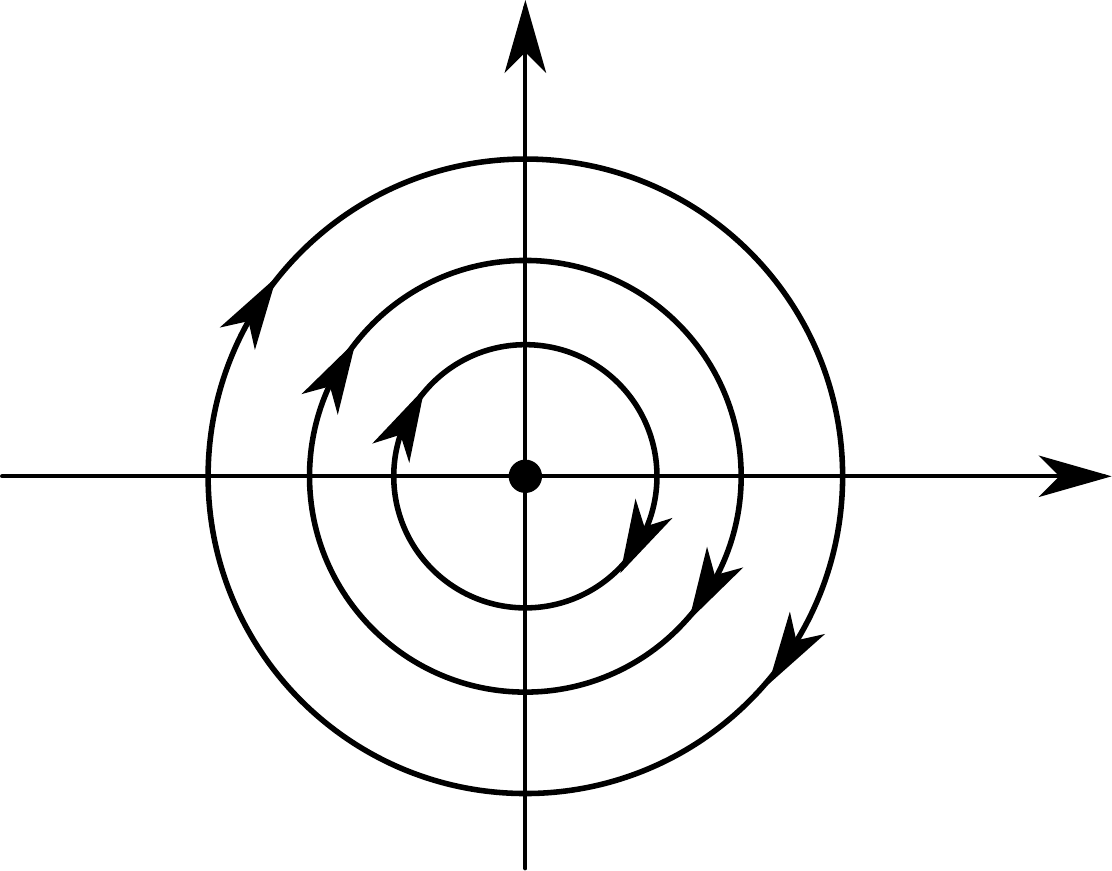}
	\put(50,73){{\footnotesize $B$}}	
\put(102,33){{\footnotesize $b$}}				
	\end{overpic}
	\bigskip
	\caption{ Projection of the phase space of the unperturbed system in the $bB$-plane. }	
	\label{bB_fig}
\end{figure} 


Combining \eqref{Zk+} and \eqref{pk} in the energy level $H=h,$ we define 
\begin{equation}
\Lambda^{\pm}_{\kappa_1,\kappa_2}=q_{\kappa_1}^{\pm}\times P_{\kappa_2}=\left\{(\pm\infty, 4\sqrt{\kappa_1}, b, B);\ b^2+B^2=2\kappa_2/\omega\right\},
\end{equation}
for every $\kappa_1,\kappa_2\geq 0$ such that $\kappa_1+\kappa_2=h$. 
Notice that
\begin{itemize}
	\item  If $\kappa_2=0$, then $\Lambda^{\pm}_{h,0}$ is a  degenerate saddle (parabolic) point of \eqref{rescaled_aut_system};
	\item If $\kappa_2>0$, then $\Lambda^{\pm}_{\kappa_1,\kappa_2}$
are  degenerate saddle (parabolic)  periodic orbits of \eqref{rescaled_aut_system}.	
\end{itemize}


For simplicity, we denote the limit cases $\kappa_1=0$ and $\kappa_2=0$ by
\begin{equation}
\label{extremum}
\begin{array}{lcl}
\Lambda_h^{\pm}&=&\Lambda^{\pm}_{0,h}=\left\{(\pm\infty,0,b,B), b^2+B^2=2h/\omega\right\},\vspace{0.2cm}\\
p_h^{\pm}&=&\Lambda^{\pm}_{h,0}=(\pm\infty,4\sqrt{h},0,0),
\end{array}
\end{equation}
respectively. We stress that $p_h^{\pm}$ are points and $\Lambda_h^{\pm}$ are periodic orbits, 
both contained in the planes $X=\pm\infty$ and in the energy level $H=h$.

These invariant objects have invariant manifolds. Denote 
\begin{equation}\label{invmanifolds}
	W(\kappa_1,\kappa_2)=\varUpsilon_{\kappa_1}^{}\times P_{\kappa_2}=\left\{(X,Z,b,B);\ Z=4\sqrt{\kappa_1-U(X)} \textrm{ and }b^2+B^2=2\kappa_2\sqrt{\e}/\Omega  \right\},
\end{equation}
for each $\kappa_1,\kappa_2\geq 0$ such that $\kappa_1+\kappa_2=h$. 
%
%
\begin{itemize}
	\item[(\textbf{1D-0})] 
	$W(0,0)= W^u_0(p_0^-)=W^s_0(p_0^+)$ is a $1$-dimensional heteroclinic connection (separatrix) between the points $p_0^-$ and $p_0^+$;
	\item[(\textbf{1D-$\kappa_1$})] 
	$W(h,0)= W^u_0(p_h^-)=W^s_0(p_h^+)$ is a $1$-dimensional heteroclinic connection between the points $p_h^-$ and $p_h^+$;
	\item[(\textbf{2D-0})] 
	If $h>0$, then $W(0,h)= W^u_0(\Lambda_h^-)=W^s_0(\Lambda_h^+)$ is a $2$-dimensional heteroclinic manifold  (separatrix) between $\Lambda_h^{-}$ and $\Lambda_h^{+}$;		
	\item[(\textbf{2D-$\kappa_1$})] 
	If $\kappa_1,\kappa_2> 0$, then $W(\kappa_1,\kappa_2)$ is a $2$-dimensional heteroclinic manifold between $\Lambda^{-}_{\kappa_1,\kappa_2}$ and $\Lambda^{+}_{\kappa_1,\kappa_2}$.
\end{itemize}

For $h>0$ fixed, the level energy $H=h$ is a $3$-dimensional  manifold. 
Eliminating the variable $Z$ by the Hamiltonian conservation, the manifolds $W(\kappa_1,\kappa_2)$ project into the the $bXB$-space as horizontal 
cylinders centered along the $X$-axis.
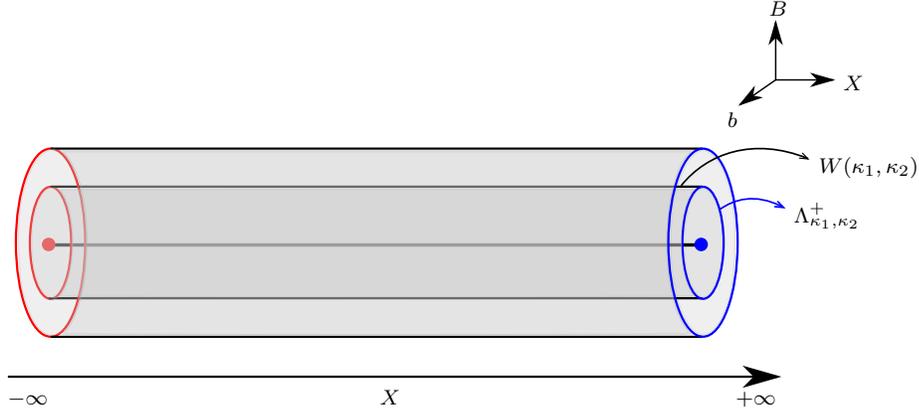
\begin{figure}[!]	\label{cylinders}
	\centering
	\begin{overpic}[width=11cm]{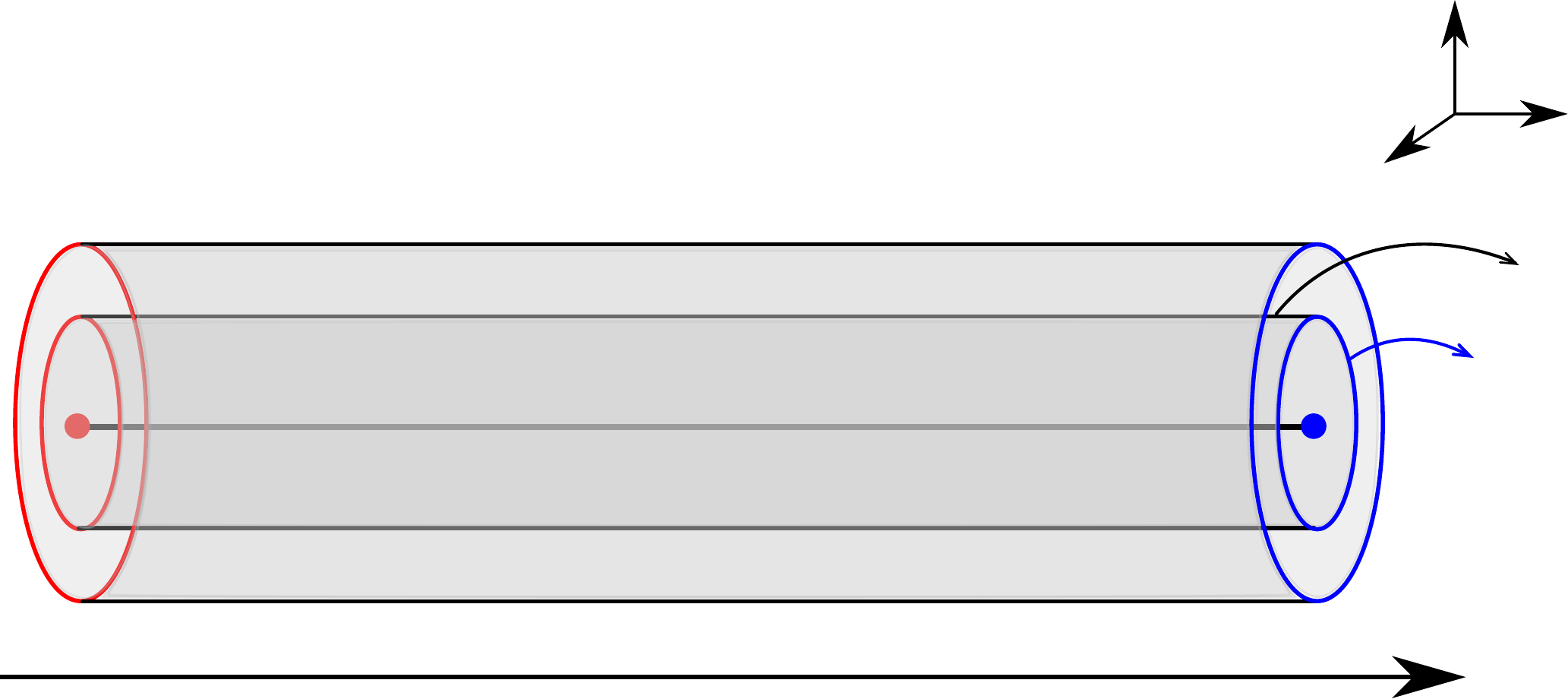}
		\put(98,26){{\footnotesize $W(\kappa_1,\kappa_2)$}}		
		\put(95,20){{\footnotesize $\Lambda_{\kappa_1,\kappa_2}^+$}}	
		\put(0,-2){{\footnotesize $-\infty$}}				
		\put(88,-2){{\footnotesize $+\infty$}}	
		\put(45,-2){{\footnotesize $X$}}	
		\put(101,36){{\footnotesize $X$}}	
		\put(87,31.5){{\footnotesize $b$}}	
		\put(92,45){{\footnotesize $B$}}																		
	\end{overpic}
	\bigskip
	\caption{ Projection of the heteroclinic manifolds $W(\kappa_1,\kappa_2)$ in the $bXB$-space. In the figure, the most external cylinder is the projection of $W(0,h)$ and the straight line represents the projection of $W(h,0)$.}	
\end{figure} 

%
	

In this unperturbed case, there is no exchange of energy between the pendulum and the oscillator through the heteroclinic connections of 
$W(\kappa_1,\kappa_2)$, i.e. $H_\mathrm{p}$ and $H^\e_{\mathrm{osc}}$ are first integrals.
In the perturbed case \eqref{rescaled_aut_system} ($F\neq0$)
the coupling term $R$ (see \eqref{Hamiltonian}) goes to $0$ as $X\rightarrow\pm\infty$, thus, the system is uncoupled at $X=\pm\infty$. 
As a consequence, $\Lambda^{\pm}_{\kappa_1,\kappa_2}$ are orbits of system \eqref{rescaled_aut_system} in the sense of Remark \ref{parabolic_rem}.
Nevertheless, the system may exchange energy between the pendulum and the oscillator when $X$ varies, 
through the appearance of heteroclinic connections between different $\Lambda^-_{\kappa_1,\kappa_2}$ and 
$\Lambda^+_{\kappa_1',\kappa_2'}$ such that $\kappa_1+\kappa_2=\kappa'_1+\kappa'_2=h$.

Recall that a quasi-kink (see Section \ref{motivation}) is a solution 
$(X(t),Z(t),b(t),B(t))$ which has initial velocity $v_i>0$ and final velocity $v_f>0$ and satisfies the asymptotic conditions
\begin{align}
\displaystyle\lim_{t\rightarrow-\infty} X(t)&=-\infty,\ \displaystyle\lim_{t\rightarrow-\infty} Z(t)=v_i,
\ \displaystyle\lim_{t\rightarrow-\infty} b(t)=\displaystyle\lim_{t\rightarrow-\infty} B(t)=0,\label{asympc}\\
 \displaystyle\lim_{t\rightarrow+\infty} X(t)&=+\infty,\ \displaystyle\lim_{t\rightarrow+\infty} Z(t)=v_f.\label{asympf}
\end{align}
 and $(b(t), B(t))$ are asymptotic to periodic functions as $X\to+\infty$. For such solutions
$$h_i=H(X(t),Z(t),b(t),B(t))=\dfrac{v_i^2}{16},$$
for every $t\in\R$.

Thus, considering $h_i=v_i^2/16$ and $\kappa_f= v_f^2/16$, we have that, the quasi-kink solution $(X(t),Z(t),b(t),B(t))$ satisfying \eqref{asympc} and \eqref{asympf} is a heteroclinic connection between the $1$-dimensional unstable manifold of $p_{h_i}^-$ and the $2$-dimensional stable manifold of $\Lambda^+_{\kappa_f,h_i-\kappa_f}$. 

\subsection{Main results}\label{goal}
Our aim is to look for solutions traveling from $X=-\infty$ to $X=+\infty$. More concretely, we prove the existence of $v_c>0$ such that the solutions  $X$ of \eqref{system-ns} incoming with velocity $v_i$ escape the defect location and continue traveling towards $X+=\infty$ with (asymptotic) final velocity $v_f$, provided $v_i\geq v_c$.



%

Therefore, the critical energy $h_c$ is characterized as the lowest energy level $h_c=v_c^2/16$ such that for any $h\geq h_c$, there exist $\kappa_1,\kappa_2>0$ with $\kappa_1+\kappa_2=h$ such that $W_{\e}^u(p_{h}^-)\subset W_{\e}^s(\Lambda^+_{\kappa_1,\kappa_2})$. 

Notice that $W_{\e}^u(p_{h}^-)\subset W_{\e}^s(\Lambda^+_{\kappa_1,\kappa_2})$ implies that the final velocity of the corresponding orbit $X(t)$ (which has initial velocity $4\sqrt{h}$) is given by $v_f=4\sqrt{\kappa_1}$. 

To analyze the existence of heteroclinic orbits between the invariant objects at $X=\pm\infty$ we consider the section $X=0$, which is transversal to the flow. Restricting  to the energy level $H=h$, eliminating the variable $Z$ and using \eqref{exprUF}, this section becomes  the disk
\begin{equation}
\label{sh}
\Sigma_h =\left\{(0,b,B); b^2+B^2\leq \dfrac{(4+2h)\sqrt{\e}}{\Omega} \right\}.
\end{equation}
We  compute intersections between unstable and stable manifolds in $\Sigma_h$.

%




In the unperturbed case $F=0$,  the one-dimensional heteroclinic connection between the ``infinity points" $p_h^+$ and $p_h^-$, $W(h,0)=W_0^u(p_h^-)=W_0^s(p_h^+)$  intersect $\Sigma_h$ at the point $(0,0)$. 
In the following theorem, we show that it breaks down when $F\neq 0$ (see Figure \ref{splitting_fig}).

\begin{mtheorem}[Breakdown of kinks]\label{splitting_thmA}
	Consider system \eqref{rescaled_aut_system}. There exists $\e_0>0$ and $h_0>0$ sufficiently small such that, for every $0<\e<\e_0$ and $0\leq h\leq h_0$, the invariant manifolds $W^{u,s}_{\e}(p_h^{\mp})$ intersect $\s_0$ (given in \eqref{sh}). Denoting by $P_h^{u,s}$ the first intersection points,
	\begin{equation}\label{dif}
	\begin{split}
	\left|P_0^u - P_0^s\right|&=d_0(\e)=\dfrac{2\pi  \e^{3/4}}{\sqrt{\Omega}}e^{-\Omega\sqrt{2/\e}} +\er\left(\e^{7/4}e^{-\Omega\sqrt{2/\e}}\right), \textrm{ where }\Omega=\sqrt{1-\dfrac{\e^2}{4}}\\
	\left|P_h^u - P_h^s\right|&=d_0(\e)+\er(\e^{7/4}\sqrt{h}).
	\end{split}
	\end{equation}

\end{mtheorem}
The first statement of this theorem is proven in Section \ref{breakdown} and the second one is a consequence of Theorem \ref{approx} stated in Section \ref{parametrizaciones} below.

\begin{remark}
	In the asymptotic formula \eqref{dif}, we could write $\Omega=1$. Nevertheless, we keep $\Omega=\sqrt{1-\e^2/4}$ in order to compare our results with \cite{GH04}. The same remark holds for Theorems $B$, $C$ and $D$ below.
\end{remark}
%

When $F=0$, the energy level $h$ has a family of heteroclinic manifolds $W(\kappa_1,\kappa_2)$, with $\kappa_1+\kappa_2=h$, $\kappa_1,\kappa_2>0$, connecting the periodic orbits $\Lambda^{\pm}_{\kappa_1,\kappa_2}$. Each one intersects $\Sigma_h$ at a circle centered at $(0,0)$ with radius $\sqrt{2\kappa_2\sqrt{\e}/\Omega}$, which generates a disk of radius $\sqrt{2h\sqrt{\e}/\Omega}$ when we vary $0<\kappa_2\leq h$ (see \eqref{Zk+} and \eqref{pk}).


We  show that, for the perturbed case, $W^u_\e(\Lambda^-_{\kappa_1,\kappa_2})$ and $W^s_\e(\Lambda^+_{\kappa_1,\kappa_2})$ also intersect $\Sigma_h$ in closed curves near circles of radius $\sqrt{2\kappa_2\sqrt{\e}/\Omega}$ centered in $P_h^u$ and $P_h^s$. Thus, varying $0\leq \kappa_2\leq h$, we can see that $W_{\e}^{u,s}(\Lambda^\pm_{\kappa_1,\kappa_2})$ intersect $\s_h$ in topological disks $\mathcal{D}_h^u$ and $\mathcal{D}_h^s$ near the disks of radius $\sqrt{2h\sqrt{\e}/\Omega}$ centered in $P_h^u$ and $P_h^s$, respectively (see Figure \ref{splitting_fig}).

The existence of heteroclinic connections continuation of the unperturbed ones corresponds to intersections between the disks  $\mathcal{D}_h^u$ and $\mathcal{D}_h^s$.
Even if in the energy level $h=0$, there is no (first round) heteroclinic connections between the points at $X=\pm\infty$ ($p_0^-$ and $p_0^+$), the heteroclinic connections between the periodic orbits $\Lambda^{\pm}_{\kappa_1,\kappa_2}$ may certainly exist when $h>0$, since the two disks may intersect for some values of $h$. The lowest energy level $h_s>0$ for which these heteroclinic connections exist is reached when the boundaries of these disks are tangent  (see Figure \ref{splitting_fig}). Equivalently, when $W_{\e}^{u}(\Lambda_h^-)$ intersects $W_{\e}^{s}(\Lambda_h^+)$ in the energy level $h_s=h_s(\e)$.

\begin{mtheorem}[Existence of oscillating kinks]\label{perper_thm} 
	Fix $h_0>0$. There exists $\e_0>0$ sufficiently small such that, for every $0<\e<\e_0$ and $0\leq h\leq h_0$, the invariant manifolds $W^{u}_{\e}(\Lambda_h^{-})$, $W^{s}_{\e}(\Lambda_h^{+})$ intersect $\s_h$  (given in \eqref{sh}). The first intersection is given by closed curves, which we denote  by $\partial\mathcal D_h^{u,s}$.
Then,
%
there exists 
	$$h_s(\e)=\dfrac{\e\pi^2e^{-2\Omega\sqrt{2/\e}}}{2}(1+\er(\e)), \ \textrm{with }\Omega=\sqrt{1-\dfrac{\e^2}{4}},$$ such that the following statements hold for system \eqref{rescaled_aut_system}.
	\begin{enumerate}
		\item If $0\leq h< h_s(\e)$, the closed curves $\partial\mathcal D_h^{u,s}$ do not intersect each other.
		\item If $ h_s(\e)\leq h\leq h_0$, the closed curves $\partial\mathcal D_h^{u,s}$ intersect at least once.
	\end{enumerate}
Furthermore, given $\mu>1$, there exists $\e_{\mu}>0$ and
$$h_{\mu}(\e)=\dfrac{\e\pi^2e^{-2\Omega\sqrt{2/\e}}}{2}(\mu+\er(\e))^2\geq h_s(\e),$$
such that, for $0<\e<\e_\mu$  and $h_{\mu}(\e)\leq h\leq h_0$,  the closed curves $\partial\mathcal D_h^{u,s}$ have at least two intersections.
\end{mtheorem}

Thus, we can see that there is a family of heteroclinic connections between elements of $X=\pm\infty$ which are contained in the energy level $h$, for $h>h_s$. 


Actually, we prove that, in the energy level $H=h_s$, $\partial\mathcal{D}^u_{h_s}$ and $\partial\mathcal{D}^s_{h_s}$ intersect (tangentially) at least once, and for this reason, $\partial\mathcal{D}^u_{h_s}\cap\partial\mathcal{D}^s_{h_s}$ may have more than one point. Also, our methods show that, for $h>h_s$, $\partial\mathcal{D}^u_{h}\cap\partial\mathcal{D}^s_{h}$ has at least two points and  $\mathcal{D}^u_{h}\cap\mathcal{D}^s_{h}$ has at least one connected component with positive Lebesgue measure
(see Figure \ref{planos2}).



\begin{figure}[!]	\label{planos1}
	\centering
	\begin{overpic}[width=10cm]{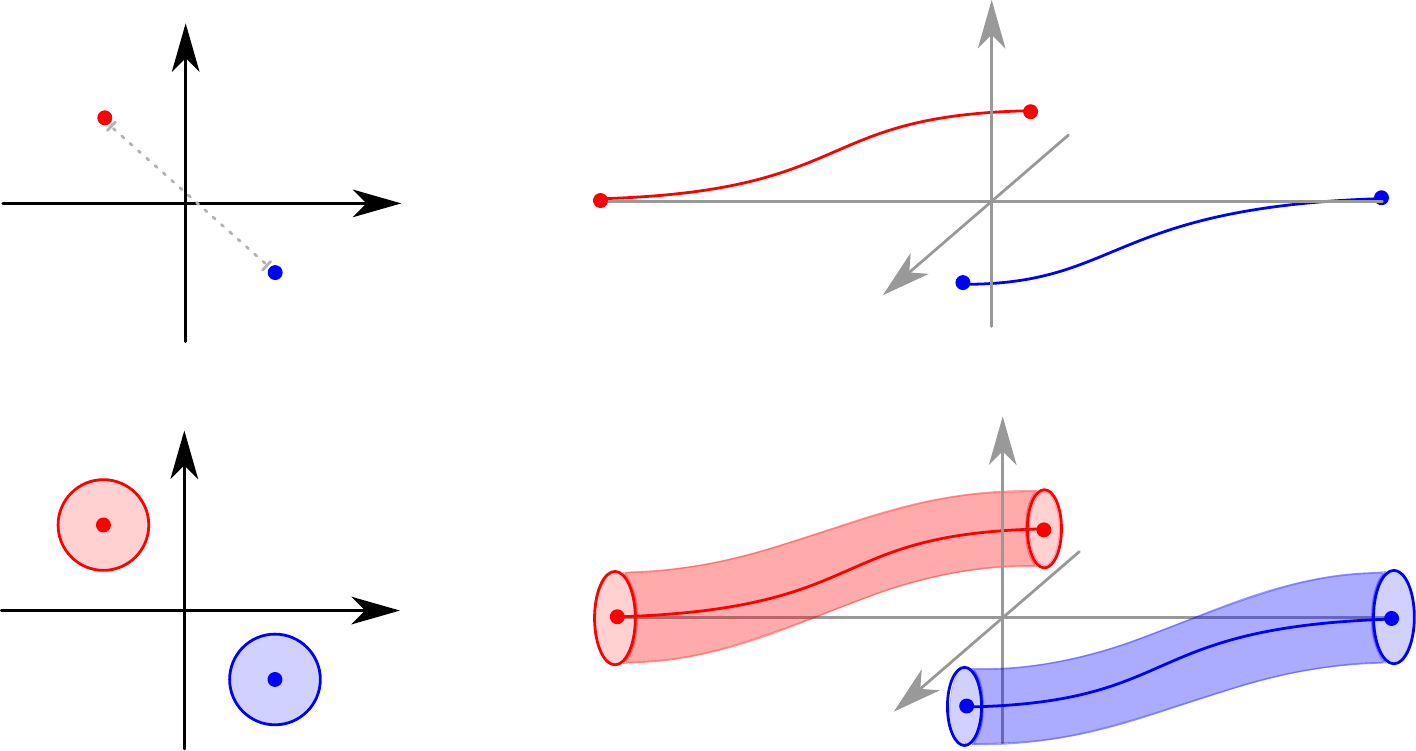}
		\put(3,45){{\footnotesize $P_0^u$}}		
		\put(20.5,32.5){{\footnotesize $P_0^s$}}	
		\put(13.5,39.5){{\footnotesize $d_0(\e)$}}	
		\put(29,38){{\footnotesize $b$}}		
		\put(15,49){{\footnotesize $B$}}					
		\put(-6,39.5){{\footnotesize $(a)$}}
		\put(41,35.5){{\footnotesize $p_0^-$}}				
		\put(95,35.5){{\footnotesize $p_0^+$}}			
		\put(60,30){{\footnotesize $b$}}	
		\put(66,50){{\footnotesize $B$}}
		\put(74,45){{\footnotesize $P_0^u$}}				
		\put(65,29.5){{\footnotesize $P_0^s$}}	
		\put(98,38){{\footnotesize $X$}}		

		\put(3,20.5){{\footnotesize $\mathcal{D}_h^u$}}		
\put(21.5,0){{\footnotesize $\mathcal{D}_h^s$}}	
\put(29,9){{\footnotesize $b$}}		
\put(15,20){{\footnotesize $B$}}					
\put(-6,10){{\footnotesize $(b)$}}
\put(60,0){{\footnotesize $b$}}	
\put(67,20){{\footnotesize $B$}}
		\put(76,18){{\footnotesize $\mathcal{D}_h^u$}}		
\put(64,-2){{\footnotesize $\mathcal{D}_h^s$}}	
\put(101,8){{\footnotesize $X$}}																				
	\end{overpic}
	\bigskip
	\caption{ Splitting of the invariant manifolds contained in the energy level $h$ (in the section $\Sigma_h$) on the left and their projections in the $bXB$-space on the right, for $(a)$ $h=0$ and $(b)$ $h>0$ small. }	
	\label{splitting_fig}
\end{figure}
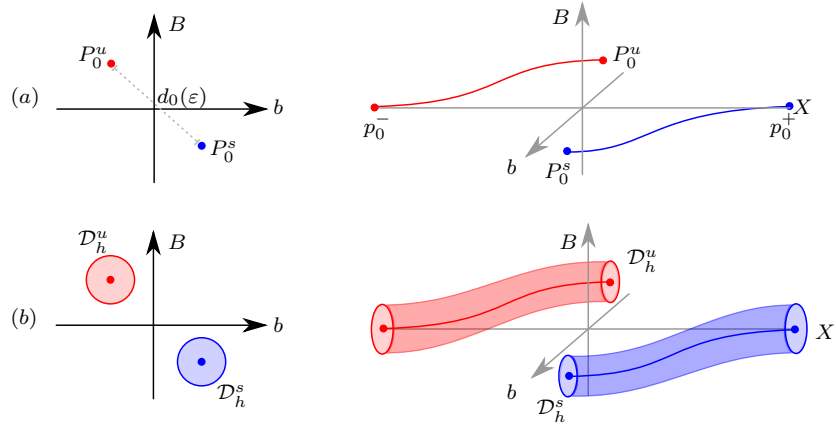 


 \subsubsection{The Critical Energy Level $h_c$}

From our approach and the definitions of Section \ref{model}, the critical energy level occurs for the smallest $h$ such that $W^u_{\e}(p_{h}^-)\subset W^s_{\e}(\Lambda^+_{\kappa_1,\kappa_2})$, for some $\kappa_1,\kappa_2$ satisfying $\kappa_1+\kappa_2=h$. Thus, $h_c$ occurs when $W^u_{\e}(p_{h_c}^-)\subset W^s_{\e}(\Lambda_{h_c}^+)$.

Geometrically speaking, $h_c$ is characterized as the energy level such that $P_{h_c}^u$ belongs to the boundary of the (topological) disk $\mathcal{D}^s_{h_c}$ ``centered'' in $P_{h_c}^s$ (see Figure \ref{planos2}). In the next theorem, we compute $h_c=h_c(\e)$. 

\begin{mtheorem}[Existence of quasi-kinks]\label{perpun_thm}
	
	Consider system \eqref{rescaled_aut_system}. There exist $\e_0>0$, $h_0>0$ and a function 
$$h_c(\e)=2\pi^2\e e^{-2\Omega\sqrt{2/\e}}(1+\er(\e)), \ \text{ with }0<\e<\e_0 \ \text{ and } \ \Omega=\sqrt{1-\dfrac{\e^2}{4}},$$
such that, for every $0<\e<\e_0$ and $0<h<h_0$, the invariant manifolds $W^{u}_{\e}(p_h^{-})$, $W^{s}_{\e}(\Lambda_h^{+})$ intersect $\s_h$  (given in \eqref{sh}). The first intersection of $W^{u}_{\e}(p_h^{-})$, $W^{s}_{\e}(\Lambda_h^{+})$ with $\s_h$ is given by a point and a closed curve, denoted by $P_h^u$ and $\partial\mathcal D_h^{s}$, respectively. Then, $P_h^u\in \partial\mathcal D_h^{s}$ if, and only if  $h=h_c(\e)$.

	
\end{mtheorem}

Theorem \ref{perpun_thm} also holds if we change $p_h^-$ and $\Lambda_h^+$ by $p_h^+$ and $\Lambda_h^-$, respectively.

Now, given $h\geq h_c$, we compute the radius $\kappa_2=\kappa_2(h)$ of the periodic orbit $\Lambda_{\kappa_1,\kappa_2}^+$ such that $p_h^-$ connects to $\Lambda_{\kappa_1,\kappa_2}^+$ through a heteroclinic orbit.

\begin{mtheorem}\label{perpunh_thm}
There exist $\e_0>0$, $h_0>0$ sufficiently small such that, for each $0<\e<\e_0$ and $h_c(\e)\leq h< h_c(\e)+2\pi^2\e e^{-2\Omega\sqrt{2/\e}}h_0$, where $h_c(\e)$ is given by Theorem \ref{perpun_thm} and $\Omega=\sqrt{1-\e^2/4}$, there exists a function $$\kappa:\left(h_c(\e), h_c(\e)+2\pi^2\e e^{-2\Omega\sqrt{2/\e}}h_0\right)\rightarrow \R,$$  such that:
\begin{enumerate}
	\item $0<\kappa(h)<h$ and $\displaystyle\lim_{h\rightarrow h_c(\e)^+}\kappa(h)=0$;
	\item For system \eqref{rescaled_aut_system}, $W^u_{\e}(p_h^-)\subset W^s_{\e}(\Lambda_{\kappa(h),h-\kappa(h)}^+)$;
	\item There exists an orbit of \eqref{rescaled_aut_system} with input velocity $v_i=4\sqrt{h}$ and output velocity $v_f=4\sqrt{\kappa(h)}$. Furthermore, define $v_c=4\sqrt{h_c}$, then 
	\begin{equation}
	v_f=\sqrt{2v_c c_\e} \sqrt{v_i-v_c}+\er((v_i-v_c)^{3/2}),
	\end{equation}
	where $c_{\e}=1+\er(\e)$.
\end{enumerate}

\end{mtheorem}

%

The last item of Theorem \ref{perpunh_thm} proves the conjecture $v_f\approx \er\left((v_i-v_c)^{1/2}\right)$ raised in \cite{GH04}.

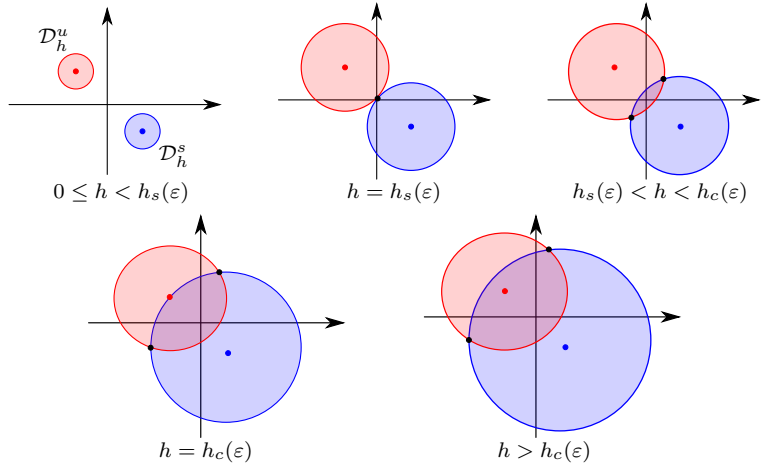
\begin{figure}[H]	
	\centering
	\begin{overpic}[width=10cm]{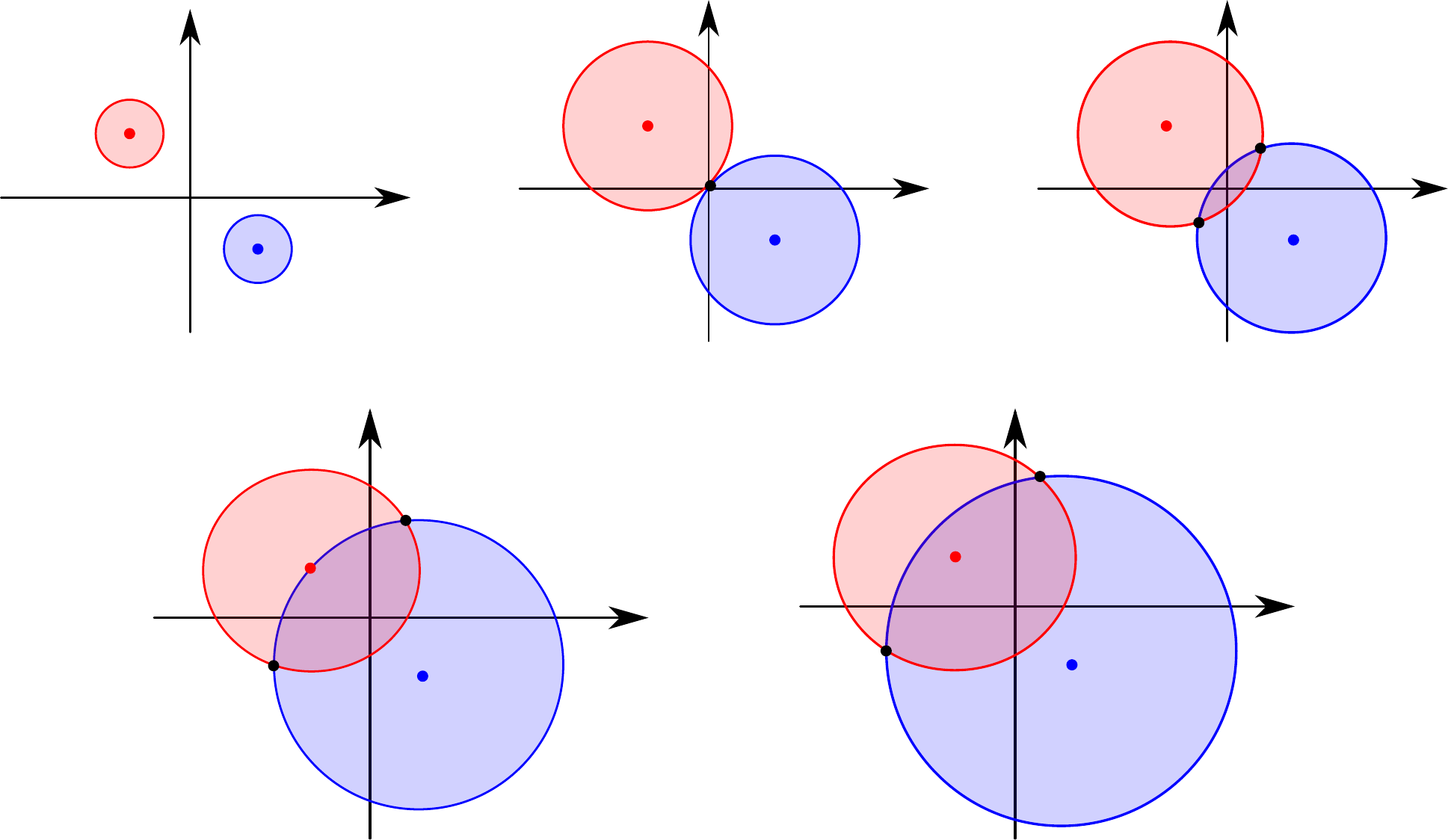}
				\put(6,32){{\footnotesize $0\leq h<h_s(\e)$}}		
				\put(45,32){{\footnotesize $h=h_s(\e)$}}
				\put(75,32){{\footnotesize $h_s(\e)<h<h_c(\e)$}}	
				\put(20,-2.5){{\footnotesize $h=h_c(\e)$}}
\put(65,-2.5){{\footnotesize $h>h_c(\e)$}}																
				\put(4,52.5){{\footnotesize $\mathcal{D}_h^u$}}		
					\put(20,37){{\footnotesize $\mathcal{D}_h^s$}}			
	\end{overpic}
	\bigskip
	\caption{ Relative position of the disks $\mathcal{D}_h^u$ and $\mathcal{D}_h^s$ in the section $\s_h$ in function of the energy level $h$. }	
	\label{planos2}
\end{figure}

\section{Proofs of Theorems A, B, C and D}
Applying the change of coordinates $\Gamma=B+ib$ and $\Theta=B-ib$ to \eqref{rescaled_aut_system} we obtain
\begin{equation}\label{complexcoord_system}
\left\{
\begin{array}{l}
X'=\vspace{0.2cm} \dfrac{Z}{8},\\
Z'=\vspace{0.2cm} - U'(X) -\dfrac{\delta}{\sqrt{2\Omega}} F'(X) \dfrac{(\Gamma-\Theta)}{2i},\\
\Gamma'=\vspace{0.2cm}\omega i \Gamma -\dfrac{\delta}{\sqrt{2\Omega}} F(X),\\
\Theta'=\vspace{0.2cm}-\omega i\Theta -\dfrac{\delta}{\sqrt{2\Omega}} F(X),
\end{array}
\right. \textrm{ with }\left\{\begin{array}{l}\dg=\e^{3/4}, \vspace{0.3cm}\\\omega=\dfrac{\Omega}{\sqrt{\e}},\vspace{0.3cm}\\ \Omega=\sqrt{1-\dfrac{\e^2}{4}}.\end{array}\right. 
\end{equation}
This system is Hamiltonian with respect to
\begin{equation}\label{hamilcc}
\mathcal{H}(X,Z,\Gamma,\Theta)=\dfrac{Z^2}{16}+U(X)+ \dfrac{\dg}{\sqrt{2\Omega}} F(X) \dfrac{\Gamma-\Theta}{2i} + \dfrac{\omega}{2}\Gamma\Theta.
\end{equation}
and the symplectic form $dX\wedge dZ+\dfrac{1}{2i}d\Gamma\wedge d\Theta$.


\subsection{Decoupled System ($F=0$)}\label{dec_sec}

We parameterize the invariant manifolds $W(\kappa_1,\kappa_2)$ (see \eqref{invmanifolds}) of the decoupled system \eqref{complexcoord_system} (with $\dg=0$)  in the coordinates $(X,Z,\Gamma,\Theta)$.

\begin{lemma}\label{sep1h}
	The one-dimensional invariant manifold $W(h,0)=W^{u}_0(p_h^{-})=W^{s}_0(p_h^{+})$ is parameterized in the coordinate system $(X,Z,\Gamma,\Theta)$ by 
	\begin{equation}\label{nh0}
	N_{h,0}(v)= (X_h(v),Z_h(v), 0,0),\ v\in\R
	\end{equation} 
such that:
	\begin{enumerate}
		\item If $h=0$, then 
		
		\begin{equation}\label{par1}
		\left\{\begin{array}{l}
		\vspace{0.2cm}X_0(v)=\arcsinh\left(\dfrac{\sqrt{2}}{2} v\right),\\
		\vspace{0.2cm}Z_0(v)=8(X_0)'(v)=\dfrac{8}{\sqrt{v^2+2}}.
		\end{array}\right.
		\end{equation}
		
		\item	If $h>0$, then
		\begin{equation}\label{par2}
		\left\{\begin{array}{l}
		\vspace{0.2cm}X_{h}(v)=\arcsinh\left(\sqrt{\frac{2+h}{h}} \sinh\left( v\sqrt{h}/2 \right)\right),\\
		\vspace{0.2cm}Z_{h}(v)=8(X_h)'(v)=\dfrac{4\cosh(v\sqrt{h}/2 )}{\sqrt{\frac{1}{2+h}+\frac{\sinh^2(v \sqrt{h}/2 )}{h}}}.
		\end{array}\right.
		\end{equation}
	\end{enumerate}		
\end{lemma}

A simple application of the L'Hospital rule shows us that $X_{h}(v)\rightarrow X_0(v)$, point-wisely, as $h\rightarrow 0$. Nevertheless, the decay of $X_h$ at $\infty$ is significantly different from $X_0$ (for $h=0$, the decay is polynomial and for $h>0$ is exponential). Notice that $N_{0,0}(v)$ has poles at the points $\pm\sqrt{2}i$, whereas the poles of $N_{h,0}(v)$ are all contained in the imaginary axis and the closest to the real line are $\pm\sqrt{2}i+\er(h)$.


\begin{lemma}\label{sep2h}
	The two-dimensional invariant manifold $W(\kappa_1,\kappa_2)=W^{u}_0(\Lambda^{-}_{\kappa_1,\kappa_2})=W^{s}_0(\Lambda^{+}_{\kappa_1,\kappa_2})$, with $\kappa_1\geq0$, $\kappa_2>0$ and $\kappa_1+\kappa_2=h$ is parameterized in the coordinate system $(X,Z,\Gamma,\Theta)$ by 
	\begin{equation}\label{n2d}
	N_{\kappa_1,\kappa_2}(v,\tau)= (X_{\kappa_1}(v),Z_{\kappa_1}(v), \Gamma_{\kappa_2}(\tau),\Theta_{\kappa_2}(\tau)),
	\end{equation} 
	with $v\in\R$ and $\tau\in\mathbb{T}$, such that
	\begin{equation}\label{periodic}
	\Gamma_{\kappa_2}(\tau)=\sqrt{\dfrac{2\kappa_2}{\omega}}e^{i\tau}, \textrm{ and }
	\Theta_{\kappa_2}(\tau)=\sqrt{\dfrac{2\kappa_2}{\omega}}e^{-i\tau},
	\end{equation}
	and $X_{\kappa_1}$, $Z_{\kappa_1}$ are given in \eqref{par1} ($\kappa_1=0$) and \eqref{par2} ($\kappa_1>0$).
	%
	%
\end{lemma}
%
%
%

\begin{remark}
	Notice that, if $\kappa_2=0$, then $N_{\kappa_1,\kappa_2}$ depends on one variable and if $\kappa_2>0$, then it depends on two variables.
\end{remark}


Roughly speaking, in the case $\kappa_1>0$, the parameterization of the invariant manifolds $W(\kappa_1,\kappa_2)$ have the dependence on $v$ expressed in terms of $e^{v\sqrt{\kappa_1}/2}$. Thus, if we consider $v$ in compact domains, these functions can be easily understood by expanding them in a Taylor series in $\kappa_1$. Nevertheless, we must control them for values of $v$ at infinity and $\kappa_1$ near of $0$, which generates an undetermined situation. For this reason, we have a singular dependence of $N_{\kappa_1,\kappa_2}$ at the parameter  $\kappa_1=0$.

Notice that $N_{\kappa_1,\kappa_2}(v,\tau)\rightarrow N_{\kappa_1,0}(v)$ as $\kappa_2\rightarrow 0$ uniformly, and thus the dependence of $N_{\kappa_1,\kappa_2}$ is regular at $\kappa_2=0$.

\begin{remark}\label{rem0}
	The $N_{\kappa_1,\kappa_2}(v,\tau)$, with $v$ or $\tau$ fixed, do not parameterize the solutions of \eqref{complexcoord_system}. Nevertheless, if $\dg=0$, and $\phi_t^0(\cdot)$ is the flow of \eqref{complexcoord_system}, we have $$\phi_t^0(N_{\kappa_1,\kappa_2}(v,\tau))=N_{\kappa_1,\kappa_2}(v+t,\tau+\omega t),$$
	therefore they are invariant by the flow. 
\end{remark}

\subsection{Proof of Theorem \ref{splitting_thmA} (First statement)}\label{breakdown}


The first step to compute the splitting of the separatrix $W(0,0)$ (parameterized by $N_{0,0}(v)$ in \eqref{nh0}) in the energy level $h=0$ is to consider parameterizations 
\begin{equation}\label{def:perturb00}
N_{0,0}^{\star}(v)=(X_0(v),Z_{0}^{\star}(v),\Gamma_{0}^{\star}(v),\Theta_{0}^{\star}(v)), \ \star=u,s
\end{equation}
of the invariant manifolds $W^{u}_\e(p_0^{-})$ and $W^{s}_\e(p_0^{+})$ near $N_{0,0}$, in the complex domains
\begin{equation}\label{domainlocal}
\begin{array}{l}
\vspace{0.2cm}D^{u}_{\e}=\{v\in\C;\  |\Ip(v)|<-\tan\beta\Rp(v)+\sqrt{2}-\sqrt{\e}\},\\
D^{s}_{\e}=\{v\in\C;\ -v\in D^{u}_{\e}\},
\end{array}
\end{equation}
where $0<\beta<\pi/4$ is a fixed angle independent of $\e$ (see Figure \ref{domains}). The parameterization $N_{0,0}(v)$ in \eqref{par1} has singularities only at $\pm\sqrt{2}i$, thus $N_{0,0}$ is analytic in $D^{u,s}_{\e}$.

\begin{figure}[!]	
	\centering
	\begin{overpic}[width=12cm]{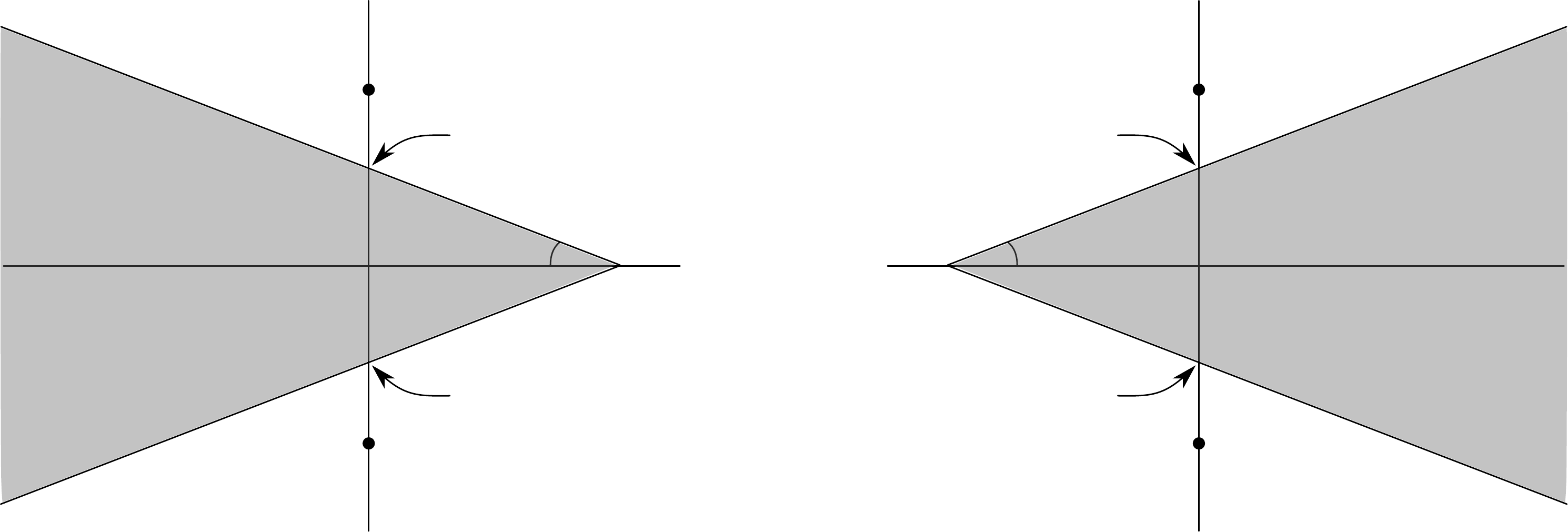}
				\put(25,27.5){{\footnotesize $i\sqrt{2}$}}	
				\put(71,27.5){{\footnotesize $i\sqrt{2}$}}			
				\put(25,5){{\footnotesize $-i\sqrt{2}$}}	
\put(69,5){{\footnotesize $-i\sqrt{2}$}}	
				\put(30,24){{\footnotesize $i(\sqrt{2}-\kappa\sqrt{\e})$}}		
				\put(30,8){{\footnotesize $-i(\sqrt{2}-\kappa\sqrt{\e})$}}			
				\put(57,24){{\footnotesize $i(\sqrt{2}-\kappa\sqrt{\e})$}}		
\put(55,8){{\footnotesize $-i(\sqrt{2}-\kappa\sqrt{\e})$}}	
\put(33,17){{\footnotesize $\bg$}}	
\put(65.5,17){{\footnotesize $\bg$}}	
\put(3,25){{\footnotesize $D^u_{\e}$}}	
\put(95,25){{\footnotesize $D^s_{\e}$}}														
	\end{overpic}
	\bigskip
	\caption{ Complex domains $D^u_{\e}$ and $D^s_{\e}$. }	
	\label{domains}
\end{figure}


We state all the results for the unstable case, since it is analogous for the stable one.   Based on a fixed point argument, we prove the following theorem in Section \ref{par0_sec}. 

\begin{theorem}\label{parameterization1D0}
	
	Given $\nu>0$. There exists $\e_0>0$	such that, for $0<\e\leq \e_0$, the one-dimensional manifold $W^{u}_{\e}(p_0^{-})$ is parameterized by  
	\begin{equation}
	\label{par1D}
	N_{0,0}^{u}(v)=(X_0(v),Z_{0}^{u}(v),\Gamma_{0}^{u}(v),\Theta_{0}^{u}(v)),\end{equation}
	with $v\in D^{u}_{\e}$, where $X_0$ is given in \eqref{par1}, $Z_{0}^{u}(v)$ is obtained from $\mathcal{H}(N_{0,0}^{u}(v))=0$ ($\mathcal{H}$ given in \eqref{hamilcc}) and
\begin{equation}\label{eq1d}
\left\{
\begin{array}{lcl}
\Gamma_{0}^{u}(v)&=&Q^0(v)+\gamma_{0}^{u}(v),\vspace{0.2cm}\\
\Theta_{0}^{u}(v)&=&-Q^0(v)+\theta_{0}^{u}(v),
\end{array}
\right.
\end{equation}
with \begin{equation}\label{first0}Q^0(v)=-i\dfrac{\dg}{\omega\sqrt{2\Omega}}F(X_0(v)).\end{equation}

 Furthermore, $\gamma_{0}^{u}(v),\theta_{0}^{u}(v)$ are analytic functions such that $\theta_{0}^{u}(v)=\overline{\gamma_{0}^{u}(v)},$ for every $v\in\R\cap D^{u}_{\e}$, and there exists a constant $M>0$ independent of $\e$ such that 
	\begin{enumerate}
		\item $\left|\gamma_{0}^{u}(v)\right|,\left|\theta_{0}^{u}(v)\right|\leq M\dfrac{\dg}{\omega^2}\dfrac{1}{|v|^2},$
		for each $v\in D^{u}_{\e}$, $|\Rp(v)|\leq \nu$;\vspace{0.2cm}
		\item $\left|\gamma_{0}^{u}(v)\right|,\left|\theta_{0}^{u}(v)\right|\leq M\dfrac{\dg}{\omega^2}\dfrac{1}{|v^2+2|^2},$
		for each $v\in D^{u}_{\e}$, $|\Rp(v)|\geq \nu$;
	\end{enumerate}
with $\dg=\e^{3/4}$, $\omega=\Omega/\sqrt{\e}$ and $\Omega=\sqrt{1-\e^2/4}$.
\end{theorem}

\begin{remark}
	Notice the points $p_0^{\pm}$ behave as degenerate-saddles at infinity, and thus the existence of local invariant manifolds for the perturbed system is not standard. Nevertheless, these singularities at infinity behave as parabolic points (see Remark \ref{parabolic_rem}) and Theorem \ref{parameterization1D0} gives the existence of their invariant manifolds.
	
	 
	
%
\end{remark}
\begin{figure}[!]	
	\centering
	\begin{overpic}[width=9cm]{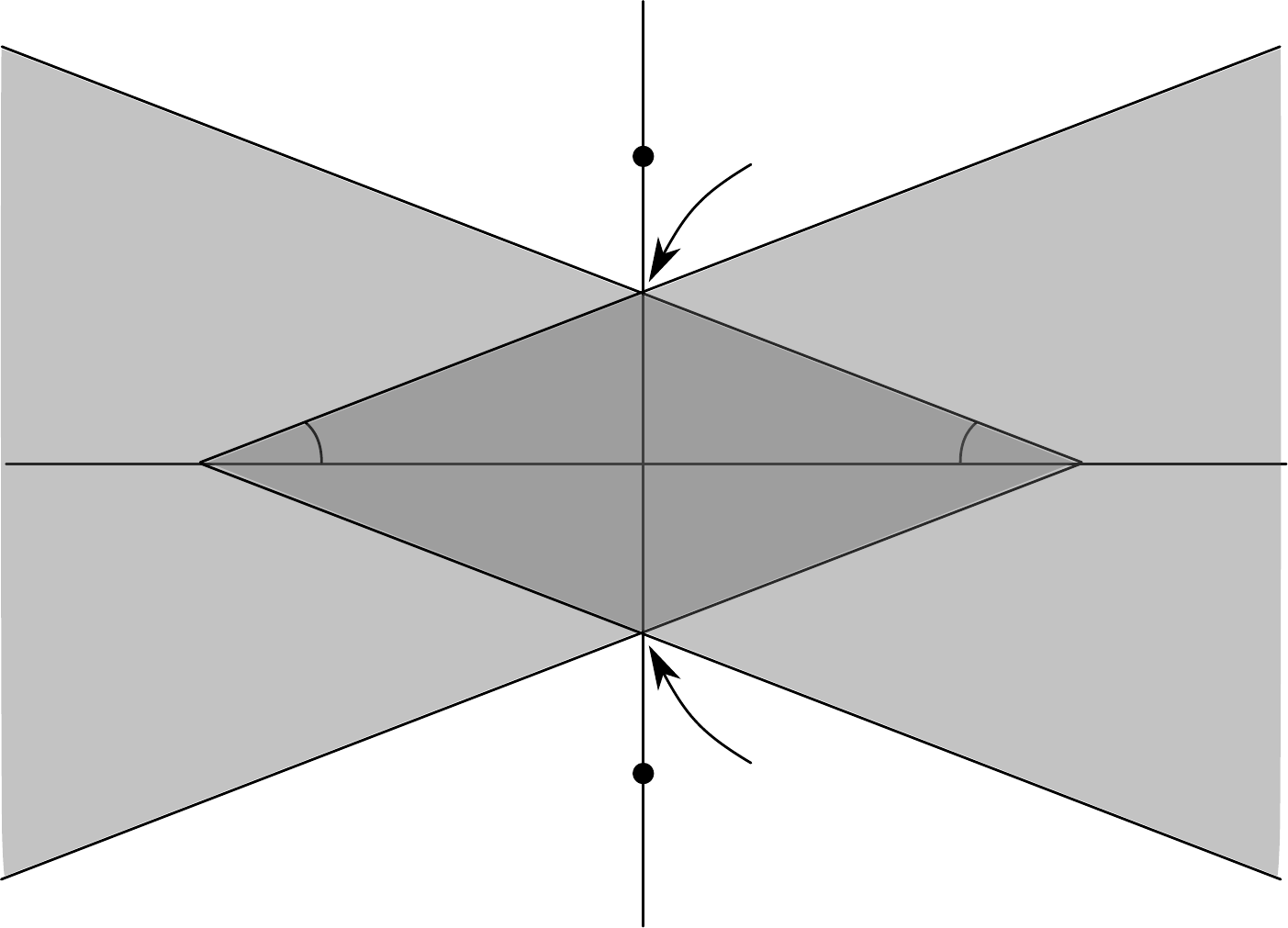}
		\put(42,59){{\footnotesize $i\sqrt{2}$}}
		\put(39.5,11){{\footnotesize $-i\sqrt{2}$}}		
		\put(59,60){{\footnotesize $i(\sqrt{2}+\sqrt{\e})$}}			
		\put(58,11){{\footnotesize $-i(\sqrt{2}+\sqrt{\e})$}}		
		\put(26,37){{\footnotesize $\bg$}}	
		\put(72,37){{\footnotesize $\bg$}}	
		\put(5,55){{\footnotesize $D^u_{\e}$}}	
		\put(92,55){{\footnotesize $D^s_{\e}$}}
		\put(45,38){{\footnotesize $\mathcal{D}_{\e}$}}																																					
	\end{overpic}
	\bigskip
	\caption{ Domain $\mathcal{D}_{\e}$. }\label{int}	
\end{figure} 
By Theorem \ref{parameterization1D0}, both parameterizations $N_{0,0}^{u,s}(v)$ are defined in the complex domain $\mathcal{D}_{\e}=D_{\e}^u\cap D^s_{\e}$, which contains 0 (see Figure \ref{int}). To compute the difference between the invariant manifolds in the section $\s_0$ (see \eqref{sh}), we analyze $\Delta\xi(v)$   given by
$$	\Delta\xi(v)=\left(\begin{array}{l}
\Gamma^{u}_{0}(v)-\Gamma^{s}_{0}(v)\\
\Theta^{u}_{0}(v)-\Theta^{s}_{0}(v)
\end{array}\right),$$
for  $v\in\mathcal{I}_{\e}=\mathcal{D}_{\e}\cap \R$. We prove that $\Delta\xi$ satisfies 
$$
\Delta\xi'= \left(\begin{array}{cc}
\omega i & 0\\
0 & -\omega i
\end{array}\right)\Delta\xi+ B(v)\Delta\xi,
$$
where the entries of the matrix $B$ are small functions of order $\er(\dg^2)$.

Notice that, if $B\equiv 0$, then $\Delta\xi$ is the  analytic function 
$$
\Delta\xi(v)=\left(
\begin{array}{c}
e^{\omega i(v-v_0)}\Delta\xi(v_0)\vspace{0.2cm}\\
e^{-\omega i(v-v_1)}\Delta\xi(v_1)
\end{array}
\right),
$$
for fixed $v_0,v_1\in \mathcal{D}_{\e}$. Thus, choosing $v_0=-i(\sqrt{2}-\sqrt{\e})$ and $v_1=i(\sqrt{2}-\sqrt{\e})$, we have that $|\Delta\xi(v)|\leq Me^{-\sqrt{2}\omega}\leq 2Me^{-\sqrt{\frac{2}{\e}}}$, for  $v\in\mathcal{I}_{\e}$, and therefore it is exponentially small with respect to $\e$.

Roughly speaking, we prove in Section \ref{splitting_sec} that this reasoning will also be true when $B\neq 0$, by using ideas from \cite{BS06}, and we prove the following theorem.

\begin{theorem}\label{splitting_thm}
Consider system \eqref{complexcoord_system}. Given any compact interval $\mathcal{I}\subset \R$ containing $0$, there exists $\e_0>0$ sufficiently small such that, for every $0<\e<\e_0$, the parameterizations $N_{0,0}^{\star}(v)$, $\star=u,s$, given in \eqref{def:perturb00}, are  defined for $v\in\mathcal{I}$ and satisfy
\begin{equation}
\left\{\begin{array}{lcl}
\Gamma_{0}^{u}(0)-\Gamma_{0}^{s}(0)&=&-i\dfrac{2\pi\dg}{\sqrt{\Omega}}e^{-\sqrt{2}\omega}+ \er(\omega\dg^3e^{-\sqrt{2}\omega}),\vspace{0.2cm}\\
\Theta_{0}^{u}(0)-\Theta_{0}^{s}(0)&=&i\dfrac{2\pi\dg}{\sqrt{\Omega}}e^{-\sqrt{2}\omega} + \er(\omega\dg^3e^{-\sqrt{2}\omega}),
\end{array}\right. \Omega=\sqrt{1-\dfrac{\e^2}{4}},\ \omega=\dfrac{\Omega}{\sqrt{\e}} \textrm{ and } \dg=\e^{3/4}.
\end{equation}
\end{theorem}

First statement of Theorem \ref{splitting_thmA} follows as a corollary of Theorem \ref{splitting_thm}.

\subsection{Parameterization of the Invariant Manifolds $W^{u}_{\e}(\Lambda^-_{\kappa_1,\kappa_2})$ and $W^{s}_{\e}(\Lambda^+_{\kappa_1,\kappa_2})$}\label{parametrizaciones}

In this section we find parameterizations of the invariant manifolds $W^{u}_{\e}(\Lambda^-_{\kappa_1,\kappa_2})$ and $W^{s}_{\e}(\Lambda^+_{\kappa_1,\kappa_2})$, for $\kappa_1,\kappa_2\geq0$ and $\kappa_1+\kappa_2=h>0$. Even if one theorem could contain all the results for $\kappa_1\geq 0$ and $\kappa_2\geq 0$, we state three separate theorems, Theorem  \ref{parameterization2Dh} ($\kappa_1=0$),  Theorem
\ref{parameterization1Dh} ($\kappa_2=0$) and Theorem \ref{parameterization2Dk1k2} ($\kappa_1,\kappa_2 >0$), to clarify the exposition (and the corresponding proofs).

\subsubsection{Zero Energy for the Pendulum (Separatrix Case $\kappa_1=0$ and $\kappa_2=h>0$)}\label{regular}
We look for  parameterizations  of the  $2$-dimensional invariant manifolds $W^{u}_{\e}(\Lambda^-_h)$ and $W^{s}_{\e}(\Lambda^+_h)$,
$$N_{0,h}^{\star}(v,\tau)=\left(X_0(v),Z_{0}(v)+Z_{0,h}^{\star}(v,\tau),\Gamma_{h}(\tau)+\Gamma_{0,h}^{\star}(v,\tau),\Theta_{h}(\tau)+\Theta_{0,h}^{\star}(v,\tau)\right),\, \star=u,s$$
as perturbations of $W(0,h)$ (see Lemma \ref{sep2h}).
%


For our purpose, it is not necessary to extend $N_{0,h}^{\star}$ to  a domain which is $\sqrt{\e}$-close to the singularities of $Z_0$. Thus, it is sufficient to consider the domains
\begin{equation}\label{outerdomain}
\begin{array}{l}
D^u=\left\{v\in\C;\ |\Ip(v)|\leq -\tan(\bg)\Rp(v)+\sqrt{2}/2\right\},\\
D^s=\{v\in\C;\ -v\in D^u\},
\end{array}
\end{equation}
for some $0<\bg<\pi/4$ fixed. We also   consider \begin{equation}\label{tsig}
\mathbb{T}_\sigma=\{\tau\in\C;\ |\Ip(\tau)|<\sigma \textrm{ and } \Rp(\tau)\in \mathbb{T}\}.
\end{equation}


We prove the following theorem in Section \ref{parhper_sec}. 

\begin{theorem}\label{parameterization2Dh}
	
	Fix $\sigma>0$ and $h_0>0$. There exists $\e_0>0$ sufficiently small 
	such that, for $0<\e\leq \e_0$ and $0<h\leq h_0$, $W^{u}_{\e}(\Lambda_h^{-})$ is parameterized by  $$N_{0,h}^{u}(v,\tau)=(X_0(v),Z_{0}(v)+Z_{0,h}^{u}(v,\tau),\Gamma_{h}(\tau)+\Gamma_{0,h}^{u}(v,\tau),\Theta_{h}(\tau)+\Theta_{0,h}^{u}(v,\tau)),$$
	with $v\in D^{u}$ (see \eqref{outerdomain}) and $\tau\in\mathbb{T}_{\sigma}$, where $X_0,Z_0,\Gamma_h,\Theta_h$ are given by \eqref{par1} and \eqref{periodic},
\begin{equation}\label{eq2d}
	\left\{\begin{array}{l}
Z_{0,h}^{u}(v,\tau)= Z_{0,h}(v,\tau)+ z_{0,h}^{u}(v,\tau),\vspace{0.2cm}\\
\Gamma_{0,h}^{u}(v,\tau)= Q^0(v)+ \gamma_{0,h}^{u}(v,\tau),\vspace{0.2cm}\\
\Theta_{0,h}^{u}(v,\tau)= -Q^0(v)+ \theta_{0,h}^{u}(v,\tau),
\end{array}\right.
\end{equation}
	where $Q^0$ is given by \eqref{first0}, and
\begin{equation}\label{Zfirst}Z_{0,h}(v,\tau)=\dfrac{\dg}{\omega\sqrt{2\Omega}}F'(X_0(v))\dfrac{\Gamma_h(\tau)+\Theta_h(\tau)}{2}.\end{equation}
	
	
	Furthermore, $z_{0,h}^{u}$ is a real-analytic function and $\gamma_{0,h}^{u},\theta_{0,h}^{u}$ are analytic functions satisfying
	\[
	\theta_{0,h}^{u}(v,\tau)=\overline{\gamma_{0,h}^{u}(v,\tau)}, \ (v,\tau)\in\R^2\cap D^u\times\mathbb{T}_{\sigma},
	\]
	such that there exists a constant $M>0$ independent of $\e$ and $h$ such that, for $(v,\tau)\in D^{u}\times \mathbb{T}_\sigma$,
	\begin{equation}
	\label{bounds38}
	|z_{0,h}^{u}(v,\tau)|, |\gamma_{0,h}^{u}(v,\tau)|,|\theta_{0,h}^{u}(v,\tau)|\leq M\dfrac{\dg}{\omega}\dfrac{1}{|\sqrt{v^2+2}|}
	\end{equation}
 with $\dg=\e^{3/4}$, $\omega=\Omega/\sqrt{\e}$ and $\Omega=\sqrt{1-\e^2/4}$.
\end{theorem}

\begin{remark}
	We stress that the bounds in \eqref{bounds38} are only valid for $v^2+2>1/2$ and therefore do not give any information about the behavior of $N_{0,h}^u(v,\tau)$ near the singularities $v=\pm i \sqrt{2}$. We use this norm to control the functions at $X=\pm\infty$.
\end{remark}

\subsubsection{Positive Energy for the Pendulum}\label{singular}
This section is devoted to study the invariant manifolds of the periodic orbits $\Lambda_{\kappa_1,\kappa_2}^\mp$ for $\kappa_1>0$. First, we consider the case $\kappa_1=h$ and $\kappa_2=0$. In this case $\Lambda_{h,0}^\mp=p_h^\mp$ is a critical point. We  apply  the same ideas of Section \ref{breakdown} to parameterize $W^{u}_{\e}(p_h^{-})$ as
\[
N_{h,0}^{u}(v)=(X_h(v),Z_{h,0}^{u}(v),\Gamma_{h,0}^{u}(v),\Theta_{h,0}^{u}(v)),
\]
where $X_h(v)$ has been introduced in \eqref{par2}. The main difference is that we need to take into account the singular dependence on the parameter $h$ at $h=0$.
%

 
As in Theorem \ref{parameterization2Dh}, for our purposes it is  sufficient to parameterize the manifolds in the  domains $D^{u,s}$ (see \eqref{outerdomain}). We prove the following theorem in Section \ref{parh_sec}.
 
 \begin{theorem}\label{parameterization1Dh}
 	
 	There exist $\e_0>0$ and $h_0>0$ sufficiently small	such that, for $0<\e\leq \e_0$ and $0<h\leq h_0$,  $W^{u}_{\e}(p_h^{-})$ is parameterized by  $$N_{h,0}^{u}(v)=(X_h(v),Z_{h,0}^{u}(v),\Gamma_{h,0}^{u}(v),\Theta_{h,0}^{u}(v)),\ v\in D^{u},$$
 	where $X_h$ is given by \eqref{par2}, $Z_{h,0}^{u}(v)$ is obtained from $\mathcal{H}(N_{h,0}^{u}(v))=h$ ($\mathcal{H}$ given in \eqref{hamilcc}) and
 	\begin{equation}\label{solh0}
  	\left\{
 \begin{array}{lcl}
 \Gamma_{h,0}^{u}(v)&=&Q^h(v)+\gamma_{h,0}^{u}(v),\vspace{0.2cm}\\
 \Theta_{h,0}^{u}(v)&=&-Q^h(v)+\theta_{h,0}^{u}(v),
 \end{array}
 \right.	
 	\end{equation}
with  	\begin{equation}\label{firsth}Q^h(v)=-i\dfrac{\dg}{\omega\sqrt{2\Omega}}F(X_h(v)).\end{equation}
 Furthermore, $\gamma_{h,0}^{u}(v),\theta_{h,0}^{u}(v)$ are analytic functions satisfying $\theta_{h,0}^{u}(v)=\overline{\gamma_{h,0}^{u}(v)}$ for  $v\in\R\cap D^u$ such that there exists a constant $M>0$ independent of $\e$ such that for  $v\in D^u$
 	\begin{equation}
 	\label{cota2}
 	\left|\gamma_{h,0}^{u}(v)\right|,\left|\theta_{h,0}^{u}(v)\right|\leq M\dfrac{\dg}{\omega^2}\dfrac{1}{|v^2+2|},
 	\end{equation}
 	with $\dg=\e^{3/4}$, $\omega=\Omega/\sqrt{\e}$ and $\Omega=\sqrt{1-\e^2/4}$.
 \end{theorem}

Finally we deal with the case $\kappa_1,\kappa_2>0$. Next theorem, proven in Section \ref{parf_sec}, gives the parameterizations of $W^{u}_\de(\Lambda^{-}_{\kappa_1,\kappa_2})$. 



\begin{theorem}\label{parameterization2Dk1k2}
	
	Fix $\sigma>0$. There exist $\e_0>0$ and $h_0>0$ sufficiently small 
	such that, for $0<\e\leq \e_0$, $0<h\leq h_0$, and $\kappa_1>0,\ \kappa_2\geq0$  with $\kappa_1+\kappa_2=h$, the invariant manifold $W^{u}_{\e}(\Lambda^{-}_{\kappa_1,\kappa_2})$ is parameterized by  $$N_{\kappa_1,\kappa_2}^{u}(v,\tau)=(X_{\kappa_1}(v),Z_{\kappa_1}(v)+Z_{\kappa_1,\kappa_2}^{u}(v,\tau),\Gamma_{\kappa_2}(\tau)+\Gamma_{\kappa_1,\kappa_2}^{u}(v,\tau),\Theta_{\kappa_2}(\tau)+\Theta_{\kappa_1,\kappa_2}^{u}(v,\tau)),$$
	for $(v,\tau)\in D^{u}\times\mathbb{T}_{\sigma}$, where $X_{\kappa_1}$, $Z_{\kappa_1},\Gamma_{\kappa_2},\Theta_{\kappa_2}$ are given by \eqref{par2} and \eqref{periodic}, 
	\begin{equation}
	\label{form311}
	\left\{\begin{array}{l}
	\vspace{0.2cm}Z_{\kappa_1,\kappa_2}^u(v,\tau)= Z_{\kappa_1,\kappa_2}(v,\tau)+ z_{\kappa_1,\kappa_2}^{u}(v,\tau),\\
	\vspace{0.2cm} \Gamma_{\kappa_1,\kappa_2}^u(v,\tau)= Q^{\kappa_1}(v)+ \gamma_{\kappa_1,\kappa_2}^{u}(v,\tau),\\
	 \Theta_{\kappa_1,\kappa_2}^u(v,\tau)= -Q^{\kappa_1}(v)+ \theta_{\kappa_1,\kappa_2}^{u}(v,\tau),
	\end{array}\right.
	\end{equation}	
	where $Q^{\kappa_1}$ is given in \eqref{firsth} and
	$$Z_{\kappa_1,\kappa_2}(v,\tau)=\dfrac{\dg}{\omega\sqrt{2\Omega}}F'(X_{\kappa_1}(v))\dfrac{\Gamma_{\kappa_2}(\tau)+\Theta_{\kappa_2}(\tau)}{2}.$$
	 Furthermore, $z_{\kappa_1,\kappa_2}^{u}$ is a real-analytic function and $\gamma_{\kappa_1,\kappa_2}^{u},\theta_{\kappa_1,\kappa_2}^{u}$ are analytic functions satisfying $\theta_{\kappa_1,\kappa_2}^{u}(v,\tau)=\overline{\gamma_{\kappa_1,\kappa_2}^{u}(v,\tau)}$ for  $(v,\tau)\in\R^2\cap D^u\times \mathbb{T}_{\sigma}$ such that there exists a constant $M>0$ independent of $\e$, $\kappa_1$ and $\kappa_2$ such that, for  $(v,\tau)\in D^{u}\times \mathbb{T}_\sigma$ (see \eqref{outerdomain}),
	\begin{equation}
\label{bounds311}|z_{\kappa_1,\kappa_2}^{u}(v,\tau)|, |\gamma_{\kappa_1,\kappa_2}^{u}(v,\tau)|,|\theta_{\kappa_1,\kappa_2}^{u}(v,\tau)|\leq M\dfrac{\dg}{\omega}\dfrac{1}{|v^2+2|^{\frac{1}{2}}}\end{equation}
%
	with $\dg=\e^{3/4}$, $\omega=\Omega/\sqrt{\e}$ and $\Omega=\sqrt{1-\e^2/4}$.
\end{theorem}

\subsection{Approximation of $W^{u}_{\e}(\Lambda^{-}_{\kappa_1,\kappa_2})$ by $W^{u}_{\e}(p_0^{-})$ in the section $\Sigma_h$}\label{aproximacoes}

Recall that for the unperturbed case, we have that
 $$W(\kappa_1,\kappa_2)\cap\s_h=\{(Z,b,B);\ Z=4\sqrt{2+\kappa_1} \textrm{ and }b^2+B^2=2\kappa_2/\omega \}.$$ 
 
\begin{figure}[!]
	\centering
	\begin{overpic}[width=6cm]{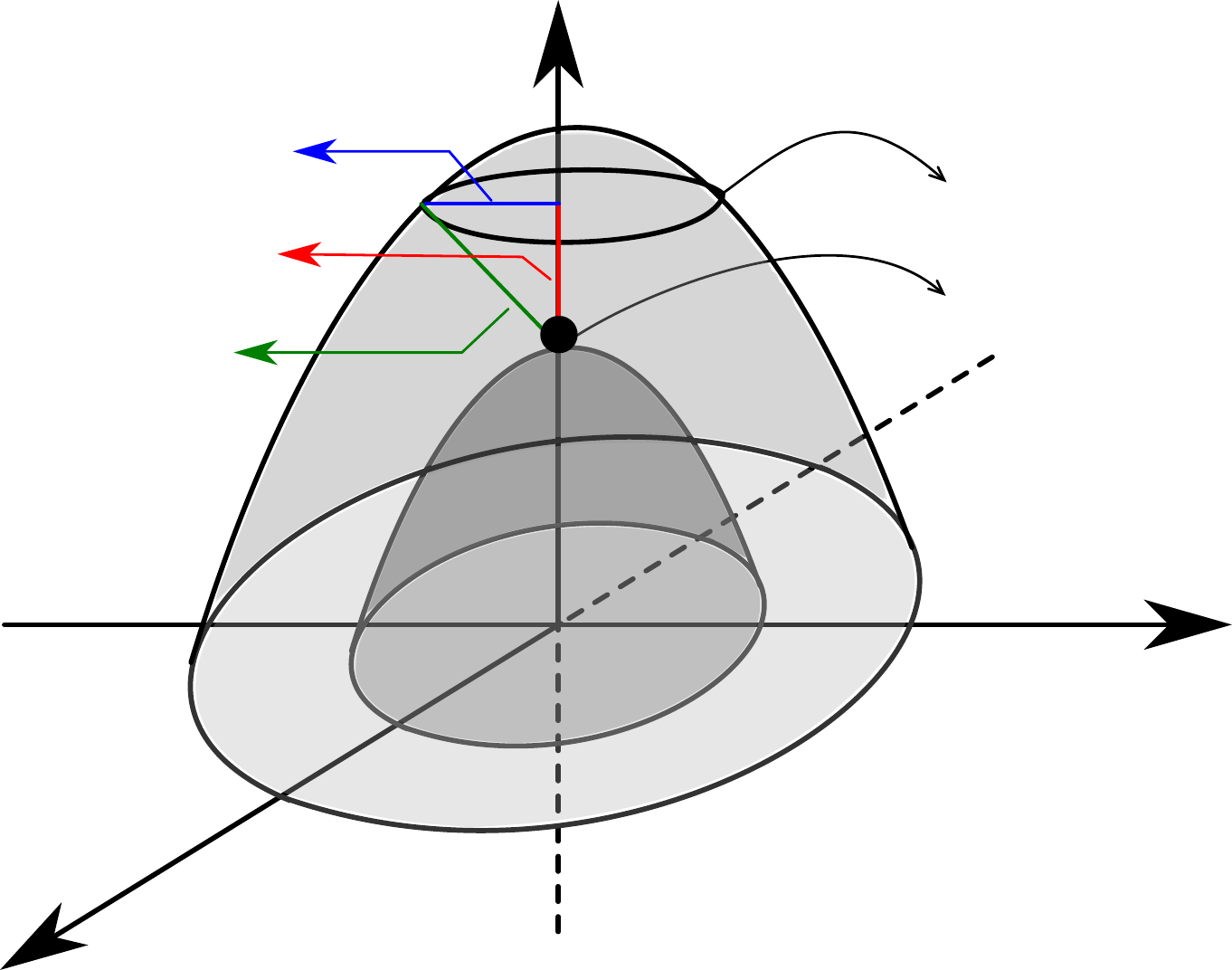}
		\put(49,73){{\footnotesize $Z$}}		
		\put(8,0){{\footnotesize $b$}}
		\put(94,22){{\footnotesize $B$}}
		\put(79,52){{\footnotesize $W(0,0)\cap\mathcal{S}_0$}}
		\put(79,61){{\footnotesize $W(\kappa_1,\kappa_2)\cap\mathcal{S}_h$}}		
		\put(11,66){{\scriptsize $\sqrt{\dfrac{2\kappa_2}{\omega}}$}}	
		\put(-7,57){{\scriptsize $\sqrt{2+\kappa_1}-\sqrt{2}$}}	
		\put(-12,49){{\scriptsize $\er(\sqrt{\kappa_2}+\sqrt{\kappa_1})$}}										
	\end{overpic}
	\bigskip
	\caption{ Comparison between $W(\kappa_1,\kappa_2)$ and $W(0,0)$ in the section $\mathcal{S}_h$ ($X=0$, $H=h$) projected in the $bZB$-space . }	
		\label{ap}
\end{figure}
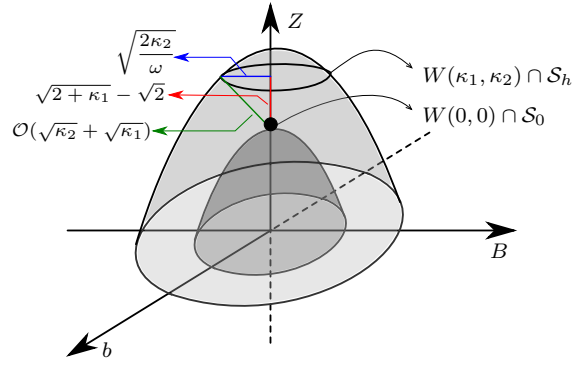 

Thus, in the section $\Sigma_h$, the sets $W(\kappa_1,\kappa_2)$ and $W(0,0)$ are $(\kappa_1+\sqrt{\kappa_2})$-close (see Figure \ref{ap}). Since the perturbed invariant manifolds are close to the unperturbed ones (see Theorems \ref{parameterization2Dh}, \ref{parameterization1Dh}, \ref{parameterization2Dk1k2}), in  the next theorem 
we  approximate $W^{u}_\e(\Lambda^{-}_{\kappa_1,\kappa_2})$ by $W^{u}_\e(p_0^{-})$
for $\kappa_1,\kappa_2$ small.
Using energy conservation and the fact that $\Gamma$ and $\Theta$ are complex conjugate for real values of the variables, it is enough to compare the invariant manifolds  only in the variable $\Gamma$. We define the projection $\pi_\Gamma (X,Z,\Gamma, \Theta)=\Gamma$.
%

\begin{theorem}\label{approx}
	Consider $\kappa_1,\kappa_2\geq 0$, $\kappa_1+\kappa_2=h$, and the parameterization of $N_{\kappa_1,\kappa_2}^{u}$  of $W^{u}_{\e}(\Lambda^{-}_{\kappa_1,\kappa_2})$ obtained in Theorems \ref{parameterization1D0}, \ref{parameterization2Dh}, \ref{parameterization1Dh}, and  \ref{parameterization2Dk1k2}. Then, there exist $\e_0>0$ and $h_0>0$ sufficiently small such that for $0< h\leq h_0$ and $0<\e\leq \e_0$
		\[\pi_\Gamma N_{\kappa_1,\kappa_2}^{u}(0,\tau)-\pi_\Gamma N_{0,0}^{u}(0)=\Gamma_{\kappa_2}(\tau)+\er\left(\dfrac{\dg\sqrt{\kappa_1}}{\omega^2}+\dfrac{\dg\sqrt{\kappa_2}}{\omega^{3/2}}\right),\ \tau\in\mathbb{T},\]
where $\Gamma_{\kappa_2}(\tau)$ has been introduced in \eqref{periodic}, $\dg=\e^{3/4}$, $\omega=\Omega/\sqrt{\e}$ and $\Omega=\sqrt{1-\e^2/4}$.
	
\end{theorem}
The proof of this theorem is done in Sections \ref{approx1}, \ref{approx2} and \ref{approx3}. The result of this theorem for $\kappa_1=h$ and $\kappa_2=0$ implies the second statement of Theorem \ref{splitting_thmA} (note that we are abusing notation since, in this case,  the function $N_{\kappa_1,\kappa_2}^{u}$ does not depend on $\tau$). 

\subsection{Proof of Theorem $B$}\label{periodic_sec}
%
Theorems \ref{parameterization2Dh} and \ref{approx} provide, for $0\leq h\leq h_0$, $\e\leq\e_0$, the existence of the invariant manifolds $W^{u}_{\e}(\Lambda_{h}^{-})$ and $W^{s}_{\e}(\Lambda_{h}^{+})$ which are parameterized by
$$N_{0,h}^{u,s}(v,\tau)=\left(\begin{array}{c}
X_0(v)\vspace{0.2cm}\\
Z_0(v)+Z_{0,h}(v,\tau)+z^{u,s}_{0,h}(v,\tau)\vspace{0.2cm}\\
\Gamma_{h}(\tau)+\Gamma^{u,s}_{0}(v)+F^{u,s}(v,\tau,h,\e)\vspace{0.2cm}\\
\Theta_{h}(\tau)+ \Theta^{u,s}_{0}(v)+\overline{F^{u,s}(v,\tau,h,\e)}
\end{array}  \right), \ (v,\tau)\in (D^{u,s}\cap\R)\times\mathbb{T},$$
where  $X_0$, $Z_0$ are given in \eqref{par1}, $Z_{0,h}$ and $z_{0,h}^{u,s}$ are given by \eqref{eq2d}, $\Gamma_h$ and $\Theta_h$ are given in \eqref{periodic},  $\Gamma_{0}^{u,s},\Theta_0^{u,s}$ are given in \eqref{eq1d} and 
$F^{u,s}$ are analytic functions such that 
$$F^{u,s}(v,\tau,h,\e)=\er\left(\dfrac{\dg\sqrt{h}}{\omega^{3/2}}\right).$$

Consider the section $\Sigma_h$ (which corresponds to $v=0\in D^u\cap D^s$). Then, $W^u_\e(\Lambda_h^-)$ and $W^s_\e(\Lambda_h^+)$ intersect along a heteroclinic orbit if and only if there exist $\tau^u,\ \tau^s$ in $[-\pi,\pi)$ such that $N_{0,h}^u(0,\tau^u)=N_{0,h}^s(0,\tau^s)$. Moreover, using energy conservation, $N_{0,h}^u(0,\tau^u)=N_{0,h}^s(0,\tau^s)$ if, and only if,
$$\left\{\begin{array}{lcl}
\Gamma_{h}(\tau^u)+\Gamma^{u}_{0}(0)+F^{u}(0,\tau^u,h,\e)&=& \Gamma_{h}(\tau^s)+\Gamma^{s}_{0}(0)+F^{s}(0,\tau^s,h,\e)\vspace{0.2cm}\\
\Theta_h(\tau^u)+ \Theta^{u}_{0}(0)+\overline{F^{u}(0,\tau^u,h,\e)}&=&\Theta_h(\tau^s)+ \Theta^{s}_{0}(0)+\overline{F^{u}(0,\tau^s,h,\e)}.
\end{array}\right.
$$
Since $\tau^u, \tau^s\in \mathbb{R}$,  using Theorem \ref{splitting_thm}, the expression of $\Gamma_h$ in \eqref{periodic},  the equations above are equivalent to 

\begin{equation}\label{sistemaperiodica}
\left\{\begin{array}{l}
\sqrt{\dfrac{2h}{\omega}}(\cos(\tau^u)-\cos(\tau^s))+ M_1(\e)+F_1(\tau^u,\tau^s,h,\e)= 0,\vspace{0.2cm}\\
\sqrt{\dfrac{2h}{\omega}}(\sin(\tau^u)-\sin(\tau^s))-\dfrac{2\pi\dg}{\sqrt{\Omega}}e^{-\sqrt{2}\omega}+ M_2(\e)+F_2(\tau^u,\tau^s,h,\e)=0,
\end{array}\right.
\end{equation}
where $0<\e\leq\e_0$, $0<h\leq h_0$ and $M_1,M_2,F_1,F_2$ are real-analytic functions such that 
$$M_1,M_2=\er(\omega\dg^3e^{-\sqrt{2}\omega})\textrm{ and }F_1,F_2=\er\left(\dfrac{\dg\sqrt{h}}{\omega^{3/2}}\right).$$
We change the parameter $h\geq 0$
\begin{equation}\label{rescaling h}
h=\dfrac{\pi^2\omega\dg^2e^{-2\sqrt{2}\omega}}{2\Omega}\mu^2,\textrm{ for }\mu\geq 0.
\end{equation}
Then, since $0<h\leq h_0$, it is sufficient to consider $$0<\mu\leq \mu_0=\dfrac{1}{\dg_0\pi}\sqrt{\frac{2\Omega_0 h_0}{\omega_0}}e^{\sqrt{2}\omega_0},$$
where $\Omega_0=\sqrt{1-\e_0^2/4}$, $\omega_0=\Omega_0/\sqrt{\e_0}$ and $\dg_0=\e_0^{3/4}$. Considering $\e_0>0$ sufficiently small, we can assume that $\mu_0>1$. Replacing $h$ in \eqref{sistemaperiodica} and multiplying the equation by $\dfrac{\sqrt{\Omega}}{\pi\dg }e^{\sqrt{2}\omega}>0$, we may rewrite \eqref{sistemaperiodica} as
\begin{equation}\label{sistemaperiodicax}
\left\{\begin{array}{l}
\mu(\cos(\tau^u)-\cos(\tau^s))+ \widetilde{M}_1(\e)+\widetilde{F}_1(\tau^u,\tau^s,\mu,\e) = 0,\vspace{0.2cm}\\
\mu(\sin(\tau^u)-\sin(\tau^s))-2+ \widetilde{M}_2(\e)+\widetilde{F}_2(\tau^u,\tau^s,\mu,\e) = 0,
\end{array}\right.
\end{equation}
where $\widetilde{M}_1,\widetilde{M}_2,\widetilde{F}_1,\widetilde{F}_2$, are real-analytic functions such that
$$\widetilde{M}_1,\widetilde{M}_2=\er(\omega\dg^2)\textrm{ and }\widetilde{F}_1,\widetilde{F}_2=\er\left(\dfrac{\dg}{\omega}\mu\right).$$

Define the function $G=(G_1,G_2):[-\pi,\pi]^2\times (0,\mu_0]\times[0,\e_0]\rightarrow\rn{2}$ corresponding to 
 the left-hand side of system \eqref{sistemaperiodicax}.  Recalling that $\dg=\e^{3/4}$ and $\omega=\Omega/\sqrt{\e}$, it is clear that
\begin{equation}\label{G} G(\tau^u,\tau^s,\mu,\e)=\left(\begin{array}{c}\mu(\cos(\tau^u)-\cos(\tau^s))+\er(\e)\\
\mu(\sin(\tau^u)-\sin(\tau^s))-2+\er(\e)\end{array}\right).\end{equation}
The equation $G(\tau^u,\tau^s,\mu,0)=(0,0)$ has a unique family of solutions $$\mathcal{S}_0=\left\{(\ag,-\ag,1/\sin(\ag),0);\ \arcsin(1/\mu_0)\leq \ag\leq\pi-\arcsin(1/\mu_0) \right\}.$$ 
We find zeroes of $G$ using the Implicit Function Theorem around every solution of the family $\mathcal{S}_0$.  Denote $\ag_0=\arcsin(1/\mu_0)$ and fix $0<\ag_0\leq\ag\leq\pi-\ag_0$. Then,
\begin{enumerate}
	\item $G(\ag,-\ag,1/\sin(\ag),0)=(0,0)$,\vspace{0.2cm}
	\item $\det\left(\dfrac{\partial (G_1,G_2)}{\partial(\mu,\tau^s)}\right)(\ag,-\ag,1/\sin(\ag),0)= 2\sin(\ag)\neq 0.$
\end{enumerate}
Thus, it follows from the Implicit Function Theorem that there exist $\e_{\ag}>0$ and unique functions $\tau^s_{\ag}:(\ag-\e_{\ag},\ag+\e_{\ag})\times [0,\e_{\ag})\rightarrow [-\pi,\pi]$, $\mu_{\ag}:(\ag-\e_{\ag},\ag+\e_{\ag})\times [0,\e_{\ag})\rightarrow (0,\mu_0]$ such that
$$G(\tau^u,\tau^s_{\ag}(\tau^u,\e),\mu_{\ag}(\tau^u,\e),\e)=(0,0).$$
Furthermore
\begin{equation}
\label{sist}
\left\{\begin{array}{lcl}
\tau^{s}_{\ag}(\tau^u,\e)=-\ag+\er(\tau^u-\ag,\e),\\
\mu_{\ag}(\tau^u,\e)=1/\sin(\ag)+\er(\tau^u-\ag,\e)
\end{array}\right., \quad \tau^u\in(\ag-\e_{\ag},\ag+\e_{\ag}).
\end{equation}
Consider the compact set
$K=[\ag_0,\pi-\ag_0]$.
%
We can find $n\in\N$, $\ag_1,\cdots,\ag_n$ with respectives $\e_{\ag_1},\cdots,\e_{\ag_n}$, previously found, such that the intervals $(\ag_i-\e_{\ag_i},\ag_i+\e_{\ag_i})$, $i=1,\cdots,n$ form a finite cover of $K$. Using the uniqueness of solutions obtained from the Implicit Function Theorem, it is possible to conclude that there exist $\e_1>0$ sufficiently small and functions
$$\left\{\begin{array}{lcl}
\tau_*^{s}\left(\tau^u,\e\right)=-\tau^u+\er(\e),\\
\mu_*\left(\tau^u,\e\right)=1/\sin(\tau^u)+\er(\e),
\end{array}\right.$$
defined for every $\e<\e_1$ and $\tau^u\in K$, such that 
$$G\left(\tau^u,\tau_*^{s}\left(\tau^u,\e\right),\mu_*\left(\tau^u,\e\right),\e\right)=(0,0).$$
This implies that there exists at least one heteroclinic connection in the energy level 
$$h=\dfrac{\pi^2\omega\dg^2e^{-2\sqrt{2}\omega}}{2\Omega}(\mu_*(\tau^u,\e))^2,\ \tau^u\in K.$$
Moreover, $(\mu_{*}(\tau^u,0))^2\geq (\mu_{*}(\pi/2,0))^2=1$, for every $\tau^u\in K$. Thus  $(\mu_*(\tau^u,\e))^2\geq 1+ \er(\e)$ for $\tau^u\in K$ and $\e<\e_1$. Therefore, since $\mu_*(\pi/2,\e)=1+\er(\e)$, there must exist a curve $\tau^u_\mathrm{min}(\e)$, such that
$$(\mu_*(\tau^u,\e))^2\geq (\mu_*(\tau^u_\mathrm{min}(\e),\e))^2,$$
for $\tau^u\in K$, $\e<\e_1$, and $\mu_*(\tau^u_\mathrm{min}(\e),\e)=1+\er(\e)$.

Thus, defining 
$$h_s(\e)=\dfrac{\pi^2\omega\dg^2e^{-2\sqrt{2}\omega}}{2\Omega}(\mu_*(\tau^u_\mathrm{min}(\e),\e))^2=\dfrac{\pi^2\omega\dg^2e^{-2\sqrt{2}\omega}}{2\Omega}(1+\er(\e)),$$
 system \eqref{complexcoord_system} has one heteroclinic orbit between the periodic orbits $\Lambda_h^-$ and $\Lambda_h^+$ in the energy level $0<h\leq h_0$ if, and only if $h\geq h_s(\e)$.

It only remains to prove the last statement of Theorem \ref{perper_thm}.Given $\mu_1>1$, let $\tau_1^u=\arcsin(\mu_1^{-1})\in [\ag_0,\pi/2)\subset K$, and consider the function $g(\tau^u,\e)= \mu_*(\tau^u,\e)-\mu_*(\tau_1^u,\e)$. Applying the Implicit Function Theorem to $g=0$ at the point $(\pi-\tau_1^u,0)$, there exist $\e_{\mu_1}>0$ and a unique curve $\tau_2^u=\tau_2^u(\tau_1^u, \e)$, defined for $0\leq \e<\e_{\mu_1}$, such that $\mu_*(\tau_2^u,\e)=\mu_*(\tau_1^u,\e)$
and $\tau_2^u(\tau_1^u, \e)=\pi-\tau_1^u+\er(\e)$. Moreover, taking  $\e_{\mu_1}$ small enough $\tau_1^u\neq \tau_2^u$ for $\e<\e_{\tau_1^u}$. Thus, in the energy level 
$$h_{\mu_1}= \dfrac{\pi^2\omega\dg^2e^{-2\sqrt{2}\omega}}{2\Omega}(\mu_*(\tau_1^u,\e))^2,$$
where $\mu_*(\tau_1^u,\e)=\mu_1+\er(\e)$, there exist two heteroclinic connections corresponding to $\tau_1^u$ and $\tau_2^u$.

This completes the proof of Theorem \ref{perper_thm}.
\begin{remark}
	Notice that $g(\pi/2,0)=\partial_{\tau^u}g(\pi/2,0)=0$ and $\partial^2_{\tau^u}g(\pi/2,0)\neq0$. Unfortunately, the characterization of the bifurcation of zeros for $\e>0$ becomes impossible, since there is no information on $\partial_\e g(\pi/2,0)$, and its computation requires complicated second order expansions which are beyond the objectives of this work. Nevertheless, under some non-degenericity condition, for example $\partial_\e g(\pi/2,0)\neq 0$, it is posible to detect a saddle-node bifurcation.
\end{remark}

\subsection{Proof of Theorem \ref{perpun_thm}}\label{critic_sec}

Following the same lines of Section \ref{periodic_sec}, we use Theorems \ref{parameterization2Dh} (for the invariant manifold $W^s_\e(\Lambda_h^+)$), \ref{parameterization1Dh} (for the invariant manifold $W^u_\e(p_h^-)$) and \ref{approx} (to compare them to $W^s_\e(p_0^+)$ and  $W^u_\e(p_0^-)$). Then, we can see that $W^u_{\dg}(p_h^{-})\subset W^s_{\dg}(\Lambda_h^{+})$, if and only if 
\begin{equation}\label{sistemaponto}
\left\{\begin{array}{l}
-\sqrt{\dfrac{2h}{\omega}}\cos(\tau^s)+ M_1(\e)+F_1(\tau^s,h,\e)= 0,\vspace{0.2cm}\\
-\sqrt{\dfrac{2h}{\omega}}\sin(\tau^s)-\dfrac{2\pi\dg}{\sqrt{\Omega}}e^{-\sqrt{2}\omega}+ M_2(\e)+F_2(\tau^s,h,\e)=0,
\end{array}\right.
\end{equation}
has solutions $\tau^u,\tau^s\in[-\pi,\pi]$, $0<\e\leq\e_0$, $0<h\leq h_0$ where $h_0$ is given in Theorem \ref{parameterization1Dh}.. The functions $M_j,F_j$ are real-analytic and satisfy
$$M_j=\er(\omega\dg^3e^{-\sqrt{2}\omega})\textrm{ and }F_j=\er\left(\dfrac{\dg\sqrt{h}}{\omega^{3/2}}+\dfrac{\dg\sqrt{h}}{\omega^{2}}\right),\ j=1,2 $$

%

In order to look for solutions of \eqref{sistemaponto}, we consider the change
\begin{equation}\label{reschponto}
h=\dfrac{2\pi^2\omega\dg^2e^{-2\sqrt{2}\omega}}{\Omega}\mu^2, \quad 0<\mu\leq \mu_0=\dfrac{\sqrt{\Omega_0 h_0}}{\dg_0\pi\sqrt{2\omega_0}e^{-\sqrt{2}\omega_0}}
\end{equation}
Considering $\e_0>0$ sufficiently small, we can assume that $\mu_0>1$. Replacing $h$ in \eqref{sistemaponto} and multiplying it by $\dfrac{\sqrt{\Omega}}{2\pi\dg e^{-\sqrt{2}\omega}}>0$, we may rewrite this system as
\begin{equation}\label{sistemapontox}
\left\{\begin{array}{l}
-\mu\cos(\tau^s)+ \widetilde{M}_1(\e)+\widetilde{F}_1(\tau^s,\mu,\e)= 0,\vspace{0.2cm}\\
-\mu\sin(\tau^s)-1+ \widetilde{M}_2(\e)+\widetilde{F}_2(\tau^s,\mu,\e)=0,
\end{array}\right.
\end{equation}
where $\widetilde{M}_j,\widetilde{F}_j$, are real-analytic functions such that
$$\widetilde{M}_j=\er(\omega\dg^2)\textrm{ and }\widetilde{F}_j=\er\left(\dfrac{\dg}{\omega}\mu\right), \ j=1,2.$$

Define the function $G:[-\pi,\pi]\times (0,\mu_0]\times[0,\e_0]\rightarrow\rn{2}$ as 
as the left-hand side of system \eqref{sistemapontox}.  Recalling that $\dg=\e^{3/4}$ and $\omega=\Omega/\sqrt{\e}$, we can see that
\begin{equation}\label{G1}
G(\tau^s,\mu,\e)=\left(\begin{array}{c}
-\mu\cos(\tau^s)+\er(\e)\\
-\mu\sin(\tau^s)-1+\er(\e)\end{array}\right).
\end{equation}
Since,
\begin{enumerate}
	\item $G(-\pi/2,1,0)=(0,0)$,\vspace{0.2cm}
	\item $\det\left(\dfrac{\partial (G_1,G_2)}{\partial(\tau^s,\mu)}\right)(-\pi/2,1,0)=1,$
\end{enumerate}
we can apply the Implicit Function Theorem to obtain $\e_*>0$ and functions $\tau^s_*:[0,\e_*)\rightarrow[-\pi,\pi]$, $\mu_*:[0,\e_*)\rightarrow (0,\mu_0]$ such that $G(\tau^s_*(\e),\mu_*(\e),\e)=0$ for $0\leq\e\leq \e_*$. Furthermore, $\tau^s_*(\e)=-\pi/2+\er(\e)$ and $\mu_*(\e)=1+\er(\e)$.

Defining $$h_c(\e)=\dfrac{2\pi^2\omega\dg^2e^{-2\sqrt{2}\omega}}{\Omega}(\mu_*(\e))^2=\dfrac{2\pi^2\omega\dg^2e^{-2\sqrt{2}\omega}}{\Omega}(1+\er(\e))$$ and reducing $\e_0$ to $\e_*$, Theorem \ref{perpun_thm} follows directly from these facts.

\subsection{Proof of Theorem \ref{perpunh_thm}}\label{finalv}

Following the same lines of Section \ref{periodic_sec},  we use Theorems \ref{parameterization1Dh} (for the invariant manifold $W^u_\e(p_h^-)$), \ref{parameterization2Dk1k2} (for the invariant manifold $W^s_\e(\Lambda_{\kappa_1, \kappa_2}^+)$),  and \ref{approx} (to compare them to $W^s_\e(p_0^+)$ and  $W^u_\e(p_0^-)$). We can see that $W^u_{\dg}(p_h^{-})\subset W^s_{\dg}(\Lambda_{\kappa_1,\kappa_2}^{+})$, if and only if
\begin{equation}\label{sistemapontof}
\left\{\begin{array}{l}
-\sqrt{\dfrac{2\kappa_2}{\omega}}\cos(\tau^s)+ M_1(\e)+F_1(\tau^s,h,\e)= 0,\vspace{0.2cm}\\
-\sqrt{\dfrac{2\kappa_2}{\omega}}\sin(\tau^s)-\dfrac{2\pi\dg}{\sqrt{\Omega}}e^{-\sqrt{2}\omega}+ M_2(\e)+F_2(\tau^s,h,\e)=0,
\end{array}\right.
\end{equation}
has a solution $\tau^s\in[-\pi,\pi]$ for $\e\leq\e_0$, $h\leq h_0$. The functions $M_j,F_j$ are real-analytic and 
$$M_j=\er(\omega\dg^3e^{-\sqrt{2}\omega})\textrm{ and }F_j=\er\left(\dfrac{\dg\sqrt{\kappa_2}}{\omega^{3/2}}+\dfrac{\dg\sqrt{\kappa_1}}{\omega^{2}}+\dfrac{\dg\sqrt{h}}{\omega^{2}}\right), \ j=1,2, \ \kappa_1+\kappa_2=h.$$
We consider the change of parameters and variables
\begin{align}
h&=\dfrac{2\pi^2\omega\dg^2e^{-2\sqrt{2}\omega}}{\Omega}(\mu_{*}(\e)+\mu)^2,\label{reschper}\\
\kappa_2&=\dfrac{2\pi^2\omega\dg^2e^{-2\sqrt{2}\omega}}{\Omega}(\mu_*(\e)+\mu-\xi)^2,\label{resckper}\\
\tau^s&=\tau^s_*(\e)+\tau,\label{resctper}
\end{align}
where $(\mu_*(\e),\tau^s_*(\e))$ is the solution of \eqref{sistemapontox}. Since  $\kappa_2\leq h$, $\mu_*(\e)=1+\er(\e)$ and we are looking for solutions with $\mu,\xi,\tau\approx 0$, we have that $\xi\geq 0$ and $(\mu_*(\e)+\mu-\xi)^2\leq (\mu_{*}(\e)+\mu)^2$.


Replacing $h$, $\kappa_2$ and $\kappa_1$ and multiplying it by $\dfrac{\sqrt{\Omega}}{2\pi\dg e^{-\sqrt{2}\omega}}>0$,  system  \eqref{sistemapontof} as
\begin{equation}\label{sistemapontoxf}
\left\{\begin{array}{l}
-(\mu_*(\e)+\mu-\xi)\cos(\tau^s_*(\e)+\tau)+ \widetilde{M}_1(\e)+\widetilde{F}_1(\tau,\mu,\xi,\e)= 0,\vspace{0.2cm}\\
-(\mu_*(\e)+\mu-\xi)\sin(\tau^s_*(\e)+\tau)-1+ \widetilde{M}_2(\e)+\widetilde{F}_2(\tau,\mu,\xi,\e)=0,
\end{array}\right.
\end{equation}
where $\widetilde{M}_j,\widetilde{F}_j$, are real-analytic functions such that $\widetilde{M}_j=\er(\omega\dg^2)$ and 
$$\widetilde{F}_j=\er\left(\dfrac{\dg}{\omega}\left((\mu_*(\e)+\mu-\xi)+ \dfrac{\sqrt{(\mu_*(\e)+\mu)^2+(\mu_*(\e)+\mu-\xi)^2}}{\omega^{1/2}} +\dfrac{(\mu_*(\e)+\mu)}{\omega^{1/2}} \right)\right), \ j=1,2.$$

Define the function $G:[-\chi_0,\chi_0]\times [0,\chi_0]\times [-\chi_0,\chi_0]\times[0,\chi_0]\rightarrow\rn{2}$  as 
the left hand side  of system \eqref{sistemapontox} and fix $\chi_0>0$ small enough.  Recalling that $\dg=\e^{3/4}$ and $\omega=\Omega/\sqrt{\e}$, we can see that
$$G(\tau,\mu,\xi,\e)=\left(\begin{array}{c}
-(\mu_*(\e)+\mu-\xi)\cos(\tau^s_*(\e)+\tau)+\er(\e)\\
-(\mu_*(\e)+\mu-\xi)\sin(\tau^s_*(\e)+\tau)-1+\er(\e)\end{array}\right).$$
From Section \ref{critic_sec},  $\mu_*(0)=1$ and $\tau^s_*(0)=-\pi/2$. Thus  $G(\tau,\mu,\xi,0)=(0,0)$ has a solution $\tau=0$ and $\mu=\xi$. Since, we are looking for solutions with $\mu,\xi\approx 0$,  we consider the solution $\mu=\xi=0$. Then, since

\begin{enumerate}
	\item $G(0,0,0,0)=(0,0)$,\vspace{0.2cm}
	\item $\det\left(\dfrac{\partial (G_1,G_2)}{\partial(\tau,\xi)}\right)(0,0,0,0)=1$,
\end{enumerate}
we can apply the Implicit Function Theorem to obtain $\e_{_0}>0$ and unique functions $\overline{\tau}:[0,\e_{0})\times[0,\e_0)\rightarrow[-\chi_0,\chi_0]$, $\overline{\xi}:[0,\e_{0})\times [0,\e_0)\rightarrow[-\chi_0,\chi_0]$ such that $G(\overline{\tau}(\mu,\e),\mu,\overline{\xi}(\mu,\e),\e)=0$.
Furthermore $\overline{\tau}(\mu,\e)=\er(\mu,\e)$ and $\overline{\xi}(\mu,\e)=\er(\mu,\e)$.
For $\e=0$, we have that $\xi=\mu$ and $\tau=0$ is a solution of $G(\tau,\mu,\xi,\e)=(0,0)$.
Thus $\overline{\xi}(\mu,0)=\mu$ and, for $\e$ small enough,
$\overline{\xi}(\mu,\e)=\mu+\er(\e)$. 

Finally, if $\xi=0$, then $\kappa_2=h$, $\kappa_1=0$ and therefore \eqref{sistemapontof} becomes \eqref{sistemaponto}. Thus, considering the different scalings done in the systems and the uniqueness of solutions of \eqref{sistemaponto} obtained in Section \ref{critic_sec}, we conclude that $\overline{\xi}(0,\e)=\overline{\tau}(0,\e)\equiv 0$.

These facts, allows us to see that
$$\overline{\xi}(\mu,\e)=c_{\e}\mu+\er(\mu^2),\quad \text{with } c_{\e}=1+\er(\e).$$

Hence, for $\mu\geq 0$ sufficiently small, in the energy level
$$h_\mu=\dfrac{2\pi^2\omega\dg^2e^{-2\sqrt{2}\omega}}{\Omega}(\mu_{*}(\e)+\mu)^2,$$
there exists a unique heteroclinic connection between $p_{h_\mu}^-$ and $\Lambda_{h_\mu}^-(\kappa_1^\mu,\kappa_2^\mu)$, where
$$
\kappa_2^\mu=\dfrac{2\pi^2\omega\dg^2e^{-2\sqrt{2}\omega}}{\Omega}(\mu_*(\e)+\mu-\overline{\xi}(\mu,\e))^2,
$$
and $\kappa_1^\mu=h_\mu-\kappa_2^\mu$. Moreover, if $-\mu_*(\e)<\mu<0$ there is no heteroclinic connections in the energy level $h_\mu$.

Setting $v_i=\sqrt{h_\mu}$, $v_f=\sqrt{\kappa_1^\mu}$ and $v_c=\sqrt{h_c}$, where
$$h_c(\e)=\dfrac{2\pi^2\omega\dg^2e^{-2\sqrt{2}\omega}}{\Omega}(\mu_*(\e))^2,$$
it means that a soliton starting with velocity $v_i<v_c$ is trapped and will surround the defect location, otherwise, if $v_i\geq v_c$, then it will escape the defect location and propagate itself with some output velocity $v_f$. 
In what follows we give an asymptotic formula to the output velocity $v_f$ of orbits with incoming velocity $v_i\approx v_c$.
We omit the dependence of $v_i,v_f$ on $\mu$ in order to simplify the notation.

For $\mu\geq 0$ sufficiently small, we have
$$
\begin{array}{lcl}
v_f^2 &=&\kappa_1^\mu\vspace{0.2cm}\\
&=&h_\mu-\kappa_2^\mu\vspace{0.2cm}\\
&=&\dfrac{2\pi^2\omega\dg^2e^{-2\sqrt{2}\omega}}{\Omega}\left((\mu_{*}(\e)+\mu)^2-(\mu_*(\e)+\mu-\overline{\xi}(\mu,\e))^2\right)\vspace{0.2cm}\\
&=&\dfrac{2\pi^2\omega\dg^2e^{-2\sqrt{2}\omega}}{\Omega}\left(\overline{\xi}(\mu,\e)(2(\mu_{*}(\e)+\mu)-\overline{\xi}(\mu,\e))\right)\vspace{0.2cm}\\
&=&\dfrac{2\pi^2\omega\dg^2e^{-2\sqrt{2}\omega}}{\Omega}(c_\e \mu+\er(\mu^2))(2\mu_{*}(\e)+(2-c_{\e})\mu+\er(\mu^2))\vspace{0.2cm}\\
&=&\dfrac{2\pi^2\omega\dg^2e^{-2\sqrt{2}\omega}}{\Omega}(2\mu_*(\e)c_\e \mu+\er(\mu^2)).
\end{array}
$$ 
Notice that 
$$
\begin{array}{lcl}
v_i-v_c
&=&\sqrt{\dfrac{2\pi^2\omega\dg^2e^{-2\sqrt{2}\omega}}{\Omega}}\mu.
\end{array}
$$ 
Thus
$$
\begin{array}{lcl}
v_f^2 
&=&2v_c c_\e (v_i-v_c)+\er((v_i-v_c)^2).
\end{array}
$$ 
Finally, we obtain that
$$
\begin{array}{lcl}
v_f
&=&\sqrt{2v_c c_\e} \sqrt{v_i-v_c}+\er((v_i-v_c)^{3/2}).
\end{array}
$$ 
Theorem \ref{perpunh_thm} follow directly from  these facts.

\section{Proof of Theorem \ref{parameterization1D0}}\label{par0_sec}

The strategy to prove the existence of $W^{u}_{\e}(p_0^{-})$ and $W^{s}_{\e}(p_0^{+})$ when $\dg\neq 0$ (see \eqref{complexcoord_system}), is to
look for a parameterization $N^{u}_{0,0}(v)$ of $W^{u}_{\e}(p_0^{\pm})$ as a perturbation of $N_{0,0}(v)$.

As in the unperturbed case $W(0,0)$ is parameterized as a graph over $X$ (see \eqref{par1}), we look for $N^{u}_{0,0}$ as 
%
%
%
%
%
\begin{equation}\label{N00}
N_{0,0}^{u}(v)=(X_0(v),Z_{0}^{u}(v),\Gamma_{0}^{u}(v),\Theta_{0}^{u}(v)).
\end{equation}
%

Next lemma, which is straightforward, gives the equation $N_{0,0}^{u}(v)$ has to satisfy to be invariant by the flow of \eqref{complexcoord_system}.
%
\begin{lemma}\label{pertsystem_lem} The invariant manifold $W^u_{\dg}(p_0^-)$, with $\dg\neq 0$, is parameterized by $N_{0,0}^{u}(v)$ if and only if $(\Gamma_0^u(v),\Theta_0^u(v))$ satisfy 
\begin{equation}\label{formaop} 
\left\{\begin{array}{l}
\vspace{0.2cm}\dfrac{ d\Gamma}{d v}(v)-\omega i\Gamma(v)= -\dfrac{\delta}{\sqrt{2\Omega}} F(X_0(v))+\left(\dfrac{Z_0(v)}{\widetilde{\eta}_0(v,\Gamma,\Theta)}-1\right)\left(\omega i \Gamma(v) -\dfrac{\delta}{\sqrt{2\Omega}} F(X_0(v))\right),\\

\dfrac{d \Theta}{d v}(v)+\omega i\Theta(v)=-\dfrac{\delta}{\sqrt{2\Omega}} F(X_0(v))+ \left(\dfrac{Z_0(v)}{\widetilde{\eta}_0(v,\Gamma,\Theta)}-1\right)\left(-\omega i\Theta(v) -\dfrac{\delta}{\sqrt{2\Omega}} F(X_0(v)))\right),\vspace{0.2cm}\\


\displaystyle\lim_{v\rightarrow-\infty}\Gamma(v)=\displaystyle\lim_{v\rightarrow-\infty}\Theta(v)=0.
\end{array}\right.
\end{equation}
where
\begin{equation}
\label{Z}
\widetilde{\eta}_0(v,\Gamma,\Theta)=4\sqrt{-U(X_0(v))- \dfrac{\dg}{\sqrt{2\Omega}} F(X_0(v)) \dfrac{\Gamma(v)-\Theta(v)}{2i} - \dfrac{\omega}{2}\Gamma(v)\Theta(v)},
\end{equation}
with $X_0$ given in \eqref{par1}, $U,F$ given in \eqref{exprUF},
and $Z_0^u(v)=\widetilde{\eta}_0(v,\Gamma_0^u(v),\Theta_0^u(v))$.
\end{lemma}

The term $\frac{\dg}{\sqrt{2\Omega}}F(X_0(v))$ decays  as $1/v$ as $v\to\infty$. To have integrability, we consider the change of variables \eqref{eq1d} to system \eqref{formaop}. Then,  $(\gamma_0^{u},\theta_0^u)$ satisfy
\begin{equation}\label{edo restada}
\left\{\begin{array}{l}
\vspace{0.2cm}\dfrac{d}{d v}\gamma-\omega i\gamma=\omega i \gamma(\eta_0(v,\gamma,\theta)-1)- (Q^0)'(v),\\

\vspace{0.2cm}\dfrac{d}{d v}\theta+\omega i\theta=-\omega i \theta(\eta_0(v,\gamma,\theta)-1)+ (Q^0)'(v),\\
\displaystyle\lim_{v\rightarrow -\infty}\gamma(v)=\displaystyle\lim_{v\rightarrow -\infty}\theta(v)=0,
\end{array} \right.
\end{equation}	
where $Q^0$ is given by \eqref{first0} and 
\begin{equation}
\label{alpha01d}
\eta_0(v,\gamma,\theta)=\left(1+\dfrac{4\dg^2}{\Omega\omega}\left(\dfrac{F(X_0(v))}{Z_0(v)}\right)^2-8\omega\dfrac{\gamma\theta}{(Z_0(v))^2}\right)^{-1/2}.
\end{equation}

To prove Theorem \ref{parameterization1D0}, it is sufficient to find a solution of \eqref{edo restada}. 
\begin{prop}\label{solution0}
	Fix $\nu>0$. There exists $\e_0>0$ such that for $0<\e\leq \e_0$, the equation \eqref{edo restada}  has a solution $(\gamma_{0}^{u}(v),\theta_{0}^{u}(v))$ defined in the domain $D_{\e}^u\subset \C$ (see \eqref{domainlocal}) such that $\theta_{0}^{u}(v)=\overline{\gamma_{0}^{u}(v)}$, for every $v\in D_{\e}^u\cap\R$. Furthermore, both $\gamma_0^u,\theta_0^u$ satisfy bounds $(1)$ and $(2)$ of Theorem \ref{parameterization1D0}.
	
%
%
\end{prop}

We look for a fixed point $(\gamma_0^u,\theta_0^u)$ of the operator \begin{equation}\label{fix0}
\mathcal{G}_{\omega,0}=\mathcal{G}_{\omega}\circ\mathcal{F}_0,
\end{equation} where
	\begin{equation}\label{g0}
	\mathcal{G}_{\omega}(\gamma,\theta)(v)=\left(\begin{array}{c}
	\displaystyle\int_{-\infty}^ve^{\omega i(v-r)}\gamma(r)dr\vspace{0.2cm}\\
	\displaystyle\int_{-\infty}^ve^{-\omega i(v-r)}\theta(r)dr
	\end{array}\right),
	\end{equation}
		\begin{equation}\label{F0}
	\mathcal{F}_{0}(\gamma,\theta)(v)=\left(\begin{array}{c}
	\omega i \gamma(v)(\eta_0(v,\gamma(v),\theta(v))-1)- (Q^0)'(v)\vspace{0.2cm}\\
-\omega i \theta(v)(\eta_0(v,\gamma(v),\theta(v))-1)+ (Q^0)'(v)
	\end{array}\right),
	\end{equation}
	and $Q^0, \eta_0$ are given in \eqref{first0} and \eqref{alpha01d}, respectively.

\subsection{Banach Spaces and Technical Lemmas}\label{banach0}

In this section, we introduce a Banach space which will be used to find a fixed point of $\mathcal{G}_{\omega,0}$. 

Consider the complex domain $D^u_{\e}$ given in \eqref{domainlocal}. For each analytic function $f: D^{u}_{\e}\rightarrow \C$, $\nu>0$, $\ag\geq 0$, we consider:
\begin{equation}\label{norm_local}
\|f\|_{\alpha,\nu}=\displaystyle\sup_{v\in D^{u}_{\e}\cap\{\Rp(v)\leq-\nu\}}|v^2f(v)|+ \displaystyle\sup_{v\in D^{ u}_{\e}\cap\{\Rp(v)>-\nu\}}|(v^2+2)^{\alpha}f(v)|. 
\end{equation}
%
For any $\nu>0$, and $\ag>0$ fixed, the  function space
\begin{equation}
\mathcal{X}_{\alpha,\nu}=\{f: D^{u}_{\e}\rightarrow \C;\ f \textrm{ is an analytic function such that},\ \|f\|_{\alpha,\nu}<\infty\}
\end{equation}
is a Banach space with respect to the norm $\|\cdot\|_{\alpha,\nu}$.


We also consider the product space $$\mathcal{X}_{\ag,\nu}^2=\left\{(f,g)\in\mathcal{X}_{\ag,\nu}\times \mathcal{X}_{\ag,\nu};\ g(v)=\overline{f(v)} \textrm{ for every } v\in D^{u}_{\e}\cap \R\right\}$$ endowed with the norm
$$\|(f,g)\|_{\ag,\nu}=\|f\|_{\ag,\nu}+\|g\|_{\ag,\nu}.$$


\begin{prop}\label{gomega0}
	Given $\nu>0$, $\ag> 0$ fixed, and $(f,g)\in \mathcal{X}_{\ag,\nu}^2$, we have that $\mathcal{G}_{\omega}(f,g)\in \mathcal{X}_{\ag,\nu}^2$. Furthermore, there exists a constant $M>0$ independent of $\e$ such that
	$$\left\|\mathcal{G}_{\omega}(f,g)\right\|_{\ag,\nu}\leq \dfrac{M}{\omega}\left\|(f,g)\right\|_{\ag,\nu},$$
	for every $(f,g)\in \mathcal{X}_{\ag,\nu}^2$.
\end{prop}

The proof of Proposition \ref{gomega0} follows from \cite{GOS}. 

\begin{prop}\label{lipschitz0}
	Let $\eta_0$ be the function given in \eqref{alpha01d}, and $\mathcal{F}_0$ given in \eqref{F0}. Given $\nu>0$ and $K>0$, there exist $\e_0>0$ and $M>0$ such that:
	
	For  $0<\e\leq \e_0$ and $(\gamma_j,\theta_j)\in \mathcal{B}_0(R)\subset \mathcal{X}_{2,\nu}^2$ where  $R=K\dfrac{\dg}{\omega^2}$ and $j=1,2$, the following statements hold for $v\in D^u_{\e}$.
	\begin{enumerate}
		\item $\left|\eta_0(v,\gamma_j(v),\theta_j(v))-1\right|\leq M\dg^2;$\vspace{0.2cm}
		\item $\left|\eta_0(v,\gamma_1(v),\theta_1(v))-\eta_0(v,\gamma_2(v),\theta_2(v))\right|\leq M\dg\omega^{2} \|(\gamma_1,\theta_1)-(\gamma_2,\theta_2)\|_{2,\nu};$\vspace{0.2cm}		
\item $\mathcal{F}_0(\gamma_j,\theta_j)\in\mathcal{X}_{2,\nu}^2$;\vspace{0.2cm}		
		\item $\left\|\mathcal{F}_0(\gamma_1,\theta_1)-\mathcal{F}_0(\gamma_2,\theta_2)\right\|_{2,\nu}\leq M\dg^2\omega\|(\gamma_1,\theta_1)-(\gamma_2,\theta_2)\|_{2,\nu}.$	
				
	\end{enumerate}

\end{prop}
\begin{proof}
	
	Replacing the expressions of $F$, $X_0$ and $Z_0$ given in \eqref{exprUF} and \eqref{par1} in \eqref{alpha01d}, we obtain
	\begin{equation}
	\label{eta0}
\eta_0(v,\gamma,\theta)=\left(1+\dfrac{\dg^2}{4\Omega\omega}\dfrac{v^2}{v^2+2}-(v^2+2)\omega\dfrac{\gamma\theta}{8}\right)^{-1/2}.
	\end{equation}
	
	Taking $\gamma,\theta\in\mathcal{B}_0(R)$, the first statement of the proposition comes from the following inequalities
	\[
	\begin{split}
	\left|\dfrac{\dg^2}{4\Omega\omega}\dfrac{v^2}{v^2+2}-(v^2+2)\omega\dfrac{\gamma\theta}{8}\right|
	\leq&M\dfrac{\dg^2}{\omega}, \qquad \text{if }\Rp(v)\leq -\nu\\
%
%
	\left|\dfrac{\dg^2}{4\Omega\omega}\dfrac{v^2}{v^2+2}-(v^2+2)\omega\dfrac{\gamma\theta}{8}\right|
	\leq&M\dg^2, \qquad \text{if }\Rp(v)\geq-\nu.
	\end{split}
	\]
%
%
%
%
We observe that
\begin{equation}
\label{cota}
	\begin{array}{lcl}
	\left|\eta_0(v,\gamma_1,\theta_1)-\eta_0(v,\gamma_2,\theta_2)\right|
   &\leq& M\omega |(v^2+2)\gamma_1(v)||\theta_1(v)-\theta_2(v)|\vspace{0.2cm}\\ &&+M\omega|(v^2+2)\theta_2(v)||\gamma_1(v)-\gamma_2(v)|\vspace{0.2cm}\\
	\end{array}
\end{equation}
Thus, if $\Rp(v)\leq -\nu$, then 
\begin{equation}
\label{conta2}
\left|(v^2+2)\gamma_1(v)(\theta_1(v)-\theta_2(v))\right|\leq R\left|\dfrac{v^2+2}{v^2}\right|\dfrac{\|\theta_1-\theta_2\|_{2,\nu}}{|v|^2}
\leq M\dfrac{\dg}{\omega^{2}}\|\theta_1-\theta_2\|_{2,\nu},
\end{equation}
whereas, if $\Rp(v)\geq -\nu$, 
\begin{equation}
\label{conta3}
\begin{array}{lcl}
\left|(v^2+2)\gamma_1(v)(\theta_1(v)-\theta_2(v))\right|
&\leq& M\dfrac{\dg}{\sqrt{\e}}\|\theta_1-\theta_2\|_{2,\nu}.
\end{array}
\end{equation}
Recalling that $\omega=\Omega/\sqrt{\e}$ and joining \eqref{conta2} and \eqref{conta3}, we obtain that estimate \eqref{conta3} holds in $ D^{u}_{\e}$.  The other term in \eqref{cota} is bounded in an analogous way. Thus, statement $(2)$ holds.

%
%
%
If $(\gamma_j,\theta_j)\in\mathcal{X}^2_{2,\nu}$, then $\eta_0(v,\gamma_j,\theta_j)\in\R$, for each $v\in D_{\e}^u\cap\R$, thus, it is clear that $\mathcal{F}_0(\gamma_j,\theta_j)\in\mathcal{X}_{2,\nu}^2$.

Finally, for $v\in D^u_{\e}$,
$$
\begin{array}{lcl}
\left|\pi_1\circ\mathcal{F}_0(\gamma_1,\theta_1)(v)-\pi_1\circ\mathcal{F}_0(\gamma_2,\theta_2)(v)\right|&=& \omega\left| \gamma_1(v)(\eta_0(v,\gamma_1,\theta_1)-1)- \gamma_2(v)(\eta_0(v,\gamma_2,\theta_2)-1)\right|\vspace{0.2cm}\\
&\leq& M\dg^2\left(\dfrac{1}{\omega}+1\right)\omega|\gamma_1(v)-\gamma_2(v)|\vspace{0.2cm}\\
&&+ M\dg\omega^{3} \|(\gamma_1,\theta_1)-(\gamma_2,\theta_2)\|_{2,\nu}|\gamma_2(v)|.
\end{array}
$$
Therefore,
$$
\begin{array}{lcl}
\left\|\pi_1\circ\mathcal{F}_0(\gamma_1,\theta_1)-\pi_1\circ\mathcal{F}_0(\gamma_2,\theta_2)\right\|_{2,\nu}
&\leq& M\dg^2\left(\dfrac{1}{\omega}+1\right)\omega\|\gamma_1-\gamma_2\|_{2,\nu}\vspace{0.2cm}\\
&&+ MR\dg\omega^{3} \|(\gamma_1,\theta_1)-(\gamma_2,\theta_2)\|_{2,\nu}\vspace{0.2cm}\\
&\leq& M\dg^2\omega\|(\gamma_1,\theta_1)-(\gamma_2,\theta_2)\|_{2,\nu}.
\end{array}
$$
We can prove the same bound for the second coordinate of $\mathcal{F}_{0}$ analogously. 
\end{proof}

\begin{prop}\label{firstiteration0}
	Consider the operator $\mathcal{G}_{\omega,0}=\mathcal{G}_{\omega}\circ\mathcal{F}_0$, where $\mathcal{G}_{\omega}$ and $\mathcal{F}_0$ are given in \eqref{g0} and \eqref{F0}. Given $\nu>0$,  there exists a constant $M>0$ independent of $\e$, such that
	$$\left\|\mathcal{G}_{\omega,0}(0,0)\right\|_{2,\nu}\leq M\dfrac{\dg}{\omega^2}.$$	
\end{prop}
\begin{proof}
	Recall that $\mathcal{F}_0(0,0)=(- (Q^0)'(v), (Q^0)'(v))$, where $Q^0$ is given by \eqref{first0}.
	Thus $\overline{\pi_1\circ \mathcal{F}_0(0,0)(v)}=\pi_2\circ \mathcal{F}_0(0,0)(v)$, for each $v\in D^u_{\e}\cap\R$ and
	$$\left\|\mathcal{F}_{0}(0,0)\right\|_{2,\nu}= 2 \dfrac{\dg}{\omega\sqrt{2\Omega}}\|F(X_0)'\|_{2,\nu}.$$
	
	A straightforward computation shows that 
	$$F(X_0(v))'=\dfrac{2\sqrt{2}(v^2-2)}{(v^2+2)^2}$$
	Then, 
	\[ 
	\begin{split}
	|v^2F(X_0(v))'|&\leq M\qquad \text{for }\Rp(v)\leq -\nu,\\
	|(v^2+2)^{2}F(X_0(v))'|\leq M|v^2+2|&\leq M\qquad \text{for } \Rp(v)\geq -\nu.
\end{split}
\]
The result follows directly from these bounds and Proposition \ref{gomega0}.
\end{proof}


\subsection{The Fixed Point argument}

Finally, we are able to prove the existence of a fixed point of $\mathcal{G}_{\omega,0}$.

\begin{prop}\label{fixpoint0}
	Given $\nu>0$ fixed. There exists $\e_0>0$ such that for $\e\leq \e_0$, the operator $\mathcal{G}_{\omega,0}$ has a  fixed point $(\gamma_0^{u},\theta_{0}^{u})$ in $\mathcal{X}_{2,\nu}^2$. Furthermore, there exists a constant $M>0$ independent of $\e$ such that
	$$\|(\gamma_{0}^{u},\theta_{0}^{u})\|_{2,\nu}\leq M\dfrac{\dg}{\omega^2}.$$
\end{prop}
\begin{proof}
	From Proposition \ref{firstiteration0}, there exists a  constant $b_1>0$	independent of $h$ and $\e$ such that
	$$\left\|\mathcal{G}_{\omega,0}(0,0)\right\|_{2,\nu}\leq \dfrac{b_1}{2}\dfrac{\dg}{\omega^2}.$$
	Given $(\gamma_1,\theta_1),(\gamma_2,\theta_2)\in\mathcal{B}_{0}(b_1\dg/\omega^2)\subset \mathcal{X}_{2,\nu}^{2}$, we can use Propositions \ref{gomega0}, \ref{lipschitz0} (with $K=b_1$) and the linearity of the operator $\mathcal{G}_{\omega}$ to see that
	$$
	\begin{array}{lcl}
	\left\|\mathcal{G}_{\omega,0}(\gamma_1,\theta_1)-\mathcal{G}_{\omega,0}(\gamma_2,\theta_2)\right\|_{2,\nu}&\leq& \dfrac{M}{\omega}\left\|\mathcal{F}_0(\gamma_1,\theta_1)-\mathcal{F}_0(\gamma_2,\theta_2)\right\|_{2,\nu}\vspace{0.2cm}\\
	&\leq& M\dg^2\|(\gamma_1,\theta_1)-(\gamma_2,\theta_2)\|_{2,\nu}.
	\end{array}
	$$
Thus, choosing $\e_0$ sufficiently small, we have that $\mathrm{Lip}(\mathcal{G}_{\omega,0})\leq 1/2$. Also, it follows that $\overline{\pi_1\circ\mathcal{G}_{\omega,0}(\gamma,\theta)(v)}=\pi_2\circ\mathcal{G}_{\omega,0}(\gamma,\theta)(v)$, for each $v\in D^u_{\e}\cap\R$ and $(\gamma,\theta)\in \mathcal{B}_{0}(b_1\dg/\omega^2)$.
	
	Therefore $\mathcal{G}_{\omega,0}$ sends the ball $\mathcal{B}_{0}(b_1\dg/\omega^2)$ into itself and it is a contraction. Thus, it has a unique fixed point $(\gamma_{0}^{u},\theta_{0}^{u})\in \mathcal{B}_{0}(b_1\dg/\omega^2)$. 
\end{proof}

Proposition \ref{solution0} is a consequence of Proposition \ref{fixpoint0}. 

\section{Proof of Theorem \ref{splitting_thm}}\label{sec_thmA}

\subsection{The Difference Map}\label{dif_sec}

In Proposition \ref{fixpoint0}, we have found complex functions $\Gamma^{\star}_{0}=Q^0+\gamma_0^{\star}$ and $\Theta^{\star}_{0}=-Q^0+\theta^{\star}_{0}$ defined in the complex domains $D^{\star}_{\e}$, respectively, such that,
$$N_{0,0}^{\star}(v)=(X_0(v),Z^{\star}_{0}(v), \Gamma^{\star}_{0}(v),\Theta^{\star}_{0}(v)),$$
are parameterizations of $W^{\star}_{\dg}(p_0^{\mp})$ of \eqref{complexcoord_system}.
Both $(\Gamma^{u}_{0},\Theta^{u}_{0})$ and $(\Gamma^{s}_{0},\Theta^{s}_{0})$ are defined in the complex domain
\begin{equation}\label{domain_dif}
	\mathcal{D}_{\e}=D^u_{\e}\cap D^s_{\e}.
\end{equation}
Note that $0\in \mathcal{I}_{\e}:=\mathcal{D}_{\e}\cap \R$. To prove that the heteroclinic connection between $p_0^-$ and $p_0^+$ of \eqref{complexcoord_system} is broken for $\e>0$ sufficiently small, it is sufficient to show that
\begin{equation}
	\left|N_{0,0}^{u}(v)-N_{0,0}^{s}(v)\right|\geq \left|(\Gamma^{u}_{0},\Theta^{u}_{0})(v)-(\Gamma^{s}_{0},\Theta^{s}_{0})(v)\right|>0,
\end{equation}
for some $v\in\mathcal{I}_{\e}$.
%
To this end, we study the difference map
\begin{equation}\label{deltaxi}
	\Delta\xi(v)=\left(\begin{array}{l}
	\Gamma^{u}_{0}(v)-\Gamma^{s}_{0}(v)\vspace{0.2cm}\\
	\Theta^{u}_{0}(v)-\Theta^{s}_{0}(v)
	\end{array}\right)=\left(\begin{array}{l}
	\gamma^{u}_{0}(v)-\gamma^{s}_{0}(v)\vspace{0.2cm}\\
	\theta^{u}_{0}(v)-\theta^{s}_{0}(v)
	\end{array}\right),
\end{equation}
where $(\gamma^{\star}_{0},\theta^{\star}_{0})$, $\star=u,s$, are given by Proposition \ref{fixpoint0}. 


\begin{prop}\label{propchata}
	The difference map $\Delta\xi$ satisfies the differential equation:
	\begin{equation}\label{dif_eq}
		\Delta\xi'= A\Delta\xi+ B(v)\Delta\xi,
	\end{equation}
where
\begin{equation}
\label{matrix}
A=\left(\begin{array}{cc}
\omega i & 0\\
0 & -\omega i
\end{array}\right) \textrm{ and }B(v)=\left(\begin{array}{cc}
b_{1,1}(v) & b_{1,2}(v)\\
b_{2,1}(v) & b_{2,2}(v)
\end{array}\right),
\end{equation}
and there exists a constant $M$  independent of $\e$, such that for $v\in\mathcal{D}_{\e}$,
\begin{equation}\label{bjkbound}|b_{j,k}(v)|\leq M\omega\dg^2,\quad j,k=1,2.\end{equation}
\end{prop}
\begin{proof}
Recall that both $(\gamma^{u,s}_{0},\theta^{u,s}_{0})$ satisfy \eqref{edo restada}  and therefore
	$$
	\left(\begin{array}{l}
	\gamma'-\omega i\gamma\\
	\theta'+\omega i\theta 
	\end{array}\right)=\mathcal{F}_0(\gamma,\theta),$$
	where $\mathcal{F}_0$ is given in \eqref{F0}. Therefore $\Delta\xi$ satisfies
	\begin{equation}
		\begin{array}{lcl}
			\Delta\xi'&=&A\Delta\xi +G(v),
		\end{array}	
	\end{equation}
where	$G(v)=g(v,\gamma_{0}^{u}(v),\theta_{0}^{u}(v))- g(v,\gamma_{0}^{s}(v),\theta_{0}^{s}(v))$,
with $$g(v,z_1,z_2)=\left(\begin{array}{c}
i\omega z_1(\eta_0(v,z_1,z_2)-1)\\
-i\omega z_2(\eta_0(v,z_1,z_2)-1)
\end{array}\right).$$

Notice that $G(v)$ is a known function, since $(\gamma^{u,s}_0,\theta^{u,s}_{0})$ are given by Proposition \ref{fixpoint0}. We apply the Integral Mean Value Theorem to obtain
%
%

\begin{equation}
\label{gg}
\begin{array}{lcl}
g(v,\gamma_{0}^{u},\theta_{0}^{u})- g(v,\gamma_{0}^{s},\theta_{0}^{s})&=&\left(\begin{array}{cc}
b_{1,1}(v)&b_{1,2}(v)\vspace{0.2cm}\\
b_{2,1}(v)&b_{2,2}(v)\\
\end{array}\right)\cdot\left(\begin{array}{l}
\gamma_0^u-\gamma_0^s\vspace{0.2cm}\\
\theta_0^u-\theta_0^s
\end{array}\right),
\end{array}
\end{equation}
where $b_{j,k}$ are analytic functions, $j,k=1,2$. Estimate \eqref{bjkbound} follows from Propositions \ref{lipschitz0} and   \ref{fixpoint0}. 
\end{proof}


\subsection{Exponentially Small Splitting of $W^{u}_{\e}(p_0^{-})$ and $W^{s}_{\e}(p_0^{+})$}\label{splitting_sec}

We study the solutions of \eqref{dif_eq}. Notice that, if $B=0$, then any analytic solution of \eqref{dif_eq} which is bounded in $\mathcal{D}_{\e}$ is exponentially small with respect to $\e$ for real values  $v\in\mathcal{I}_{\e}$. In this section, we follow ideas from \cite{BS06} to prove that the same holds for solutions of the full equation \eqref{dif_eq} using that $B$ (given in \eqref{matrix}) is small  for $\e$ small enough.

We are interested in obtaining an asymptotic expression for $\Delta\xi$ given in \eqref{deltaxi}. From  Proposition \ref{fixpoint0}, we have that $(\gamma_{0}^{u,s},\theta_0^{u,s})$ is obtained as a fixed point of $\mathcal{G}_{\omega,0}^{u,s}$. Thus, the difference map can be expressed as $$\Delta\xi=\mathcal{G}_{\omega,0}^u(\gamma_{0}^{u},\theta_{0}^{u})-\mathcal{G}_{\omega,0}^s(\gamma_{0}^{s},\theta_{0}^{s}).$$ 

Therefore, as $\gamma_{0}^{u,s},\theta_{0}^{u,s}$ are small, it suggests that the dominant part of $\Delta\xi$ should be given by $\mathcal{M}=\mathcal{G}_{\omega,0}^{u}(0,0)-\mathcal{G}_{\omega,0}^{s}(0,0)$. For this reason, we decompose    
\begin{equation}\label{decomp}
\Delta\xi=\mathcal{M}+\Delta\xi_1,
\end{equation}
where $\mathcal{M}=(\mathcal{M}_\Gamma,\mathcal{M}_{\Theta})$ is given by the Melnikov integrals
\begin{equation}\label{melnikov}
\begin{array}{c}\mathcal{M}_\Gamma(v)=
ie^{i\omega v}\displaystyle\int_{-\infty}^{\infty} e^{-i\omega r}\dfrac{2\dg(r^2-2)}{\omega\sqrt{\Omega}(r^2+2)^2}dr=c_1^0e^{i\omega v},\vspace{0.3cm}\\
\mathcal{M}_\Theta(v)=
-ie^{-i\omega v}\displaystyle\int_{-\infty}^{\infty} e^{i\omega r}\dfrac{2\dg(r^2-2)}{\omega\sqrt{\Omega}(r^2+2)^2}dr=c_2^0e^{-i\omega v}
\end{array}
\end{equation}
and $\Delta\xi_1=(\Delta^1_{\Gamma},\Delta^1_{\Theta})$.

A straightforward computation proves the following lemma.

\begin{lemma}\label{melnikov_lem} The constants $c_1^0$ and  $c_2^0$ are given by
	\begin{equation}\label{melnikovctes}
	c_1^0=-i\dfrac{2\pi\dg}{\sqrt{\Omega}}e^{-\sqrt{2}\omega}, 	\textrm{ and }c_2^0=\overline{c_1^0}.
	\end{equation}

\end{lemma}

Theorem \ref{splitting_thm} is equivalent to the following theorem. The remainder of Section \ref{splitting_sec} is devoted to prove it.

\begin{theorem}\label{exponentiallysmall} There exists $\e_0>0$ sufficiently small such that for $v\in \mathcal{I}_{\e}\subset\R$, $0<\e\leq\e_0$,
	\begin{equation}\label{asymp}
	\Delta\xi(v)= \mathcal{M}(v) + \er(\omega\dg^3e^{-\sqrt{2}\omega}),
	\end{equation}
	where $\mathcal{M}=(\mathcal{M}_{\Gamma},\mathcal{M}_{\Theta})$ is the Melnikov vector defined in \eqref{melnikov}.
\end{theorem}
\subsubsection{A Fixed Point Argument for the error $\Delta\xi_1$}

We write $\Delta\xi_1$ in \eqref{decomp} as solution of a fixed point equation in the  functional space
\begin{equation}
\mathcal{E}=\left\{ f:\mathcal{D}_{\e}\rightarrow\C^2;\ f \textrm{ is analytic and }\|f\|_{\mathcal{E}}<\infty \right\},
\end{equation}
where
\begin{equation}
\|f\|_{\mathcal{E}}= \displaystyle\sum_{j=1}^2\displaystyle\sup_{v\in\mathcal{D}_{\e}}|(v^2+2)^2\pi_j\circ f(v)|.
\end{equation}
We also consider the linear operator $\mathcal{H}_0$ given by
\begin{equation}
\mathcal{H}_0(g)(v)=\left(\begin{array}{c}
\vspace{0.2cm}e^{\omega i v}\displaystyle\int_{v^*}^ve^{-i\omega r}\pi_1(B(r)\cdot g(r))dr\\
e^{-\omega i v}\displaystyle\int_{\overline{v*}}^ve^{i\omega r}\pi_2(B(r)\cdot g(r))dr
\end{array}\right),
\end{equation}
where $v^*=-(\sqrt{2}-\sqrt{\e})i$ and $B$ is the matrix given \eqref{matrix}.

Using \eqref{bjkbound}, the operator $\mathcal{H}_0$ is well-defined from $\mathcal{E}_{\e}$ to itself. To simplify the notation, we introduce the function
\begin{equation}\label{expI}
I(k_1,k_2)(v)= e^{Av}\left(\begin{array}{c}
k_1\\
k_2
\end{array}\right)=\left(\begin{array}{c}
e^{i\omega v}k_1\\
e^{-i\omega v}k_2
\end{array}\right),
\end{equation}
where $k_j\in\C$, $j=1,2$,  $v\in\mathcal{D}_{\e}$ and $A$ is the matrix given by \eqref{matrix}. Notice that $\mathcal{M}(v)=I(c_1^0,c_2^0)(v)$.

\begin{lemma}\label{lemma35}
	The difference map $\Delta\xi$ belongs to $\mathcal{E}_{\e}$ and $\|\Delta\xi\|_{\mathcal{E}}\leq M\e$. Furthermore, there exist $(c_1,c_2)\in\C^2$ such that:
	\begin{equation}\label{diffunc}
	\Delta\xi_1(v)=I(c_1-c_1^0,c_2-c_2^0)(v)+\mathcal{H}_0(\Delta\xi_1)(v) +\mathcal{H}_0(\mathcal{M})(v),
	\end{equation}
	and $|c_j-c_j^0|\leq M \dg^3 e^{-\sqrt{2}\omega}$, $j=1,2$, where $M$ is a constant independent of $\e$.
\end{lemma}
\begin{proof}
	Since $(\gamma_{0}^{u,s},\theta_{0}^{u,s})\in\mathcal{X}_{2,\nu}^2$, it is clear to see that $\Delta\xi\in \mathcal{E}_{\e}$. In addition, from Proposition \ref{fixpoint0},
	$$\|\Delta\xi\|_{\mathcal{E}}\leq 2(\|(\gamma_{0}^{u},\theta_{0}^{u})\|_{2,\nu}+ \|(\gamma_{0}^{s},\theta_{0}^{s})\|_{2,\nu})\leq M\dfrac{\dg}{\omega^2},$$
	where $M$ is a constant independent of $\e$.
	
	Since $\Delta\xi$ is a solution of \eqref{dif_eq},  the method of variation of parameters implies that, given $v_1,v_2\in\mathcal{D}_{\e}$, there exist $c_1,c_2\in\C$ such that
	\begin{equation}\label{condindif}
	\Delta\xi(v)=\left(\begin{array}{c}
	\vspace{0.2cm}e^{\omega i v}c_1+e^{\omega i v}\displaystyle\int_{v_1}^ve^{-i\omega r}\pi_1(B(r)\cdot \Delta\xi(r))dr\\
e^{-\omega i v}c_2+e^{-\omega i v}\displaystyle\int_{v_2}^ve^{i\omega r}\pi_2(B(r)\cdot \Delta\xi(r))dr
	\end{array}\right).
	\end{equation}
	We take $v_1=v^*$, $v_2=\overline{v^*}$, with $v^*=-(\sqrt{2}-\sqrt{\e})i$.  Thus,
	$$\Delta\xi(v)=I(c_1,c_2)(v)+\mathcal{H}_0(\Delta\xi)(v).$$
	Using that $\Delta\xi=\mathcal{M}+\Delta\xi_1$, $\mathcal{M}(v)=I(c_1^0,c_2^0)(v)$ and $\mathcal{H}_0$ is linear,
	$$\Delta\xi_1(v)=I(c_1-c_1^0,c_2-c_2^0)(v)+\mathcal{H}_0(\Delta\xi_1)(v) +\mathcal{H}_0(\mathcal{M})(v).$$
	Now, we bound $|c_j-c_j^0|$, $j=1,2$. By \eqref{decomp} and Proposition \ref{fixpoint0},
	$$
	\begin{array}{lcl}
	\vspace{0.2cm}\|\Delta\xi_1\|_{2,\nu}&=&\|\Delta\xi -\mathcal{M}\|_{2,\nu} \\
	\vspace{0.2cm}& = & \|(\gamma_{0}^{u},\theta_{0}^{u})-(\gamma_{0}^{s},\theta_{0}^{s}) -(\mathcal{G}_{\omega,0}^u(0,0)-\mathcal{G}_{\omega,0}^s(0,0))\|_{2,\nu} \\
	\vspace{0.2cm}& = &\|\mathcal{G}_{\omega,0}^u(\gamma_{0}^{u},\theta_{0}^{u})-\mathcal{G}_{\omega,0}^u(0,0)-(\mathcal{G}_{\omega,0}^s(\gamma_{0}^{s},\theta_{0}^{s}) -\mathcal{G}_{\omega,0}^s(0,0))\|_{2,\nu}\\
	\vspace{0.2cm}& \leq &M\dg^2( \|(\gamma_{0}^{u},\theta_{0}^{u})\|_{2,\nu} +\|(\gamma_{0}^{s},\theta_{0}^{s})\|_{2,\nu})	\\ 	
	\vspace{0.2cm}& \leq &M\dfrac{\dg^3}{\omega^2}.	 	 	 	 
	\end{array}
	$$
	Thus,
	$$|\pi_j(\Delta\xi_1(v))|\leq M\dfrac{\dg^3}{\omega^2|v^2+2|^2}\leq M\dg^3,	\textrm{ for each } v\in\mathcal{D}_{\e},\ j=1,2.$$

	
	%
	
	In particular, replacing $v=v^*$ in the first component of \eqref{diffunc}, we obtain that
	$$|e^{\omega i v^*}(c_1-c_1^0)|\leq M\dg^3 \Leftrightarrow |c_1-c_1^0|\leq M\dg^3 e^{\omega\sqrt{\e}}e^{-\sqrt{2}\omega}\leq 2M\dg^3 e^{-\sqrt{2}\omega}.$$
	
	 Analogously, taking $v=\overline{v^*}$ in the second component of \eqref{diffunc}, we obtain that $|c_2-c_2^0|\leq 2M\dg^3e^{-\sqrt{2}\omega}$.
\end{proof}

%
%
%
%
%
%
%
%
%
%
%

%

\subsubsection{Exponentially Smallness of $\Delta\xi_1$}
Consider the functional space
\begin{equation}\label{Zesp}
\mathcal{Z}=\{f: \mathcal{D}_{\e}\rightarrow\C^2;\ f \textrm{ is analytic and } \|f\|_{\mathcal{Z}}<+\infty \},
\end{equation}
where
\begin{equation}\label{Znorm}
\|f\|_{\mathcal{Z}}=\displaystyle\sum_{j=1}^2\displaystyle\sup_{v\in\mathcal{D}_{\e}}\left|e^{\omega(\sqrt{2}-|\Ip(v)|)}\pi_j\circ f(v)\right|.
\end{equation}

In order to prove Theorem \ref{exponentiallysmall}, it is enough to check that $\Delta\xi_1$ belongs to $\mathcal{Z}$ and that $\|\Delta\xi_1\|_{\mathcal{Z}}\leq M\omega\dg^3$. Our strategy to achieve these results is to prove that both $I(c_1-c_1^0,c_2-c_2^0)$ and $\mathcal{H}_0(\mathcal{M})$ belong to $\mathcal{Z}$ and that the operator $\mathrm{Id}- \mathcal{H}_0$ is invertible in $\mathcal{Z}$.


\begin{lemma}\label{lemmadoop}
	There exists $\e_0>0$, such that the linear operator $\mathrm{Id}-\mathcal{H}_0$ is invertible in $\mathcal{Z}$ for $\e\leq\e_0$. Furthermore, 	there exists $M>0$ independent of $\e$ such that  $\|\mathcal{H}_0\|_{\mathcal{Z}}\leq M\omega\dg^2$ and hence
	\begin{equation}\label{eqop}
	\|(\mathrm{Id}-\mathcal{H}_0)^{-1}\|_{\mathcal{Z}} \leq (1-\|\mathcal{H}_0\|_{\mathcal{Z}})^{-1}\leq 1+ M\omega\dg^2.
	\end{equation}
\end{lemma}
\begin{proof}
	Since $\mathcal{H}_0$ is a linear operator, to prove this lemma, it is sufficient to show that $\|\mathcal{H}_0\|_{\mathcal{Z}}\leq M\omega\dg^2<1$.
	
	Let $h\in\mathcal{Z}$ and denote by $M$ any constant independent of $\e$. Using \eqref{bjkbound} and \eqref{Znorm}, we have that 	for  $v\in\mathcal{D}_{\e}$ and $j=1,2$,
	
\[|\pi_j(B(v)\cdot h(v))|  \leq \displaystyle\sum_{k=1}^2 |b_{j,k}(v)\pi_k( h (v))|
 \leq M\omega\dg^2e^{-\omega(\sqrt{2}-|\Ip(v)|)}\|h\|_{\mathcal{Z}}.
	\]
Thus
	$$
	\begin{array}{lcl}
	\vspace{0.2cm}	|e^{\omega(\sqrt{2}-|\Ip(v)|)}\pi_1(\mathcal{H}_0(h)(v))| & = & \left|e^{\sqrt{2}\omega}\displaystyle\int_{v^*}^ve^{-i\omega (r-v-i|\Ip(v)|)}\pi_1(B(r)\cdot h(r))dr\right|\\
	\vspace{0.2cm} & \leq &M\omega\dg^2 e^{-\sqrt{2}\omega}e^{\sqrt{2}\omega}\|h\|_{\mathcal{Z}}\displaystyle\int_{v^*}^v\left|e^{-i\omega (r-v-i|\Ip(v)|)}\right|e^{\omega|\Ip(r)|}dr\\
	\vspace{0.2cm} & \leq &M\omega\dg^2\|h\|_{\mathcal{Z}}\displaystyle\int_{v^*}^ve^{\omega (\Ip(r)+|\Ip(r)|-\Ip(v)-|\Ip(v)|)}dr.
	\end{array}
	$$	
%
%
%
Since $\Ip(v^*)\leq \Ip(r)\leq \Ip(v)$, we have that $\Ip(r)+|\Ip(r)|-\Ip(v)-|\Ip(v)|\leq 0$, then
$$\left|\displaystyle\int_{v^*}^ve^{\omega (\Ip(r)+|\Ip(r)|-\Ip(v)-|\Ip(v)|)}dr\right|\leq M.$$ 
%
%
%
%
Analogously, we have that
	
	$$
	\begin{array}{lcl}
	\vspace{0.2cm}	|e^{\omega(\sqrt{2}-|\Ip(v)|)}\pi_2(\mathcal{H}_0(h)(v))| 
& \leq &M\omega\dg^2\|h\|_{\mathcal{Z}},
	\end{array}
	$$	
%
%
	and thus $\|\mathcal{H}_0(h)\|_{\mathcal{Z}}\leq M\omega\dg^2\|h\|_{\mathcal{Z}}.$ Since, $\|\mathcal{H}_0\|_{\mathcal{Z}}<1$, for $\e$ sufficiently small, the linear operator $\mathrm{Id}-\mathcal{H}_0$ is invertible and satisfies \eqref{eqop}.
	
	
\end{proof}

Now, recall that $\mathcal{M}=I(c_1^0,c_2^0)$, where $I$ is given by \eqref{expI} and $c_1^0,c_2^0$ are given by \eqref{melnikovctes}. Moreover, from Lemma \ref{lemma35}, we have that
\begin{equation}\label{idf0}
(\mathrm{Id}-\mathcal{H}_0)\Delta\xi_1= I(c_1-c_1^0,c_2-c_2^0) + \mathcal{H}_0(I(c_1^0,c_2^0)).
\end{equation}

Since $\mathrm{Id}-\mathcal{H}_0$ is invertible in $\mathcal{Z}$, it only remains to show that $ I(c_1-c_1^0,c_2-c_2^0)$ and $I(c_1^0,c_2^0)$ belong to $\mathcal{Z}$.

\begin{lemma}\label{lemmadoI}
	Given $k_1,k_2\in\C$, then the function $I$ given in \eqref{expI} satisfies
	\begin{equation}
	\|I(k_1,k_2)\|_{\mathcal{Z}}\leq Me^{\sqrt{2}\omega}(|k_1|+|k_2|),
	\end{equation}
	where $M$ is a constant independent of $\e$.
\end{lemma}
%
%
%

To prove Lemma \ref{lemmadoI} it is enough to recall the definitions of $\|\cdot\|_{\mathcal{Z}}$ in \eqref{Znorm} and $I$ in \eqref{expI}.

\begin{lemma}\label{bounderror_lem}
	The error vector $\Delta\xi_1$ given in \eqref{decomp} belongs to $\mathcal{Z}$ and it is determined by
	\begin{equation}\label{error}
	\Delta\xi_1= (\mathrm{Id}-\mathcal{H}_0)^{-1}\left(I(c_1-c_1^0,c_2-c_2^0\right) + (\mathrm{Id}-\mathcal{H}_0)^{-1}\left(\mathcal{H}_0(\mathcal{M})\right).
	\end{equation}
	Furthermore, there exists a constant $M>0$ independent of $\e$ such that 
	\begin{equation}\label{bounderror}
	\|\Delta\xi_1\|_{\mathcal{Z}}\leq M\omega\dg^3.
	\end{equation}	
\end{lemma}
\begin{proof}
From Lemmas \ref{lemma35} and \ref{melnikov_lem}, we have that $|c_j-c_j^0|\leq M\dg^3e^{-\sqrt{2}\omega},$ and  $ |c_j^0|\leq M \dg e^{-\sqrt{2}\omega}$,   $j=1,2$. Therefore, it follows from Lemma \ref{lemmadoI} that $I(c_1-c_1^0,c_2-c_2^0)\in \mathcal{Z}$, and $\mathcal{M}=I(c_1^0,c_2^0)\in\mathcal{Z}$. Furthermore
	$$\|I(c_1-c_1^0,c_2-c_2^0)\|_{\mathcal{Z}}\leq M\dg^3 \,\text{ and } \,\|\mathcal{M}\|_{\mathcal{Z}}\leq M\dg.$$
As $\mathrm{Id}-\mathcal{H}_0$ is invertible in $\mathcal{Z}$ by Lemma \ref{lemmadoop}, formula \eqref{error} is equivalent to \eqref{diffunc}.  Therefore, $\Delta\xi_1\in\mathcal{Z}$ and, using again Lemma \ref{lemmadoop},
	$$
	\begin{array}{lcl}
	\vspace{0.2cm}\|\Delta\xi_1\|_{\mathcal{Z}} &\leq & \|(\mathrm{Id}-\mathcal{H}_0)^{-1}\|_{\mathcal{Z}}\left(\|I(c_1-c_1^0,c_2-c_2^0)\|_{\mathcal{Z}}+\|\mathcal{H}_0(\mathcal{M})\|_{\mathcal{Z}}\right)\\
	\vspace{0.2cm} &\leq & M\dg^3 + M\|\mathcal{H}_0\|_{\mathcal{Z}}\|\mathcal{M}\|_{\mathcal{Z}}\\
	\vspace{0.2cm} &\leq & M\omega\dg^3.
	\end{array}
	$$
\end{proof}

\begin{proof}[Proof of Theorem \ref{exponentiallysmall}]

Finally, we prove that $\Delta\xi_1$ is exponentially small and we obtain an asympotic formula for $\Delta\xi$. From \eqref{bounderror} and the definition of the norm \eqref{Znorm}, we have
\begin{equation}\label{expf}
|e^{\omega(\sqrt{2}-|\Ip(v)|)}\pi_j\circ \Delta\xi_1(v)|\leq M\omega\dg^3,\, \text{ for } \,v\in\mathcal{D}_{\e},\text{ and }j=1,2.
\end{equation}
 In particular, if $v\in \mathcal{I}_{\e}=\mathcal{D}_{\e}\cap\R$,
$|\Delta\xi_1(v)|\leq M\omega\dg^3 e^{-\sqrt{2}\omega},$
for $j=1,2$. 
The result follows directly from this bound and \eqref{decomp}.

\end{proof}

\section{Proof of Theorem \ref{parameterization2Dh}}\label{parhper_sec}


In this section we look for parameterizations  of the invariant manifolds $W^{u}_\e(\Lambda_h^{-})$ of the periodic orbits $\Lambda_h^{-}$ of the form
\begin{equation}\label{n2dh}
N_{0,h}^{u}(v,\tau)=(X_0(v),Z_{0}(v)+Z_{0,h}^{u}(v,\tau),\Gamma_{h}(\tau)+\Gamma_{0,h}^{u}(v,\tau),\Theta_{h}(\tau)+\Theta_{0,h}^{u}(v,\tau)),
\end{equation}
where $Z_0,\Gamma_{h},\Theta_{h}$ are given in \eqref{par1} and \eqref{periodic}, as a perturbation of $N_{0,h}(v,\tau)$ (see \eqref{n2d}).

\begin{lemma}\label{pertsystem_lemh} The invariant manifold $W^u_{\dg}(\Lambda_h^-)$, with $\dg\neq 0$, can be parameterized by $N_{0,h}^{u}(v,\tau)$ in \eqref{n2dh} if $(Z_{0,h}^u(v,\tau),\Gamma_{0,h}^u(v,\tau),\Theta_{0,h}^u(v,\tau))$ satisfy the following system of partial differential equations
	\begin{equation}\label{edp sem restarh}
	\left\{
	\begin{array}{l}
	\begin{array}{rcl}
	\partial_v Z + \omega\partial_\tau Z+\dfrac{Z_0'(v)}{Z_0(v)}Z&=& -\dfrac{Z}{Z_0(v)}\partial_v Z -\dfrac{\delta}{\sqrt{2\Omega}} F'(X_0(v)) \dfrac{\Gamma-\Theta}{2i}\vspace{0.3cm}\\
	&&- \dfrac{\delta}{\sqrt{2\Omega}} F'(X_0(v)) \dfrac{\Gamma_{h}(\tau)-\Theta_h(\tau)}{2i},\vspace{0.3cm}\\
	\partial_v \Gamma + \omega\partial_\tau \Gamma&=&-\dfrac{Z}{Z_0(v)}\partial_v \Gamma+\omega i \Gamma-\dfrac{\dg}{\sqrt{2\Omega}}F(X_0(v)),\vspace{0.3cm}\\
	\partial_v \Theta + \omega\partial_\tau \Theta&=& -\dfrac{Z}{Z_0(v)}\partial_v \Theta-\omega i \Theta-\dfrac{\dg}{\sqrt{2\Omega}}F(X_0(v)),\vspace{0.3cm}
	\end{array}\\
	\displaystyle\lim_{v\rightarrow-\infty}Z(v,\tau)=\displaystyle\lim_{v\rightarrow-\infty}\Gamma(v,\tau)=\displaystyle\lim_{v\rightarrow-\infty}\Theta(v,\tau)=0, \textrm{ for each $\tau\in[0,2\pi]$,	}
	\end{array}
	\right.
	\end{equation}
	and $Z_{0,h}^u,\Gamma_{0,h}^u,\Theta_{0,h}^u$ are $2\pi$-periodic in the variable $\tau.$	
	
\end{lemma}
	In contrast to the $1$-dimensional case, for technical reasons, we do not use that $\mathcal{H}(W^{u}_{\e}(\Lambda_h^-))=h$ to obtain $Z=Z(X,\Gamma,\Theta)$. Thus, we deal with the problem in dimension $3$.

 As in the $1$-dimensional case \eqref{formaop}, if we set $Z=\Gamma=\Theta=0$, the right-hand side of \eqref{edp sem restarh} decays as $1/|v|$ as $v\rightarrow-\infty$. To have quadratic decay as $|v|\rightarrow \infty$ to have integrability, we perform with the change \eqref{eq2d} to  system \eqref{edp sem restarh}. Then, $(z_{0,h}^u,\gamma_{0,h}^u,\theta_{0,h}^u)$ satisfy
\begin{equation}\label{edp restada}
	\left\{
	\begin{array}{l}\begin{array}{rcl}
			\partial_v z + \omega\partial_\tau z+\dfrac{Z_0'(v)}{Z_0(v)}z&=&f_1^h(v,\tau) -\dfrac{z+Z_{0,h}(v,\tau)}{Z_0(v)}\partial_v z-\dfrac{\partial_vZ_{0,h}(v,\tau)}{Z_0(v)}z\vspace{0.3cm}\\ &&-\dfrac{\delta}{\sqrt{2\Omega}} F'(X_0(v)) \dfrac{\gamma-\theta}{2i}\vspace{0.3cm}\\
			\partial_v \gamma + \omega\partial_\tau \gamma-\omega i\gamma&=& f_2^h(v,\tau)-\dfrac{(Q^0)'(v)}{Z_0(v)}z-\dfrac{z+Z_{0,h}(v,\tau)}{Z_0(v)}\partial_v\gamma,\vspace{0.3cm}\\
			\partial_v \theta + \omega\partial_\tau \theta+\omega i\theta
			&=& -f_2^h(v,\tau)+\dfrac{(Q^0)'(v)}{Z_0(v)}z-\dfrac{z+Z_{0,h}(v,\tau)}{Z_0(v)}\partial_v\theta,
		\end{array}\vspace{0.3cm}\\
		\displaystyle\lim_{v\rightarrow -\infty}z(v,\tau)=\displaystyle\lim_{v\rightarrow -\infty}\gamma(v,\tau)=\displaystyle\lim_{v\rightarrow -\infty}\theta(v,\tau)=0,
	\end{array}\right.
\end{equation}	
where 

\begin{align}
	f_1^h(v,\tau)=&-\partial_vZ_{0,h}(v,\tau)-\dfrac{Z_0'(v)}{Z_0(v)}Z_{0,h}(v,\tau)\label{f1}\\
		 &-\dfrac{\dg}{\sqrt{2\Omega}}F'(X_0(v))\dfrac{Q^0(v)}{i}-\dfrac{Z_{0,h}(v,\tau)\partial_vZ_{0,h}(v,\tau)}{Z_0(v)},\notag\\
	f_2^h(v,\tau)=&-(Q^0)'(v)-\dfrac{Z_{0,h}(v,\tau)(Q^0)'(v)}{Z_0(v)},\label{f2}
\end{align} 
and $Q^0,\ Z_{0,h}$ are given by \eqref{first0}, \eqref{Zfirst}, respectively.

We consider  equation \eqref{edp restada} with $(v, \tau)\in D^u\times \mathbb{T}_{\sigma}$ (see \eqref{outerdomain} and \eqref{tsig}), and asymptotic conditions 
$	\displaystyle\lim_{\Rp(v)\rightarrow -\infty}z(v,\tau)=\displaystyle\lim_{\Rp(v)\rightarrow -\infty}\gamma(v,\tau)=\displaystyle\lim_{\Rp(v)\rightarrow -\infty}\theta(v,\tau)=0,$
for every $\tau\in\mathbb{T}_{\sigma}$.


\begin{prop}\label{existence}
	Fix $\sigma>0$ and $h_0>0$. There exists $\e_0>0$ sufficiently small 
	such that for $0<\e\leq \e_0$ and $0\leq h\leq h_0$, equation \eqref{edp restada} has a solution $(z_{0,h}^u,\gamma_{0,h}^u,\theta_{0,h}^u)$ defined in $D^{u}\times \mathbb{T}_\sigma$ such that $z_{0,h}^u$ is real-analytic, $\gamma_{0,h}^u,\theta_{0,h}^u$ are analytic, $\theta_{0,h}^u(v,\tau)=\overline{\gamma_{0,h}^u(v,\tau)}$ for each $(v,\tau)\in\rn{2}$, and
	$$\displaystyle\lim_{\Rp(v)\rightarrow-\infty}z_{0,h}^u(v,\tau)=\displaystyle\lim_{\Rp(v)\rightarrow-\infty}\gamma_{0,h}^u(v,\tau)=\displaystyle\lim_{\Rp(v)\rightarrow-\infty}\theta_{0,h}^u(v,\tau)=0,$$
	for every $\tau\in\mathbb{T}_{\sigma}$. Furthermore, $(z_{0,h}^u,\gamma_{0,h}^u,\theta_{0,h}^u)$ satisfy the bounds in \eqref{bounds38}.
\end{prop}
We devote the rest of this section to prove Proposition \ref{existence}.  
Equation \eqref{edp restada} can be written as the functional equation

\begin{equation}\label{funcform}
\mathcal{L}_{\omega}(z,\gamma,\theta)= \mathcal{P}_h(z,\gamma,\theta),
\end{equation}
where $\mathcal{L}_{\omega}$ and $\mathcal{P}_h$ are the operators 

\begin{equation}\label{Lo}
\mathcal{L}_{\omega}(z,\gamma,\theta)=\left(\begin{array}{l}
\partial_v z + \omega\partial_\tau z+\dfrac{Z_0'(v)}{Z_0(v)}z\vspace{0.2cm}\\	
\partial_v \gamma + \omega\partial_\tau \gamma-\omega i\gamma \vspace{0.2cm}\\
\partial_v \theta + \omega\partial_\tau \theta+\omega i\theta
\end{array}\right),
\end{equation}
\begin{equation}\label{Fo}
\mathcal{P}_h(z,\gamma,\theta)=\left(\begin{array}{l}
f_1^h(v,\tau) -\dfrac{z+Z_{0,h}(v,\tau)}{Z_0(v)}\partial_v z-\dfrac{\partial_vZ_{0,h}}{Z_0(v)}z-\dfrac{\delta}{\sqrt{2\Omega}} F'(X_0(v)) \dfrac{\gamma-\theta}{2i} \vspace{0.2cm}\\
f_2^h(v,\tau)-\dfrac{(Q^0)'(v)}{Z_0(v)}z-\dfrac{z+Z_{0,h}(v,\tau)}{Z_0(v)}\partial_v\gamma\vspace{0.2cm}\\
-f_2^h(v,\tau)+\dfrac{(Q^0)'(v)}{Z_0(v)}z-\dfrac{z+Z_{0,h}(v,\tau)}{Z_0(v)}\partial_v\theta
\end{array}\right).
\end{equation}


\subsection{Banach spaces and technical results}\label{banachper}



For analytic functions $f:D^{u}\rightarrow \C$ and $g:D^{u}\times\mathbb{T}_{\sigma}\rightarrow \C$  and $\ag>0$,  we define
\begin{equation}\label{normabonita}
\begin{split}
\|f\|_{\alpha}&= \sup_{v\in D^{u}}|(v^2+2)^{\ag/2} f(v)|,\\
\|g\|_{ \alpha,\sigma}&= \displaystyle\sum_{k\in\Z} \|g^{[k]}\|_{\alpha}e^{|k|\sigma},
\end{split}
\end{equation}
where $g(v,\tau)=\displaystyle\sum_{k\in\Z}g^{[k]}(v)e^{ik\tau}$.

\begin{remark}\label{remark do peso}
	Notice that there exists a constant $d>0$ independent of $\e$ such that  the distance between each $v\in D^u$ (given in \eqref{outerdomain})  and the poles $\pm i\sqrt{2}$ of $N_{0,h}(v,\tau)$ (given in \eqref{n2d}) is greater than $d$. The weight $|v^2+2|^{\ag/2}$ in the norm $\|\cdot\|_{\alpha}$ is chosen to control the behavior as $\Rp v\to -\infty$ and to have it well-defined for $v=0\in D^u$. In fact, at infinity this norm is equivalent to the norm with weight $|v|^{\ag}$.
\end{remark}

We also define 
\begin{equation}
\llbracket g \rrbracket_{\alpha,\sigma}=\max\{ \|g\|_{\alpha,\sigma},\|\partial_\tau g\|_{\alpha,\sigma},\|\partial_v g\|_{\alpha+1,\sigma}\}
\end{equation}

and  the Banach spaces
\begin{equation*}
\begin{array}{lcl}
\mathcal{X}_{\alpha,\sigma}&=&\{g:D^u\times\mathbb{T}_{\sigma}\rightarrow \C \textrm{ is an analytic function, }\textrm{ such that }  \|g\|_{\ag,\sigma}<\infty \},\\
\mathcal{Y}_{ \ag,\sigma}&=&\{g:D^u\times\mathbb{T}_{\sigma}\rightarrow \C \textrm{ is an analytic function},\textrm{ such that }  \llbracket g\rrbracket_{\ag,\sigma}<\infty \}.\\
\end{array}
\end{equation*}
Consider the product spaces 
\begin{equation*}
\begin{array}{lcl}\mathcal{X}_{\ag,\sigma}^{3}&=&\left\{(f,g,h)\in \mathcal{X}_{\ag,\sigma}\times\mathcal{X}_{\ag,\sigma}\times \mathcal{X}_{\ag,\sigma};\  f\ \textrm{is real-analytic, } g(v,\tau)=\overline{h(v,\tau)}, \textrm{ for every }v\in D^u\cap\R,\tau\in\mathbb{T}
\right\},\vspace{0.2cm}\\
\mathcal{Y}_{\ag,\sigma}^{3}&=&\left\{(f,g,h)\in \mathcal{Y}_{\ag,\sigma}\times\mathcal{Y}_{\ag,\sigma}\times \mathcal{Y}_{\ag,\sigma};\ f\ \textrm{is real-analytic, } g(v,\tau)=\overline{h(v,\tau)}, \textrm{ for every }v\in D^u\cap\R,\tau\in\mathbb{T}
\right\},\end{array}\end{equation*} endowed with the norms \begin{equation*}
\begin{array}{lcl}\|(f,g,h)\|_{\ag,\sigma}&=&\|f\|_{\ag,\sigma}+\|g\|_{\ag,\sigma}+\|h\|_{\ag,\sigma},\vspace{0.2cm}\\ \llbracket(f,g,h)\rrbracket_{\ag,\sigma}&=&\llbracket f\rrbracket_{\ag,\sigma}+ \llbracket g\rrbracket_{\ag,\sigma}+ \llbracket h\rrbracket_{\ag,\sigma},\end{array}\end{equation*}
respectively. We present some properties of the norm $\|\cdot\|_{\ag,\sigma}$, which are proven in \cite{BFGS12}.
\begin{lemma}\label{properties}
	Given real-analytic functions $f:\C\rightarrow \C$, $g,h:D^u\times\mathbb{T}_{\sigma}\rightarrow\C$, the following statements hold
	\begin{enumerate}
		\item If $\ag_1\geq\ag_2\geq 0$, then
		$$\|h\|_{\ag_2,\sigma}\leq \|h\|_{\ag_1,\sigma}.$$	
		\item If $\ag_1,\ag_2\geq 0$, and $\|g\|_{\ag_1,\sigma},\|h\|_{\ag_2,\sigma}<\infty$, then
		$$\|gh\|_{\ag_1+\ag_2,\sigma}\leq \|g\|_{\ag_1,\sigma}\|h\|_{\ag_2,\sigma}.$$	
		\item If $\|g\|_{\ag,\sigma}$, $\|h\|_{\ag,\sigma}\leq R_0/4$, where $R_0$ is the convergence ratio of $f'$ at $0$, then
		$$\|f(g)-f(h)\|_{\ag,\sigma}\leq M \|g-h\|_{\ag,\sigma}.$$
	\end{enumerate}	
\end{lemma}

\subsection{The Operators $\mathcal{L}_{\omega}$ and $\mathcal{G}_{\omega}$}
Let $f$, $g$, and $h$ be analytic functions defined in $D^u\times\mathbb{T}_{\sigma}$. We define
\begin{equation}\label{inteq}
\begin{array}{l}
F^{[k]}(f)(v)= \displaystyle\int_{-\infty}^v\dfrac{e^{\omega i k(r-v)}Z_0(r)}{Z_0(v)}f^{[k]}(r)dr,\vspace{0.2cm}\\
G^{[k]}(g)(v)=\displaystyle\int_{-\infty}^{v}e^{\omega i (k-1) (r-v)}g^{[k]}(r)dr\vspace{0.2cm},\\
H^{[k]}(h)(v)=\displaystyle\int_{-\infty}^{v}e^{\omega i (k+1) (r-v)}h^{[k]}(r)dr,
\end{array}
\end{equation}
and consider the linear operator $\mathcal{G}_{\omega}$ given by
\begin{equation}\label{Go}
\mathcal{G}_{\omega}(f,g,h)=\left(\begin{array}{l}
\displaystyle\sum_k F^{[k]}(f)(v) e^{ik\tau}\vspace{0.2cm}\\
\displaystyle\sum_k G^{[k]}(g)(v) e^{ik\tau}\vspace{0.2cm}\\
\displaystyle\sum_k H^{[k]}(h)(v) e^{ik\tau}
\end{array}\right).
\end{equation}


\begin{lemma}\label{opgo}
	Fix $\alpha\geq 1$ and $\sigma>0$, the operator $$\mathcal{G}_{\omega}:
	\mathcal{X}_{\alpha+1,\sigma}^3\rightarrow \mathcal{Y}_{\alpha,\sigma}^3$$ given in \eqref{Go} is well-defined and the following statements hold:
	\begin{enumerate}
		\item $\mathcal{G}_{\omega}$ is an inverse of the operator $\mathcal{L}_{\omega}:\mathcal{Y}_{\ag,\sigma}^{3}\rightarrow \mathcal{X}_{\alpha+1,\sigma}^3$ given in \eqref{Lo}, i.e. $\mathcal{G}_{\omega}\circ\mathcal{L}_\omega=\mathcal{L}_{\omega}\circ\mathcal{G}_\omega= \mathrm{Id}$;
		\item $\llbracket \mathcal{G}_{\omega}(f,g,h)\rrbracket_{\ag,\sigma}\leq M \|(f,g,h)\|_{\ag+1,\sigma}$;
		\item If $f^{[0]}=g^{[1]}=h^{[-1]}=0$, then $\llbracket \mathcal{G}_{\omega}(f,g,h)\rrbracket_{\ag,\sigma}\leq \dfrac{M}{\omega} \llbracket (f,g,h)\rrbracket_{\ag,\sigma}$.
	\end{enumerate}
	%
	%
	%
\end{lemma}

The proof of Lemma \ref{opgo} can be found in  \cite{BFGS12}.

To find a solution of \eqref{edp restada}, it is sufficient to find a fixed point of the operator
\begin{equation}\label{gob}
\overline{\mathcal{G}}_{\omega,h}=\mathcal{G}_{\omega}\circ \mathcal{P}_h,
\end{equation}
where $\mathcal{G}_{\omega}$ is given by \eqref{Go} and $\mathcal{P}_h$ is given by \eqref{Fo}.

\subsection{The Operator $\mathcal{P}_h$}

We show some properties of the operator $\mathcal{P}_h$ defined in \eqref{Fo}.

\begin{lemma}\label{firstit}
	Fix $\sigma>0,\ h_0>0$. For  $0\leq h\leq h_0$, the operator $\mathcal{P}_h$ defined in \eqref{Fo} satisfies
	%
	%
	
	$$\left\|\mathcal{P}_h(0,0,0)\right\|_{2,\sigma}\leq M\dfrac{\dg}{\omega}.$$
\end{lemma}
\begin{proof}
	Notice that $\mathcal{P}_{h}(0,0,0)=(f_1^h,f_2^h,-f_2^h)$, where $f_1^h$ and $f_2^h$ are given by \eqref{f1}, and \eqref{f2} respectively, and involve the functions $F'(X_0),\ Z_0,\ Z_0',\ Q^0,\ Q_0',$ $Z_{0,h},\partial_v Z_{0,h}$. By \eqref{exprUF}, \eqref{par1}, \eqref{first0} and \eqref{Zfirst}, we can see that
	
%
%
	$$
	\begin{array}{l}
	\|Q^0\|_{1,\sigma}, \|(Q^0)'\|_{2,\sigma} \leq M\dfrac{\dg}{\omega},\vspace{0.2cm}\\	
	\left\|Z_{0,h}\right\|_{1,\sigma},\|\partial_vZ_{0,h}\|_{2,\sigma}\leq M\dfrac{\dg\sqrt{h}}{\omega^{3/2}},\vspace{0.2cm}\\ \|Z_0\|_{1,\sigma}, \|Z_0'\|_{2,\sigma}, \|F'(X_0)\|_{1,\sigma} \leq M.
	\end{array}
	$$
	It follows from these bounds and Lemma \ref{properties} that 
	$$\|f_1^h\|_{2,\sigma}\leq M\max\left\{\dfrac{\dg\sqrt{h}}{\omega^{3/2}},\dfrac{\dg^2}{\omega},\dfrac{\dg^2}{\omega^3}h\right\}=M\max\left\{\dfrac{\dg\sqrt{h}}{\omega^{3/2}},\dfrac{\dg^2}{\omega}\right\},$$
	$$\|f_2^h\|_{2,\sigma}\leq M\max\left\{\dfrac{\dg}{\omega},\dfrac{\dg^2}{\omega^{5/2}}\sqrt{h}\right\}=M\dfrac{\dg}{\omega}.$$
	%
	%
	%
	%
	%
	%
	%

\end{proof}

\begin{lemma}\label{lip}
	Fix $\sigma>0,\ h_0>0$ and $K>0$. If $0\leq h\leq h_0$,  the operator 
	$$\mathcal{P}_h: \mathcal{Y}_{1,\sigma}^{3}\rightarrow\mathcal{X}_{2,\sigma}^{3}$$
	is well defined.
	Moreover, given $(z_j,\gamma_j,\theta_j)\in\mathcal{B}_{0}(K\dg/\omega)\subset\mathcal{Y}_{1,\sigma}^{3},\ j=1,2$, 
	\begin{equation}\label{lipph}
	\left\|	\mathcal{P}_h(z_1,\gamma_1,\theta_1)- \mathcal{P}_h(z_2,\gamma_2,\theta_2)\right\|_{2,\sigma}\leq M\left(\dg+\dfrac{\dg}{\omega^{3/2}}\sqrt{h}\right)\left\llbracket	(z_1,\gamma_1,\theta_1)- (z_2,\gamma_2,\theta_2)\right\rrbracket_{1,\sigma},
	\end{equation}
	where $M$ is a constant independent of $\e$ and $h$.	
\end{lemma}
\begin{proof}
	It is straightforward to see that $\mathcal{P}_h$ is well defined. Denote $\mathcal{P}_h^j=\pi_j\circ\mathcal{P}_{h}$. We  show the bound of the difference for $\mathcal{P}_h^1$ and $\mathcal{P}_h^2$, since the bound of $\mathcal{P}_h^3$ can be obtained in exactly the same way as $\mathcal{P}_h^2$. 
	
	Notice that
	$$
	\begin{array}{lcl}
	\mathcal{P}_h^1(z_1,\gamma_1,\theta_1)- \mathcal{P}_h^1(z_2,\gamma_2,\theta_2) &=&  -\dfrac{\dg}{\sqrt{2\Omega}}F'(X_0(v))\dfrac{(\gamma_1-\gamma_2)-(\theta_1-\theta_2)}{2i}\vspace{0.2cm}\\
	&&-\dfrac{\partial_vZ_{0,h}(v,\tau)}{Z_0(v)}(z_1-z_2)-\partial_vz_2\dfrac{z_1-z_2}{Z_0(v)}		\vspace{0.2cm}\\
	&&-\dfrac{z_1+Z_{0,h}(v,\tau)}{Z_0(v)}(\partial_vz_1-\partial_vz_2).
	\end{array}
	$$
	Using the bounds contained in the proof of Lemma \ref{firstit} and that $Z_0$ is lower bounded in $D^u$ by a positive constant independent of $\e$, one can see that
	$$\left\|\mathcal{P}_h^1(z_1,\gamma_1,\theta_1)- \mathcal{P}_h^1(z_2,\gamma_2,\theta_2)\right\|_{2,\sigma}\leq  M\max\left\{\dg,\dfrac{\dg}{\omega^{3/2}}\sqrt{h}\right\}\left\llbracket	(z_1,\gamma_1,\theta_1)- (z_2,\gamma_2,\theta_2)\right\rrbracket_{1,\sigma}.$$
	
	
	Now,
	$$
	\begin{array}{lcl}
	\mathcal{P}_h^2(z_1,\gamma_1,\theta_1)- \mathcal{P}_h^2(z_2,\gamma_2,\theta_2) &=&   -\dfrac{(Q^0)'(v)}{Z_0(v)}(z_1-z_2)-\partial_v\gamma_2\dfrac{z_1-z_2}{Z_0(v)}		\vspace{0.2cm}\\
	&&-\dfrac{z_1+Z_{0,h}(v,\tau)}{Z_0(v)}(\partial_v\gamma_1-\partial_v\gamma_2)\vspace{0.2cm}\\
	\end{array}
	$$
	which, proceeding analogously,
	$$\left\|\mathcal{P}_h^2(z_1,\gamma_1,\theta_1)- \mathcal{P}_h^2(z_2,\gamma_2,\theta_2)\right\|_{2,\sigma}\leq M\max\left\{\dfrac{\dg}{\omega},\dfrac{\dg}{\omega^{3/2}}\sqrt{h}\right\}\left\llbracket	(z_1,\gamma_1,\theta_1)- (z_2,\gamma_2,\theta_2)\right\rrbracket_{1,\sigma}.$$
	
\end{proof}

\subsection{The Fixed Point Theorem}	

Now, we write Proposition \ref{existence} in terms of Banach spaces and we prove it through a fixed point argument  applied to the operator $\overline{\mathcal{G}}_{\omega,h}$ given by \eqref{gob}. 
\begin{prop}\label{fixpointgob}
	Fix $\sigma>0$ and $h_0>0$. There exists $\e_0>0$ such that for $0<\e\leq \e_0$, the operator $\overline{\mathcal{G}}_{\omega,h}$ in \eqref{gob} has a  fixed point $(z_{0,h}^{u},\gamma_{0,h}^{u},\theta_{0,h}^{u})\in\mathcal{Y}_{1,\sigma}^3$. Furthermore, there exists a constant $M>0$ independent of $\e$ and $h$ such that
	$$\llbracket(z_{0,h}^{u},\gamma_{0,h}^{u},\theta_{0,h}^{u})\rrbracket_{1,\sigma}\leq M\dfrac{\dg}{\omega}.$$
\end{prop}
\begin{proof}
	From Lemmas \ref{opgo} and \ref{firstit}, there exists a constant $b_2>0$ independent of $\e$ and $h$ such that
	$$\llbracket\overline{\mathcal{G}}_{\omega,h}(0,0,0)\rrbracket_{1,\sigma}\leq M \|\mathcal{P}_h(0,0,0)\|_{2,\sigma}\leq \dfrac{b_2}{2}\dfrac{\dg}{\omega}.$$
	Consider the operator $\overline{\mathcal{G}}_{\omega,h}=\mathcal{G}_{\omega}\circ\mathcal{P}_h: \mathcal{B}_{0}(b_2\dg/\omega)\subset\mathcal{Y}_{1,\sigma}\rightarrow \mathcal{Y}_{1,\sigma}$. Notice that Lemmas \ref{opgo} and \ref{lip} imply that it is well defined in these spaces. 
	
	To show that $\overline{\mathcal{G}}_{\omega,h}$ sends $\mathcal{B}_0(b_2\dg/\omega)$ into itself, consider $K=b_2$ in Lemma \ref{lip} and $(z_j,\gamma_j,\theta_j)\in\mathcal{B}_0(b_2\dg/\omega),\ j=1,2$.
	It follows from Lemmas \ref{opgo}, \ref{lip} and the fact that $\mathcal{G}_{\omega}$ is a linear operator that
	$$
	\begin{array}{lcl}
	\llbracket \overline{\mathcal{G}}_{\omega,h}(z_1,\gamma_1,\theta_1)- \overline{\mathcal{G}}_{\omega,h}(z_2,\gamma_2,\theta_2)\rrbracket_{1,\sigma}&\leq& M \left\|	\mathcal{P}_h(z_1,\gamma_1,\theta_1)- \mathcal{P}_h(z_2,\gamma_2,\theta_2)\right\|_{2,\sigma},\vspace{0.2cm}\\
	&\leq& M\dg\left\llbracket	(z_1,\gamma_1,\theta_1)- (z_2,\gamma_2,\theta_2)\right\rrbracket_{1,\sigma}.
	\end{array}
	$$
	
	Choosing $\e_0$ sufficiently small such that $\mathrm{Lip}(\overline{\mathcal{G}}_{\omega,h})<1/2$, $\overline{\mathcal{G}}_{\omega,h}$ sends $ \mathcal{B}_0(b_2\dg/\omega)$ into itself and it is a contraction. Thus, it has a unique fixed point $(z_{0,h}^{u},\gamma_{0,h}^{u},\theta_{0,h}^{u})\in \mathcal{B}_0(b_2\dg/\omega)$.
\end{proof}



\section{Proof of Theorem \ref{parameterization1Dh}}\label{parh_sec}


The strategy used to prove Theorem \ref{parameterization1Dh} is analogous to the one of Theorem \ref{parameterization1D0} taking into account that  all the expressions appearing become  singular as $h \to 0$.


We write
\begin{equation}\label{n1dh}
N_{h,0}^{u}(v)=(X_h(v),Z_{h,0}^{u}(v),\Gamma_{h,0}^{u}(v),\Theta_{h,0}^{u}(v)).
\end{equation}

\begin{lemma}\label{pertsystemh_lem} Given $h>0$, the invariant manifold $W^u_{\dg}(p_h^-)$, with $\dg\neq 0$, is parameterized by $N_{h,0}^{u}(v)$ if and only if $(\Gamma_{h,0}^u(v),\Theta_{h,0}^u(v))$ satisfy
	\begin{equation}\label{edo sem restarh}
	\left\{\begin{array}{l}\begin{array}{l}
	\vspace{0.2cm}\dfrac{ d\Gamma}{d v}(v)= \dfrac{Z_h(v)}{\widetilde{\eta}_h(v,\Gamma,\Theta)}\left(\omega i \Gamma(v) -\dfrac{\delta}{\sqrt{2\Omega}} F(X_h(v))\right),\\
	
	\dfrac{d \Theta}{d v}(v)= \dfrac{Z_h(v)}{\widetilde{\eta}_h(v,\Gamma,\Theta)}\left(-\omega i\Theta(v) -\dfrac{\delta}{\sqrt{2\Omega}} F(X_h(v)))\right).
	\end{array}\vspace{0.2cm}\\
	\displaystyle\lim_{v\rightarrow-\infty}\Gamma(v)=\displaystyle\lim_{v\rightarrow-\infty}\Theta(v)=0,
	\end{array}
	\right.
	\end{equation}
	and 
	$$\widetilde{\eta}_h(v,\Gamma,\Theta)=4\sqrt{h-U(X_h(v))- \dfrac{\dg}{\sqrt{2\Omega}} F(X_h(v)) \dfrac{\Gamma(v)-\Theta(v)}{2i} - \dfrac{\omega}{2}\Gamma(v)\Theta(v)},$$
	with $X_h$ given in \eqref{par1}, $U,F$given in \eqref{exprUF},
	and $Z_{h,0}^u(v)=\widetilde{\eta}_h(v,\Gamma_{h,0}^u(v),\Theta_{h,0}^u(v))$.
\end{lemma}

As in Section \ref{par0_sec}, we  compute an explicit term of $(\Gamma^u_{h,0},\Theta^u_{h,0})$. Thus, the solution of \eqref{edo sem restarh} can be written as \eqref{solh0} and $(\gamma_{h,0}^{u},\theta_{h,0}^u)$ satisfy
	\begin{equation}\label{edo restadah}
	\left\{\begin{array}{l}
	\vspace{0.2cm}\dfrac{d}{d v}\gamma-\omega i\gamma=\omega i \gamma(\eta_h(v,\gamma,\theta)-1)- (Q^h)'(v),\\
	
	\vspace{0.2cm}\dfrac{d}{d v}\theta+\omega i\theta=-\omega i \theta(\eta_h(v,\gamma,\theta)-1)+ (Q^h)'(v),\\
	\displaystyle\lim_{v\rightarrow -\infty}\gamma(v)=\displaystyle\lim_{v\rightarrow -\infty}\theta(v)=0,
	\end{array} \right.
	\end{equation}	
	where $Q^h$ is given in \eqref{firsth} and
	\begin{equation}
	\label{alphah1d}
	\eta_h(v,\gamma,\theta)=\left(1+\dfrac{4\dg^2}{\Omega\omega}\left(\dfrac{F(X_h(v))}{Z_h(v)}\right)^2-8\omega\dfrac{\gamma\theta}{(Z_h(v))^2}\right)^{-1/2}.
	\end{equation}
We prove Theorem \ref{parameterization1Dh} by finding a solution of \eqref{edo restadah} in the next proposition.

\begin{prop}\label{solutionh}
	There exists $\e_0>0$ and $h_0>0$ such that for $0<h\leq h_0$ and $0<\e\leq \e_0$,  equation \eqref{edo restadah}  has a solution $(\gamma_{h,0}^{u}(v),\theta_{h,0}^{u}(v))$ defined in $D^u$ (see \eqref{outerdomain})  such that $\theta_{h,0}^{u}(v)=\overline{\gamma_{h,0}^{u}(v)}$ for every $v\in\R$. Furthermore, $(\gamma_{h,0}^u,\theta_{h,0}^u)$ satisfy the bound \eqref{cota2}.
\end{prop}

To prove Proposition \ref{solutionh}, it is sufficient to find a fixed point $(\gamma_{h,0}^u,\theta_{h,0}^u)$ of the operator
\begin{equation}\label{fixh}\mathcal{G}_{\omega,h}=\mathcal{G}_{\omega}\circ\mathcal{F}_h,\end{equation} where $\mathcal{G}_{\omega}$ is given in \eqref{g0} and
	\begin{equation}\label{Fh}
	\mathcal{F}_{h}(\gamma,\theta)(v)=\left(\begin{array}{c}
	\omega i \gamma(v)(\eta_h(v,\gamma(v),\theta(v))-1)- (Q^h)'(v)\vspace{0.2cm}\\
	-\omega i \theta(v)(\eta_h(v,\gamma(v),\theta(v))-1)+ (Q^h)'(v)
	\end{array}\right),
	\end{equation}
	and $Q^h,\eta_h$ are given in \eqref{firsth} and \eqref{alphah1d}, respectively.

The rest of this section is devoted to find a fixed point of \eqref{fixh}.


\subsection{Banach spaces and technical lemmas}\label{banachh}

By \eqref{exprUF}, \eqref{par2} and \eqref{firsth}
\begin{equation}
\label{FXh}
\begin{array}{lcl}
Q^h(v)
&=&\dfrac{2\dg i}{\omega\sqrt{2\Omega}}\left(\dfrac{\sqrt{\frac{2+h}{h}}\sinh(v\sqrt{h}/2)}{1+\frac{2+h}{h}\sinh^2(v\sqrt{h}/2)}\right),
\end{array}
\end{equation}
which has poles at
\begin{equation}
s_{h,k}^{\pm,j}=i\dfrac{2}{\sqrt{h}}\left(\dg_{j,1}\pi\pm\arcsin\left(\sqrt{\dfrac{h}{2+h}}\right)+ 2k\pi\right),
\end{equation}
where $\dg_{j,1}$ is the delta of Kronecker, $j=0,1$ and $k\in\Z$. All these singularities are contained in the imaginary axis and satisfy
$$s_{h,k}^{\pm,j}=i\left(\pm\sqrt{2}+\er(h)+ \dfrac{2}{\sqrt{h}}\left(\dg_{j,1}\pi+2k\pi\right)\right).$$
Thus, for $h$  sufficiently small $\left|s_{h,k}^{\pm,j}\right|\geq 3\sqrt{2}/4$,  $j=0,1$ and $k\in\Z$.  


Therefore, we can consider the same domain $D^u$ in \eqref{outerdomain}. It satisfies the following property, whose proof is straightforward.

\begin{lemma}\label{domain_lem}
	If $v\in D^u$ is such that $|\Rp(v)|\geq \chi_0$, for some $\chi_0>0$, then 
	$$|\Ip(v)|\leq \dfrac{\chi_0+1}{\chi_0}|\Rp(v)|.$$
\end{lemma}


For  $\ag\geq 0$, we consider the Banach space
\begin{equation}\label{space2}
\mathcal{X}_{\ag}=\left\{f: D^u\rightarrow C;\ f \textrm{ is analytic and } \|f\|_\ag<\infty  \right\}
\end{equation}
endowed with the norm 
\begin{equation}\label{norm2}
\|f\|_{\ag}=\displaystyle\sup_{v\in D^u}|(v^2+2)^{\ag/2}f(v)|,
\end{equation}
and the product space $$\mathcal{X}_{\ag}^2=\left\{(f,g)\in \mathcal{X}_{\ag}\times \mathcal{X}_{\ag};\ g(v)=\overline{f(v)} \textrm{ for every }v\in\R  \right\}$$ endowed with the norm 
$\|(f,g)\|_\ag= \|f\|_\ag+ \|g\|_\ag$.
Remark \ref{remark do peso} and  Lemma \ref{properties} also apply to $\|\cdot\|_{\ag}$.



\begin{lemma}\label{sinh}
	Given $0<h_0\leq 1$, there exists a constant $M^*>0$ such that, for each $v\in D^u$ and $0<h\leq h_0$,
	$$\left|\sinh(v\sqrt{h}/2)\right|\geq M^*\sqrt{h}|v|, \qquad \ \left|\cosh(v\sqrt{h}/2)\right|\geq M^*.$$
\end{lemma}
The following Lemma is proved in \cite{INMA}.
\begin{lemma}\label{cauchy} Let $1/2<\bg<\pi/4$ be fixed. The following statements hold
	\begin{enumerate}
		\item 	There exists $\bg_0>0$ sufficiently small such that $D^u\subset D^u(\bg_0)$, where
		$$D^u(\bg_0)=\left\{v\in\C;\ |\Ip(v)|\leq -\tan(\bg+\bg_0)\Rp(v)+2\sqrt{2}/3\right\}.$$
		\item Given $\ag>0$, if $f: D^u(\bg_0)\rightarrow \C$ is a real-analytic function such that 
		$$m_{\ag}(f)=\sup_{v\in D^u(\bg_0)}|(v^2+2)^{\ag/2}f(v)|<\infty,$$	
		then, for any $n\in\N$
		$$\|f^{(n)}\|_{\ag+n}\leq Mm_{\ag}(f).$$
	\end{enumerate}
	
	%
\end{lemma}

	In the remaining of this paper, all the Landau symbols $\er(f(v,h,\e))$ denote a function dependent on $v,h$ and $\e$  such that there exists a constant $M>0$ independent of $h$ and $\e$ such that $|\er(f(v,h,\e))|\leq M|f(v,h,\e)|$, for every $(v,h,\e)$ in the domain considered.

\begin{lemma}\label{Fs}
	There exist $h_0\in(0,1)$ and a constant $M>0$  such that, for $v\in D^u$ and $0<h\leq h_0$,
	\begin{enumerate}
		\item $|F(X_h(v))|\leq \dfrac{M}{|\sqrt{v^2+2}|}$;\vspace{0.2cm}
		\item $|F(X_h(v))'|\leq \dfrac{M}{|v^2+2|}$.
	\end{enumerate}
	where $X_h$  given in \eqref{par2} and $F(X)$ in \eqref{exprUF}.
\end{lemma}
\begin{proof}
	By \eqref{firsth} and \eqref{FXh}, we have that
	$$
	\begin{array}{lcl}
	F(X_h(v))
	&=&-2\sqrt{\dfrac{h}{2+h}}\dfrac{1}{\sinh(v\sqrt{h}/2)}\left(\dfrac{1}{1+\frac{h}{2+h}\frac{1}{\sinh^2(v\sqrt{h}/2)}}\right).
	\end{array}
	$$
	Then, Lemma \ref{sinh} implies
	$$
	\begin{array}{lcl}
	|F(X_h(v))|&\leq&M \sqrt{h}\dfrac{1}{\sqrt{h}|v|}\left(\dfrac{1}{\left|1+\frac{h}{2+h}\frac{1}{\sinh^2(v\sqrt{h}/2)}\right|}\right).
	\end{array}
	$$
	Notice that
	$$
	\begin{array}{lcl}
	\left|1+\dfrac{h}{2+h}\dfrac{1}{\sinh^2(v\sqrt{h}/2)}\right|&\geq & 1- \dfrac{h}{2+h}\left|\dfrac{1}{\sinh^2(v\sqrt{h}/2)}\right|
	\end{array}
	$$	
	and, by Lemma \ref{sinh},
	$$
	\frac{h}{2+h}\left|\dfrac{1}{\sinh^2(v\sqrt{h}/2)}\right| \leq \frac{h}{2+h}\dfrac{1}{(M^*)^2h|v|^2}\leq \frac{1}{2(M^*)^2|v|^2}.
	$$	
	Thus, for $|v|\geq (M^*)^{-1}$, 
	\begin{equation}\label{denom}
	\begin{array}{lcl}
	\left|1+\dfrac{h}{2+h}\dfrac{1}{\sinh^2(v\sqrt{h}/2)}\right|&\geq & 1/2.
	\end{array}
	\end{equation}
	We also know that, if $|v|\geq(M^*)^{-1}$,  $|\sqrt{v^2+2}|\leq\sqrt{1+2M^*} |v|$. Hence
	$$|(\sqrt{v^2+2})F(X_h(v))|\leq M\dfrac{|\sqrt{v^2+2}|}{|v|}\leq M.$$
	Now, assume that $|v|\leq (M^*)^{-1}$. Hence $|v\sqrt{h}/2|\leq M$ and expanding $\sinh(z)$ at $0$ we obtain
		$$
	\begin{array}{lcl}
	F(X_h(v))&=&-2\dfrac{\sqrt{\frac{2+h}{h}}\left(v\sqrt{h}/2+\er(h^{3/2}v^3)\right)}{1+\frac{2+h}{h}\left(hv^2/4+\er(h^2v^4)\right)}\vspace{0.2cm}\\
	&=&-2\dfrac{\sqrt{2+h}(v/2+\er(h))}{1 +v^2/2+\er(h)}.
	\end{array}
	$$
	Since $v\in D^u$, we have that there exists $M>0$ such that
	$$|1+v^2/2+\er(h)|\geq |1+v^2/2|-\er(h)\geq M-\er(h).$$ 
	Therefore, for $h>0$ sufficiently small, we have that $|F(X_h(v))|\leq M$, for $|v|\leq (M^*)^{-1}$, and since $|\sqrt{v^2+2}|$ is inferiorly and superiorly bounded by nonzero constants in this domain, we have that
	$$|(\sqrt{v^2+2})F(X_h(v))|\leq M \quad \text{for }\ |v|\leq (M^*)^{-1}.$$
	This concludes the proof of the first item. One can obtain item $2$ using Lemma \ref{cauchy}.

\end{proof}

\begin{lemma}\label{zh}
	Given $0<h_0\leq 1$, there exists a constant $M>0$ such that, for $v\in D^u$ and $0<h\leq h_0$,
	$$\left|\dfrac{1}{Z_h^2(v)}\dfrac{1}{v^2+2}\right|\leq M$$
	where $Z_h$ in \eqref{par2}.
\end{lemma}
The proof is analogous to the one of Lemma \ref{Fs}.

\subsection{The Fixed Point Theorem}

Now, we study the operator $\mathcal{G}_{\omega,h}$ in order to find a fixed point in $\mathcal{X}_2^2$.
Recall the definition of $\mathcal{G}_{\omega,h}=\mathcal{G}_{\omega}\circ\mathcal{F}_h$ in \eqref{fixh}, and notice that $\mathcal{G}_{\omega}$ is the same operator of the case $h=0$. Thus, Proposition \ref{gomega0} still holds for functions in the Banach space $\mathcal{X}_2^2$. 

\begin{prop}\label{gomegah}
	Given $(f,g)\in \mathcal{X}_2^2$, we have that $\mathcal{G}_{\omega}(f,g)\in \mathcal{X}_2^2$. Furthermore, there exists a constant $M>0$ independent of $\e$ such that
	$$\left\|\mathcal{G}_{\omega}(f,g)\right\|_2\leq \dfrac{M}{\omega}\left\|(f,g)\right\|_2.$$
\end{prop}

We proceed by studying the operator $\mathcal{F}_h$ in \eqref{Fh}.

\begin{prop}\label{firstiterationh}
	There exists $h_0>0$, $\e_0>0$ and a constant $M>0$ such that for, $0<\e\leq \e_0$ and $0<h\leq h_0$,
	$$\left\|\mathcal{G}_{\omega,h}(0,0)\right\|_2\leq M\dfrac{\dg}{\omega^2}.$$	
\end{prop}
\begin{proof}
	Notice that $\mathcal{F}_h(0,0)=(- (Q^h)'(v),(Q^h)'(v))$ (see \eqref{firsth}), which implies
	%
	$$\left\|\mathcal{F}_{h}(0,0)\right\|_2= 2 \dfrac{\dg}{\omega\sqrt{2\Omega}}\|F(X_h)'\|_{2}.$$
	
	Thus, it is enough to apply Lemma \ref{Fs} and Proposition \ref{gomegah}.
\end{proof}

\begin{prop}\label{lipschitzh}
	There exist $\e_0>0$, $h_0>0$ and a constant $M>0$  such that for $0<\e\leq \e_0$, $0<h\leq h_0$:
	
	Let $\eta_h$ be given in \eqref{alphah1d} and take $(\gamma_j,\theta_j)\in \mathcal{B}_0(R)\subset \mathcal{X}_{2}^2$ with $j=1,2$ and $R=K\dfrac{\dg}{\omega^2}$, where $K$ is a constant independent of $h$ and $\e$, the following statements hold.
	\begin{enumerate}
		\item $\left|\eta_h(v,\gamma_j(v),\theta_j(v))-1\right|\leq M\dfrac{\dg^2}{\omega};$\vspace{0.2cm}
		\item $\left|\eta_h(v,\gamma_1(v),\theta_1(v))-\eta_h(v,\gamma_2(v),\theta_2(v))\right|\leq M\dfrac{\dg}{\omega}\|(\gamma_1,\theta_1)-(\gamma_2,\theta_2)\|_0;$\vspace{0.2cm}		
		
		\item $\left\|\mathcal{F}_h(\gamma_1,\theta_1)-\mathcal{F}_h(\gamma_2,\theta_2)\right\|_2\leq M\dg^2\|(\gamma_1,\theta_1)-(\gamma_2,\theta_2)\|_2;$		\vspace{0.2cm}		
	\end{enumerate}
\end{prop}
\begin{proof}
	%
	
	Lemmas \ref{Fs} and \ref{zh} and the fact that $(\gamma,\theta)\in \mathcal{B}_0(R)$ imply
	$$
	\begin{array}{lcl}
	\left|\dfrac{4\dg^2}{\Omega\omega}\left(\dfrac{F(X_h(v))}{Z_h(v)}\right)^2-8\omega\dfrac{\gamma\theta}{(Z_h(v))^2}\right|
	&\leq& M\dfrac{\dg^2}{\omega}.
	\end{array}
	$$
	
	
	Thus, using \eqref{alphah1d}, it follows that
	
	$$\left|\eta_h(v,\gamma,\theta)-1\right|\leq M \left|\dfrac{4\dg^2}{\Omega\omega}\left(\dfrac{F(X_h(v))}{Z_h(v)}\right)^2-8\omega\dfrac{\gamma\theta}{(Z_h(v))^2}\right|\leq M\dfrac{\dg^2}{\omega}$$
and using also Lemma \ref{zh}, we have
	$$\begin{array}{lcl}
	\left|\eta_h(v,\gamma_1,\theta_1)-\eta_h(v,\gamma_2,\theta_2)\right|&\leq& M\omega \left|\dfrac{\gamma_1\theta_1-\gamma_2\theta_2}{(Z_h(v))^2}\right|\vspace{0.2cm}\\
	&\leq&MR\omega \left(\dfrac{|\theta_1-\theta_2|}{|(Z_h(v))^2(v^2+2)|}+ \dfrac{|\gamma_1-\gamma_2|}{|(Z_h(v))^2(v^2+2)|}\right)\vspace{0.2cm}\\
	&\leq&M\dfrac{\dg}{\omega}\|(\gamma_1,\theta_1)-(\gamma_2,\theta_2)\|_0
	\end{array}$$
	
	Finally, it follows from items $(1)$ and $(2)$ of this proposition and \eqref{Fh} that
	$$
	\begin{array}{lcl}
	\left\|\pi_1\circ\mathcal{F}_h(\gamma_1,\theta_1)-\pi_1\circ\mathcal{F}_h(\gamma_2,\theta_2)\right\|_2
	&\leq& \omega\left\|\eta_h(v,\gamma_1,\theta_1)-1\right\|_0\|\gamma_1-\gamma_2\|_2\vspace{0.2cm}\\
	&&+\omega\|\gamma_2\|_2\left\|\eta_h(v,\gamma_1,\theta_1)-\eta_h(v,\gamma_2,\theta_2)\right\|_0\vspace{0.2cm}\\
	&\leq& M\dg^2\|\gamma_1-\gamma_2\|_2+ M\omega R\dfrac{\dg}{\omega}\|(\gamma_1,\theta_1)-(\gamma_2,\theta_2)\|_0\vspace{0.2cm}\\
	&\leq& M\dg^2\|(\gamma_1,\theta_1)-(\gamma_2,\theta_2)\|_2.
	\end{array}
	$$
	
	Analogously, we obtain the same inequality for the second component of $\mathcal{F}_h$. 
\end{proof}

Finally, we are able to prove Proposition \ref{solutionh} (and thus Theorem \eqref{parameterization1Dh}) by a fixed point argument.

\begin{prop}\label{fixpointh}
	There exist $\e_0>0$, $h_0>0$ and a constant $M>0$ such that for $0<h\leq h_0$ and $\e\leq \e_0$, the operator $\mathcal{G}_{\omega,h}$ (given in \eqref{fixh}) has a  fixed point $(\gamma_{h,0}^{u},\theta_{h,0}^{u})$ in $\mathcal{X}_2^2$ which satisfies
	$$\|(\gamma_{h,0}^{u},\theta_{h,0}^{u})\|_2\leq M\dfrac{\dg}{\omega^2}.$$
\end{prop}
\begin{proof}
	From Proposition \ref{firstiterationh}, there exists a constant $b_3>0$	independent of $h$ and $\e$ such that
	$$\left\|\mathcal{G}_{\omega,h}(0,0)\right\|_2\leq \dfrac{b_3}{2}\dfrac{\dg}{\omega^2},$$
	Now, given $(\gamma_1,\theta_1)$ and $(\gamma_2,\theta_2)$ in $\mathcal{B}_{0}(b_3\dg/\omega^2)$, we can use Propositions \ref{lipschitzh} (with $K=b_3$) and \ref{gomegah} and the linearity of the operator $\mathcal{G}_{\omega}$ to see that
	$$
	\begin{array}{lcl}
	\left\|\mathcal{G}_{\omega,h}(\gamma_1,\theta_1)-\mathcal{G}_{\omega,h}(\gamma_2,\theta_2)\right\|_2&\leq& \dfrac{M}{\omega}\left\|\mathcal{F}_h(\gamma_1,\theta_1)-\mathcal{F}_h(\gamma_2,\theta_2)\right\|_2\vspace{0.2cm}\\
	&\leq& M\dfrac{\dg^2}{\omega}\|(\gamma_1,\theta_1)-(\gamma_2,\theta_2)\|_2.
	\end{array}
	$$
	Choosing $\e_0$ sufficiently small, we have that $\mathrm{Lip}(\mathcal{G}_{\omega,h})\leq 1/2$. Therefore $\mathcal{G}_{\omega,h}$ sends the ball $\mathcal{B}_{0}(b_3\dg/\omega^2)$ into itself and it is a contraction. Thus, it has a unique fixed point $(\gamma_{h,0}^{u},\theta_{h,0}^{u})\in\mathcal{B}_{0}(b_3\dg/\omega^2)$. 
	
\end{proof}


\section{Proof of Theorem \ref{parameterization2Dk1k2}}\label{parf_sec}


In this section we prove the existence of $W_{\e}^u(\Lambda_{\kappa_1,\kappa_2}^{-})$, with $\dg\neq 0$.
As in the previous sections, we look for parameterizations $N_{\kappa_1,\kappa_2}^{u}$ of $W^{u}_\e(\Lambda_{\kappa_1,\kappa_2}^{-})$ as graphs 
\begin{equation}\label{paraf}
N_{\kappa_1,\kappa_2}^{u,s}(v,\tau)=(X_{\kappa_1}(v),Z_{\kappa_1}(v)+Z_{\kappa_1,\kappa_2}^{u,s}(v,\tau),\Gamma_{\kappa_2}(\tau)+\Gamma_{\kappa_1,\kappa_2}^{u,s}(v,\tau),\Theta_{\kappa_2}(\tau)+\Theta_{\kappa_1,\kappa_2}^{u,s}(v,\tau)),
\end{equation}
where $X_{\kappa_1},Z_{\kappa_1}$ are given in \eqref{par1} and $\Gamma_{\kappa_2},\Theta_{\kappa_2}$ are given in \eqref{periodic}.

Following  the same lines of Section \ref{parh_sec} we have a characterization of  $N_{\kappa_1,\kappa_2}^{u}$.
\begin{lemma}\label{formabuenahf}
	Write $Z^u_{\kappa_1,\kappa_2}(v,\tau)=Z_{\kappa_1,\kappa_2}(v,\tau)+ z_{\kappa_1,\kappa_2}^u(v,\tau)$, $\Gamma^u_{\kappa_1,\kappa_2}(v,\tau)=Q^{\kappa_2}(v)+ \gamma_{\kappa_1,\kappa_2}^u(v,\tau)$, $\Theta^u_{\kappa_1,\kappa_2}(v,\tau)=-Q^{\kappa_1}(v)+ \theta_{\kappa_1,\kappa_2}^u(v,\tau)$, where $Q^{\kappa_1}$ is given by \eqref{firsth} and
	\begin{equation}\label{Zfirstf}
	Z_{\kappa_1,\kappa_2}(v,\tau)=\dfrac{\dg}{\omega\sqrt{2\Omega}}F'(X_{\kappa_1}(v))\dfrac{\Gamma_{\kappa_2}(\tau)+\Theta_{\kappa_2}(\tau)}{2},
	\end{equation}
	with $\Gamma_{\kappa_1},\Theta_{\kappa_1}$ given by \eqref{periodic}.
	Then, $N_{\kappa_1,\kappa_2}^u(v,\tau)$, given in \eqref{paraf},  with $\kappa_1,\kappa_2\geq 0$ and $\kappa_1+\kappa_2=h$, parameterizes $W^u(\Lambda^-_{\kappa_1,\kappa_2})$ provided $(z_{\kappa_1,\kappa_2}^u,\gamma_{\kappa_1,\kappa_2}^u,\theta_{\kappa_1,\kappa_2}^u)$ satisfy
	\begin{equation}\label{edp restadaf}
	\left\{
	\begin{array}{l}\begin{array}{rcl}
	\partial_v z + \omega\partial_\tau z+\dfrac{Z_{\kappa_1}'(v)}{Z_{\kappa_1}(v)}z&=&f_1^{\kappa_1,\kappa_2}(v,\tau) -\dfrac{z+Z_{\kappa_1,\kappa_2}(v,\tau)}{Z_{\kappa_1}(v)}\partial_v z-\dfrac{\partial_vZ_{\kappa_1,\kappa_2}(v,\tau)}{Z_{\kappa_1}(v)}z\vspace{0.3cm}\\ &&-\dfrac{\delta}{\sqrt{2\Omega}} F'(X_{\kappa_1}(v)) \dfrac{\gamma-\theta}{2i},\vspace{0.3cm}\\
	\partial_v \gamma + \omega\partial_\tau \gamma-\omega i\gamma&=& f_2^{\kappa_1,\kappa_2}(v,\tau)-\dfrac{(Q^{\kappa_1})'(v)}{Z_{\kappa_1}(v)}z-\dfrac{z+Z_{\kappa_1,\kappa_2}(v,\tau)}{Z_{\kappa_1}(v)}\partial_v\gamma,\vspace{0.3cm}\\
	\partial_v \theta + \omega\partial_\tau \theta+\omega i\theta
	&=& -f_2^{\kappa_1,\kappa_2}(v,\tau)+\dfrac{(Q^{\kappa_1})'(v)}{Z_{\kappa_1}(v)}z-\dfrac{z+Z_{\kappa_1,\kappa_2}(v,\tau)}{Z_{\kappa_1}(v)}\partial_v\theta,
	\end{array}\vspace{0.3cm}\\
	\displaystyle\lim_{v\rightarrow -\infty}z(v,\tau)=\displaystyle\lim_{v\rightarrow -\infty}\gamma(v,\tau)=\displaystyle\lim_{v\rightarrow -\infty}\theta(v,\tau)=0,
	\end{array}\right.
	\end{equation}	
	where 
	\begin{align}
	f_1^{\kappa_1,\kappa_2}(v,\tau)=&-\partial_vZ_{\kappa_1,\kappa_2}(v,\tau)-\dfrac{Z_{\kappa_1}'(v)}{Z_{\kappa_1}(v)}Z_{\kappa_1,\kappa_2}(v,\tau)-\dfrac{\dg}{\sqrt{2\Omega}}F'(X_{\kappa_1}(v))\dfrac{Q^{\kappa_1}(v)}{i}	\label{f1f}\\
	 &-\dfrac{Z_{\kappa_1,\kappa_2}(v,\tau)\partial_vZ_{\kappa_1,\kappa_2}(v,\tau)}{Z_{\kappa_1}(v)},\notag\\
	f_2^{\kappa_1,\kappa_2}(v,\tau)=&-(Q^{\kappa_1})'(v)-\dfrac{Z_{\kappa_1,\kappa_2}(v,\tau)(Q^{\kappa_1})'(v)}{Z_{\kappa_1}(v)}.	\label{f2f}
	\end{align}
\end{lemma}

We consider the equation \eqref{edp restadaf} with $(v,\tau)\in D^u\times\mathbb{T}_\sigma$ with the asymptotic conditions
$	\displaystyle\lim_{\Rp(v)\rightarrow -\infty}z(v)=\displaystyle\lim_{\Rp(v)\rightarrow -\infty}\gamma(v)=\displaystyle\lim_{\Rp(v)\rightarrow -\infty}\theta(v)=0,$
for every $\tau\in\mathbb{T}_{\sigma}$.

Theorem \eqref{parameterization2Dk1k2} is a consequence of the following proposition.

\begin{prop}\label{existencef}
	Fix $\sigma>0$. There exist $h_0>0$ and $\e_0>0$ sufficiently small 
	such that for $0<\e\leq \e_0$, $0< h\leq h_0$ and  $\kappa_1,\kappa_2\geq 0$ with $\kappa_1+\kappa_2=h$,  system \eqref{edp restadaf} has an analytic solution $(z_{\kappa_1,\kappa_2}^{u},\gamma_{\kappa_1,\kappa_2}^{u},\theta_{\kappa_1,\kappa_2}^{u})$ defined in $D^{u}\times \mathbb{T}_\sigma$ (see \eqref{outerdomain} and \eqref{tsig}) such that $z_{\kappa_1,\kappa_2}^{u}$ is real-analytic, $\theta_{\kappa_1,\kappa_2}^{u}(v,\tau)=\overline{\gamma_{\kappa_1,\kappa_2}^{u}(v,\tau)}$ for each $(v,\tau)\in D^u\times \mathbb{T}_\sigma\cap\rn{2}$ and
	$$\displaystyle\lim_{\Rp(v)\rightarrow-\infty}z_{\kappa_1,\kappa_2}^{u}(v,\tau)=\displaystyle\lim_{\Rp(v)\rightarrow-\infty}\gamma_{\kappa_1,\kappa_2}^{u}(v,\tau)=\displaystyle\lim_{\Rp(v)\rightarrow-\infty}\theta_{\kappa_1,\kappa_2}^{u}(v,\tau)=0,$$
	for every $\tau\in\mathbb{T}_{\sigma}$. Furthermore, $(z_{\kappa_1,\kappa_2}^u,\gamma_{\kappa_1,\kappa_2}^u,\theta_{\kappa_1,\kappa_2}^u)$ satisfies the bounds in \eqref{bounds311}.
\end{prop}

Equation \eqref{edp restadaf} can be written as the functional equation
\begin{equation}\label{funcformf}
\mathcal{L}_{\omega,\kappa_1}(z,\gamma,\theta)= \mathcal{P}_{\kappa_1,\kappa_2}(z,\gamma,\theta),
\end{equation}
where $\mathcal{L}_{\omega,\kappa_1}$ and $\mathcal{P}_{\kappa_1,\kappa_2}$ are the functional operators given by
\begin{equation}\label{Lof}
\mathcal{L}_{\omega,\kappa_1}(z,\gamma,\theta)=\left(\begin{array}{l}
\partial_v z + \omega\partial_\tau z+\dfrac{Z_{\kappa_1}'(v)}{Z_{\kappa_1}(v)}z\vspace{0.2cm}\\	
\partial_v \gamma + \omega\partial_\tau \gamma-\omega i\gamma \vspace{0.2cm}\\
\partial_v \theta + \omega\partial_\tau \theta+\omega i\theta
\end{array}\right),
\end{equation}
and
\begin{equation}\label{Fof}
\mathcal{P}_{\kappa_1,\kappa_2}(z,\gamma,\theta)=\left(\begin{array}{l}
f_1^{\kappa_1,\kappa_2}(v,\tau) -\dfrac{z+Z_{\kappa_1,\kappa_2}(v,\tau)}{Z_{\kappa_1}(v)}\partial_v z-\dfrac{\partial_vZ_{\kappa_1,\kappa_2}}{Z_{\kappa_1}(v)}z-\dfrac{\delta}{\sqrt{2\Omega}} F'(X_{\kappa_1}(v)) \dfrac{\gamma-\theta}{2i} \vspace{0.2cm}\\
f_2^{\kappa_1,\kappa_2}(v,\tau)-\dfrac{(Q^{\kappa_1})'(v)}{Z_{\kappa_1}(v)}z-\dfrac{z+Z_{\kappa_1,\kappa_2}(v,\tau)}{Z_{\kappa_1}(v)}\partial_v\gamma\vspace{0.2cm}\\
-f_2^{\kappa_1,\kappa_2}(v,\tau)+\dfrac{(Q^{\kappa_1})'(v)}{Z_{\kappa_1}(v)}z-\dfrac{z+Z_{\kappa_1,\kappa_2}(v,\tau)}{Z_{\kappa_1}(v)}\partial_v\theta
\end{array}\right).
\end{equation}
We show the existence of an inverse $\mathcal{G}_{\omega}^{\kappa_1}$ of $\mathcal{L}_{\omega,\kappa_1}$ in the  Banach spaces $\mathcal{X}_{\ag,\sigma}^{3}$ and $\mathcal{Y}_{\ag,\sigma}^{3}$ introduced in Section \ref{banachper}.

Given analytic functions $f$, $g$, and $h$ defined in $D^u\times\mathbb{T}_{\sigma}$, consider
\begin{equation}\label{inteqf}
\begin{array}{l}
F^{[k]}_{\kappa_1}(f)(v)= \displaystyle\int_{-\infty}^v\dfrac{e^{\omega i k(r-v)}Z_{\kappa_1}(r)}{Z_{\kappa_1}(v)}f^{[k]}(r)dr,
\end{array}
\end{equation}
and $G^{[k]}(g)$, $H^{[k]}(h)$ given in \eqref{inteq}. Then, we define the linear operator $\mathcal{G}_{\omega}^{\kappa_1}$
\begin{equation}\label{Gof}
\mathcal{G}_{\omega}^{\kappa_1}(f,g,h)=\left(\begin{array}{l}
\displaystyle\sum_k F^{[k]}_{\kappa_1}(f)(v) e^{ik\tau}\vspace{0.2cm}\\
\displaystyle\sum_k G^{[k]}(g)(v) e^{ik\tau}\vspace{0.2cm}\\
\displaystyle\sum_k H^{[k]}(h)(v) e^{ik\tau}
\end{array}\right).
\end{equation}
%
%
\begin{lemma}\label{opgoh}
	Fix $\alpha\geq 1$ and $\sigma>0$. There exists $\kappa_1^0>0$ sufficiently small, such that, for  $0<\e\leq \e_0$ and $0<\kappa_1\leq \kappa_1^0$, the operator $$\mathcal{G}_{\omega}^{\kappa_1}:
	\mathcal{X}_{\alpha+1,\sigma}^3\rightarrow \mathcal{Y}_{\alpha,\sigma}^3$$ is well-defined and satisfies:
	\begin{enumerate}
		\item $\mathcal{G}_{\omega}^{\kappa_1}$ is an inverse of the operator $\mathcal{L}_{\omega,\kappa_1}:\mathcal{Y}_{\ag,\sigma}^{3}\rightarrow \mathcal{X}_{\alpha+1,\sigma}^3$, i.e. $\mathcal{G}_{\omega}^{\kappa_1}\circ\mathcal{L}_{\omega,\kappa_1}=\mathcal{L}_{\omega,\kappa_1}\circ\mathcal{G}_{\omega}^{\kappa_1}= \mathrm{Id}$;
		\item $\llbracket\mathcal{G}_{\omega}^{\kappa_1}(f,g,h)\rrbracket_{\ag,\sigma}\leq M \|(f,g,h)\|_{\ag+1,\sigma}$;
		\item If $f^{[0]}=g^{[1]}=h^{[-1]}=0$, then $\llbracket \mathcal{G}_{\omega}^{\kappa_1}(f,g,h)\rrbracket_{\ag,\sigma}\leq \dfrac{M}{\omega} \llbracket (f,g,h)\rrbracket_{\ag,\sigma}$,
	\end{enumerate}
	where $M$ is a constant independent of $\kappa_1$ and $\e$.
	
	%
	%
	%
\end{lemma}
The proof of the following lemma is analogous to that in Lemma \ref{Fcompacto} below.
\begin{lemma}\label{techf}
	Let $F$, $X_{\kappa_1},Z_{\kappa_1}$ be given by \eqref{exprUF} and \eqref{par2}. There exist $\kappa_1^0>0$ and a constant $M>0$  such that, for  $v\in D^u$ and $0<\kappa_1\leq \kappa^0_1$,
	\begin{enumerate}
		\item $|F(X_{\kappa_1}(v))''|\leq \dfrac{M}{|v^2+2|^{3/2}}$;\vspace{0.2cm}
		\item $\left|\dfrac{Z_{\kappa_1}'(v)}{Z_{\kappa_1}(v)}\right|\leq \dfrac{M}{|\sqrt{v^2+2}|}$. 		
	\end{enumerate}
\end{lemma}
\begin{lemma}\label{lipf}
	Fix $\sigma>0$ and $K>0$. There exist $\e_0>0$ and $h_0>0$ sufficiently small such that, for $0<\e<\e_0$,  $0\leq h\leq h_0$ and $\kappa_1,\kappa_2\geq 0$ with $\kappa_1+\kappa_2=h$, the operator 
	$\mathcal{P}_{\kappa_1,\kappa_2}: \mathcal{Y}_{1,\sigma}^{3}\rightarrow\mathcal{X}_{2,\sigma}^{3},$
	 is well defined and there exists a constant $M>0$ such that
	$$\|\mathcal{P}_{\kappa_1,\kappa_2}(0,0,0)\|_{2,\sigma}\leq M\dfrac{\dg}{\omega}.$$
	Moreover, given $(z_j,\gamma_j,\theta_j)\in\mathcal{B}_{0}(K\dg/\omega)\subset\mathcal{Y}_{1,\sigma}^{3},\ j=1,2$, 
	\begin{equation}
	\left\|	\mathcal{P}_{\kappa_1,\kappa_2}(z_1,\gamma_1,\theta_1)- \mathcal{P}_{\kappa_1,\kappa_2}(z_2,\gamma_2,\theta_2)\right\|_{2,\sigma}\leq M\left(\dg+\dfrac{\dg}{\omega^{3/2}}\sqrt{h}\right)\left\llbracket	(z_1,\gamma_1,\theta_1)- (z_2,\gamma_2,\theta_2)\right\rrbracket_{1,\sigma}.
	\end{equation}
\end{lemma}
\begin{proof}
	Recall that $\mathcal{P}_{\kappa_1,\kappa_2}(0,0,0)=(f_1^{\kappa_1,\kappa_2},f_2^{\kappa_1,\kappa_2},-f_2^{\kappa_1,\kappa_2})$, where $f_1^{\kappa_1,\kappa_2}, f_2^{\kappa_1,\kappa_2}$ are given in \eqref{f1f} and \eqref{f2f}, respectively, and involve the functions $F'(X_{\kappa_1}),$ $\ Z_{\kappa_1}'/Z_{\kappa_1},\ Q^{\kappa_1},\ (Q^{\kappa_1})',$ $Z_{\kappa_1,\kappa_2},\partial_v Z_{\kappa_1,\kappa_2}$ which can be computed using the expressions in \eqref{exprUF}, \eqref{par2}, \eqref{first0}, and \eqref{Zfirst}. By Lemmas \ref{Fs}, \ref{zh} and \ref{techf}, we have
	
%
%
	
	$$
	\begin{array}{l}
	\|Q^{\kappa_1}\|_{1,\sigma}, \|(Q^{\kappa_1})'\|_{2,\sigma} \leq M\dfrac{\dg}{\omega},\vspace{0.2cm}\\	
	\left\|Z_{\kappa_1,\kappa_2}\right\|_{1,\sigma},\|\partial_vZ_{\kappa_1,\kappa_2}\|_{2,\sigma}\leq M\dfrac{\dg\sqrt{\kappa_2}}{\omega^{3/2}},\vspace{0.2cm}\\ \|Z_{\kappa_1}'/Z_{\kappa_1}\|_{1,\sigma},  \|F'(X_{\kappa_1})\|_{1,\sigma} \leq M.
	\end{array}
	$$
	Therefore, using also Lemma \ref{properties}, one has
	\[
	\begin{split}
	\|f_1^{\kappa_1,\kappa_2}\|_{2,\sigma}&\leq M\max\left\{\dfrac{\dg\sqrt{\kappa_2}}{\omega^{3/2}},\dfrac{\dg^2}{\omega},\dfrac{\dg^2}{\omega^3}\kappa_2\right\}=M\max\left\{\dfrac{\dg\sqrt{\kappa_2}}{\omega^{3/2}},\dfrac{\dg^2}{\omega}\right\},\\
	\|f_2^{\kappa_1,\kappa_2}\|_{2,\sigma}&\leq M\max\left\{\dfrac{\dg}{\omega},\dfrac{\dg^2}{\omega^{5/2}}\sqrt{\kappa_2}\right\}=M\dfrac{\dg}{\omega}.
	\end{split}
	\]
		Thus, $\|\mathcal{P}_{\kappa_1,\kappa_2}(0,0,0)\|_{2,\sigma}\leq M\dg/\omega$.
	
	Following the lines of the proof of Lemma \ref{lip}  one can complete the proof of Lemma \ref{lipf}.
\end{proof}

Now, we write Proposition \ref{existencef} in terms of Banach spaces. Then, it can be proved in the same way as Proposition  \ref{fixpointgob} by considering the operator $\overline{\mathcal{G}}_{\omega,\kappa_1,\kappa_2}=\mathcal{G}_{\omega}^{\kappa_1}\circ \mathcal{P}_{\kappa_1,\kappa_2}$.

\begin{prop}\label{fixpointgobf}
	Fix $\sigma>0$. There exist $h_0>0$ and $\e_0>0$ such that, for $0<\e\leq \e_0$,  $0<h\leq h_0$ and $\kappa_1,\kappa_2\geq 0$ with $\kappa_1+\kappa_2=h$, the operator $\overline{\mathcal{G}}_{\omega,\kappa_1,\kappa_2}=\mathcal{G}_{\omega}^{\kappa_1}\circ \mathcal{P}_{\kappa_1,\kappa_2}$, with $\mathcal{G}_{\omega}^{\kappa_1}$ and $\mathcal{P}_{\kappa_1,\kappa_2}$ given in \eqref{Gof} and \eqref{Fof}, respectively, has a  fixed point $(z_{\kappa_1,\kappa_2}^{u},\gamma^{u}_{\kappa_1,\kappa_2},\theta_{\kappa_1,\kappa_2}^u)\in\mathcal{Y}_{1,\sigma}^3$. Furthermore, there exists a constant $M>0$ independent of $\e$, $\kappa_1$ and $\kappa_2$ such that
	$$\llbracket(z_{\kappa_1,\kappa_2}^{u},\gamma^{u}_{\kappa_1,\kappa_2},\theta_{\kappa_1,\kappa_2}^u)\rrbracket_{1,\sigma}\leq M\dfrac{\dg}{\omega}.$$
\end{prop}

This completes the proof of Theorem \ref{parameterization2Dk1k2}.

\section{Proof of Theorem \ref{approx}}

We compare the parameterizations of $W^{u}_{\e}(\Lambda_{\kappa_1,\kappa_2}^{-})$ obtained in Sections \ref{parhper_sec}, \ref{parh_sec} and \ref{parf_sec}, respectively, with the parameterization \eqref{N00} of $W^{u}_{\e}(p_0^{-})$ obtained in Section \ref{par0_sec}. 

%
%
%
%

\subsection{Approximation of $W^{u}_{\e}(\Lambda_h^{-})$ by $W_{\e}^{u}(p_0^{-})$ }
\label{approx1}


We compare the parameterizations $N_{0,h}^{u}$ and $N_{0,0}^{u}$ of $W_{\e}^{u}(\Lambda_h^{-})$ and $W_{\e}^{u}(p_0^{-})$, obtained in Theorems \ref{parameterization2Dh} and \ref{parameterization1D0}, respectively. 
\begin{prop}\label{resta}
	Let $\Gamma_0^u(v),$ $\Theta_0^u(v)$ and $\Gamma_{0,h}^u(v,\tau),$ $\Theta_{0,h}^u(v,\tau)$ be given in \eqref{eq1d} and \eqref{eq2d}, respectively. Given $h_0>0$, 
	there exists $\e_0>0$ and a constant $M>0$,  such that 	for  $v\in D^{u}\cap\R$, $\tau\in\mathbb{T}$, $0\leq\e\leq\e_0$ and $0\leq h\leq h_0$,
	\begin{equation}
	\begin{array}{l}
	\left|\partial_\tau(\Gamma_{0,h}^{u}(v,\tau)-\Gamma_{0}^{u}(v))\right|,\left|\Gamma_{0,h}^{u}(v,\tau)-\Gamma_{0}^{u}(v)\right|\leq M\dfrac{\dg\sqrt{h}}{\omega^{3/2}},\vspace{0.2cm}\\
	\left|\partial_\tau(\Theta_{0,h}^{u}(v,\tau)-\Theta_{0}^{u}(v))\right|,\left|\Theta_{0,h}^{u}(v,\tau)-\Theta_{0}^{u}(v)\right|\leq M\dfrac{\dg\sqrt{h}}{\omega^{3/2}}.\\	
	\end{array}
	\end{equation}
\end{prop}
\begin{proof}

	
	Considering $h=0$ in Theorem \ref{parameterization2Dh}, it follows that $N_{0,0}^{u}(v,\tau)$ is also a parameterization of $W^{u}_{\e}(p_0^-)$. Since $W^{u}_{\e}(p_0^-)$  is parameterized by both $N_{0,0}^{u}(v)$ (from Theorem \ref{parameterization1D0}) and $N_{0,0}^u(v,\tau)$ (from Theorem \ref{parameterization2Dh}) and both have the same first coordinate,  these parameterizations coincide. Therefore $\gamma_{0,0}^u$ and $\theta_{0,0}^u$ given in Theorem \ref{parameterization2Dh} with $h=0$ depend only on the variable $v$ and we can write
	$$\begin{array}{lcl}
	\Gamma_0^u(v)= Q^0(v)+\gamma_{0,0}^u(v),\vspace{0.2cm}\\
	\Theta_0^u(v)=-Q^0(v)+\theta_{0,0}^u(v).
	\end{array}$$

	
	Based on these arguments, we can use Theorem \ref{parameterization2Dh} and Proposition \ref{fixpointgob} to see that
	$$
	\begin{array}{lcl}
	\left(\begin{array}{c}
	\Gamma_{0,h}^{u}(v,\tau)-\Gamma_{0}^{u}(v)\vspace{0.2cm}\\
	\Theta_{0,h}^{u}(v,\tau)-\Theta_{0}^{u}(v)
	\end{array}\right) 
	&=& \left(\begin{array}{c}
	\gamma_{0,h}^{u}(v,\tau)-\gamma_{0,0}^{u}(v)\vspace{0.2cm}\\
	\theta_{0,h}^{u}(v,\tau)-\theta_{0,0}^{u}(v)
	\end{array}\right),
	\end{array}
	$$
	where $(z_{0,0}^u,\gamma_{0,0}^u,\theta_{0,0}^u)$ and $(z_{0,h}^u,\gamma_{0,h}^u,\theta_{0,h}^u)$ are fixed points of the operators $\overline{\mathcal{G}}_{\omega,0}$ and $\overline{\mathcal{G}}_{\omega,h}$ given in \eqref{gob}, respectively.

	Denoting $$\mathcal{E}=(z_{0,h}^{u}-z_{0,0}^{u},\gamma_{0,h}^{u}-\gamma_{0,0}^{u},\theta_{0,h}^{u}-\theta_{0,0}^{u}),$$
	we  compute $\|\mathcal{E}\|_{1,\sigma}$.
	
	Notice that
	$$
	\begin{array}{lcl}
	\mathcal{E} &=&(z_{0,h}^{u}-z_{0,0}^{u},\gamma_{0,h}^{u}-\gamma_{0,0}^{u},\theta_{0,h}^{u}-\theta_{0,0}^{u})\vspace{0.2cm}\\
	&=&\overline{\mathcal{G}}_{\omega,h}(z_{0,h}^{u},\gamma_{0,h}^{u},\theta_{0,h}^{u})-\overline{\mathcal{G}}_{\omega,h}(z_{0,0}^{u},\gamma_{0,0}^{u},\theta_{0,0}^{u})\vspace{0.2cm}\vspace{0.2cm}\\
	&&+\overline{\mathcal{G}}_{\omega,h}(z_{0,0}^{u},\gamma_{0,0}^{u},\theta_{0,0}^{u})-\overline{\mathcal{G}}_{\omega,0}(z_{0,0}^{u},\gamma_{0,0}^{u},\theta_{0,0}^{u}).
	\end{array}
	$$ 
	For  $0\leq h\leq h_0$,  $(z_{0,h}^{u},\gamma_{0,h}^{u},\theta_{0,h}^{u})\in\mathcal{B}_0(M\dg/\omega)$  and $\overline{\mathcal{G}}_{\omega,h}$ is Lipschitz in this ball with $\mathrm{Lip}(\overline{\mathcal{G}}_{\omega,h})\leq M\dg$. Then, 
	$$\llbracket\overline{\mathcal{G}}_{\omega,h}(z_{0,h}^{u},\gamma_{0,h}^{u},\theta_{0,h}^{u})-\overline{\mathcal{G}}_{\omega,h}(z_{0,0}^{u},\gamma_{0,0}^{u},\theta_{0,0}^{u})\rrbracket_{1,\sigma}\leq M\dg\llbracket
	\mathcal{E}\rrbracket_{1,\sigma}.$$
	Choosing $\e_0$ sufficiently small such that $\mathrm{Lip}(\overline{\mathcal{G}}_{\omega,h})<1/2$, we obtain
	$$\llbracket
	\mathcal{E}\rrbracket_{1,\sigma}\leq M\llbracket\overline{\mathcal{G}}_{\omega,h}(z_{0,0}^{u},\gamma_{0,0}^{u},\theta_{0,0}^{u})-\overline{\mathcal{G}}_{\omega,0}(z_{0,0}^{u},\gamma_{0,0}^{u},\theta_{0,0}^{u})\rrbracket_{1,\sigma}.$$
	Now, denoting $	\mathcal{P}_h(z_{0,0}^{u},\gamma_{0,0}^{u},\theta_{0,0}^{u})-\mathcal{P}_0(z_{0,0}^{u},\gamma_{0,0}^{u},\theta_{0,0}^{u})= \Delta_h^0$, where $\mathcal{P}_h$ is given in \eqref{Fo}, and using that $\left\llbracket(z_{0,0}^{u},\gamma_{0,0}^{u},\theta_{0,0}^{u})\right\rrbracket_{1,\sigma}\leq M\dg/\omega$, we have that $\left\|\Delta_h^0\right\|_{2,\sigma}\leq M\frac{\dg\sqrt{h}}{\omega^{3/2}}.$ 
	%
	%

	It follows from the linearity of $\mathcal{G}_{\omega}$ and Lemma \ref{opgo} that
	$$\left\llbracket\overline{\mathcal{G}}_{\omega,h}(z_{0,0}^{u},\gamma_{0,0}^{u},\theta_{0,0}^{u})-\overline{\mathcal{G}}_{\omega,0}(z_{0,0}^{u},\gamma_{0,0}^{u},\theta_{0,0}^{u})\right\rrbracket_{1,\sigma}\leq M\dfrac{\dg\sqrt{h}}{\omega^{3/2}}.$$
	Thus, we conclude that $\left\llbracket\mathcal{E}\right\rrbracket_{1,\sigma}\leq M\dfrac{\dg\sqrt{h}}{\omega^{3/2}}$.
\end{proof}

\subsection{Approximation of $W_{\e}^{u}(p_h^{-})$ by $W_{\e}^{u}(p_0^{-})$ }\label{approx2}

%

We compare the parameterizations $N_{0,0}^{u}$ and $N_{h,0}^{u}$ of $W_{\e}^{u}(p_0^{-})$ and $W_{\e}^{u}(p_h^{-})$, obtained in Theorems \ref{parameterization1D0} and \ref{parameterization1Dh}, respectively. 


\begin{prop}\label{approxh0}
	Let $\Gamma_0^u(v),$ $\Theta_0^u(v)$ and $\Gamma_{h,0}^u(v),$ $\Theta_{h,0}^u(v)$ be given in  \eqref{eq1d} and \eqref{solh0}, respectively. There exist $\e_0>0$, $h_0>0$ and a constant $M>0$ 
	such that, for  $0<\e\leq \e_0$ and $0\leq h\leq h_0$,
	\begin{enumerate}
		\item $	\left|\Gamma_{h,0}^{u}(0) - \Gamma_{0}^{u}(0)\right|\leq M\dfrac{\dg\sqrt{h}}{\omega^2};\vspace{0.2cm}$
		\item $	\left|\Theta_{h,0}^{u}(0) - \Theta_{0}^{u}(0)\right|\leq M\dfrac{\dg\sqrt{h}}{\omega^2}.$
	\end{enumerate}
\end{prop}

\subsubsection{Technical Lemmas}
To prove Proposition \ref{approxh0}, we first state some lemmas.

\begin{lemma}\label{Fcompacto}
	Let $X_0,\ Z_0,\ X_h,\ Z_h,\ Q^0, \textrm{ and }Q^h$ be given in \eqref{par1}, \eqref{par2}, \eqref{first0} and \eqref{firsth} and fix $M_0>0$. There exist $h_0>0$ and a constant $M>0$ such that, for $0\leq h\leq h_0$ and  $v\in D^u$ with $|h^{1/4}v|\leq M_0$, 
	\begin{enumerate}
		\item $\left|F(X_h(v))- F(X_0(v))\right|\leq \dfrac{M\sqrt{h}}{|\sqrt{v^2+2}|}$;\vspace{0.2cm}
		\item 	$\left|Z_h(v)- Z_0(v)\right|\leq \dfrac{M\sqrt{h}}{|\sqrt{v^2+2}|}$;\vspace{0.2cm}
		\item $\left|\dfrac{1}{Z_h(v)}- \dfrac{1}{Z_0(v)}\right|\dfrac{1}{|\sqrt{v^2+2}|}\leq M\sqrt{h}$;\vspace{0.2cm}
		\item $\left|(Q^h)'(v)- (Q^0)'(v)\right|\leq \dfrac{M\dg\sqrt{h}}{\omega|v^2+2|}.$
	\end{enumerate}
	
\end{lemma}
\begin{proof}
	Using the formulas \eqref{exprUF}, \eqref{par1} and \eqref{par2}, we obtain
	
	$F(X_h(v))-F(X_0(v))=-2\left(\dfrac{\sqrt{\frac{2+h}{h}}\sinh(v\sqrt{h}/2)}{1+\frac{2+h}{h}\sinh^2(v\sqrt{h}/2)}-\sqrt{2}\dfrac{v}{v^2+2}\right).$
	
	Since $|vh^{1/4}|\leq M_0$, it follows that $|v\sqrt{h}/2|\leq Mh^{1/4}\ll 1$.


	Expanding $\sinh(z)$ at $0$, we have
	
	$$
	\begin{array}{lcl}
	\dfrac{\sqrt{\dfrac{2+h}{h}}\sinh(v\sqrt{h}/2)}{1+\frac{2+h}{h}\sinh^2(v\sqrt{h}/2)} &=& \dfrac{\sqrt{\dfrac{2+h}{h}}\left(\dfrac{v\sqrt{h}}{2}+\er(h^{3/2}|v|^3)\right)}{1+\dfrac{2+h}{h}\left(\dfrac{v^2h}{4}+\er(h^2|v|^4)\right)}\vspace{0.2cm}\\
	&=& \dfrac{\sqrt{2}v+\er(\sqrt{h}|v|)}{v^2+2+\er(\sqrt{h}|v|^2)}\vspace{0.2cm}\\
	&=& \dfrac{\sqrt{2}v}{v^2+2}\left(1+\er(\sqrt{h})\right).\vspace{0.2cm}\\
	\end{array}
	$$
	
	Item $(1)$ follows directly from this expression, considering $h$ sufficiently small. Items $(2)$ and $(3)$ can be computed in an analogous way.
	
	Formulas \eqref{first0} and \eqref{firsth} imply $$\begin{array}{lcl}\left|(Q^h)'(v)-(Q^0)'(v)\right|
	&\leq&M\dfrac{\dg}{\omega}|F'(X_h(v))Z_h(v)-F'(X_0(v))Z_0(v)|.
	\end{array}$$
	
	Thus, it is enough to apply the bounds in items $(1)$ and $(2)$ to obtain item $(4)$.

	%
	%
	%
	%
\end{proof}

\begin{lemma}\label{alphacompacto}
	Let $\eta_0$ and $\eta_h$ be given in \eqref{alpha01d} and \eqref{alphah1d}, respectively, and consider the functions $(\gamma_{0}^{u},\theta_{0}^{u})$ obtained in Proposition \ref{fixpoint0}. Fix $M_0>0$. There exist $\e_0>0$, $h_0>0$ and a constant $M>0$ 
	such that for $0<\e\leq \e_0$, $0\leq h\leq h_0$ and  $v\in D^u$ with $|h^{1/4}v|\leq M_0$,
	$$\left|\eta_h(v, \gamma_{0}^{u},\theta_{0}^{u})- \eta_0(v, \gamma_{0}^{u},\theta_{0}^{u})\right|\leq \dfrac{M\dg\sqrt{h}}{\omega}.$$
	\end{lemma}
\begin{proof}
	
	%
	Using the expression of $\eta_h$ in \eqref{alphah1d} 
	and that $\left\|(\gamma_{0}^{u},\theta_{0}^{u})\right\|_2\leq M\dg/\omega^2\ll1$, it follows from Lemmas \ref{Fs}, \ref{zh} and \ref{Fcompacto} that
	$$
	\begin{array}{lcl}
	\left|\eta_h(v, \gamma_{0}^{u},\theta_{0}^{u})- \eta_0(v, \gamma_{0}^{u},\theta_{0}^{u})\right|&\leq& M\dfrac{\dg}{\omega}\left|\left(\dfrac{F(X_h)}{Z_h}\right)^2-\left(\dfrac{F(X_0)}{Z_0}\right)^2\right|+ M\omega\left|\gamma_{0}^{u}\theta_{0}^{u}\right|\left|\dfrac{1}{Z_h^2}-\dfrac{1}{Z_0^2}\right|\vspace{0.2cm}\\
	&\leq&\dfrac{M\dg\sqrt{h}}{\omega}.
	\end{array}
	$$	
\end{proof}

\subsubsection{Proof of Proposition \ref{approxh0}}
The domain $D^u$ defined in \eqref{outerdomain} is contained in the domain $D^u_{\e}$ defined in \eqref{domainlocal}. Therefore, the restriction of the fixed point obtained in Section \ref{par0_sec} can be seen as an element of the space $\mathcal{X}_2^2$ with the same bound.

\begin{prop}\label{gammah0}
	Consider $(\gamma_{0}^{u},\theta_{0}^{u})$ and  $(\gamma_{h,0}^{u},\theta_{h,0}^{u})$ obtained in Theorems \ref{fixpoint0}  and \ref{fixpointh}, respectively, and the operator $\mathcal{G}_{\omega,h}$ given by \eqref{fixh}. Then, there exist $\e_0>0$, $h_0>0$ and a constant $M>0$
	such that for $0\leq h\leq h_0$ and $0<\e\leq \e_0$,
	$$\left\|\mathcal{G}_{\omega,h}(\gamma_{h,0}^{u},\theta_{h,0}^{u}) -\mathcal{G}_{\omega,h}(\gamma_{0}^{u},\theta_{0}^{u})\right\|_0\leq M\dfrac{\dg^2}{\omega} \left\|(\gamma_{h,0}^{u},\theta_{h,0}^{u}) -(\gamma_{0}^{u},\theta_{0}^{u})\right\|_0.$$
	
\end{prop}
\begin{proof}
	By Proposition \ref{lipschitzh}, we have
	$$\left|\eta_h(v,\gamma_{h,0}^{u},\theta_{h,0}^{u})-\eta_h(v,\gamma_{0}^{u},\theta_{0}^{u})\right|\leq M\dfrac{\dg}{\omega}\left\|(\gamma_{h,0}^{u},\theta_{h,0}^{u})-(\gamma_{0}^{u},\theta_{0}^{u})\right\|_0.$$
	Thus, using the expression of $\mathcal{F}_h$ in \eqref{Fh} and Proposition \ref{lipschitzh}, 
	$$
	\begin{array}{lcl}
	\left\|\pi_1(\mathcal{F}_h(\gamma_{h,0}^{u},\theta_{h,0}^{u})-\mathcal{F}_h(\gamma_{0}^{u},\theta_{0}^{u}))\right\|_0 &\leq &\omega\left\| \eta_h(v,\gamma_{h,0}^{u},\theta_{h,0}^{u})-1\right\|_0\left\|\gamma_{h,0}^{u}-\gamma_{0}^{u}\right\|_0\vspace{0.2cm}\\
	&&+\omega\left\|\gamma_{0}^{u}\right\|_0\left\| \eta_h(v,\gamma_{h,0}^{u},\theta_{h,0}^{u})-\eta_h(v,\gamma_{0}^{u},\theta_{0}^{u})\right\|_0\vspace{0.2cm}\\	
	&\leq &M\dg^2\left\|\gamma_{h,0}^{u}-\gamma_{0}^{u}\right\|_0\vspace{0.2cm}\\
	&&+M\dg\left\|\gamma_{0}^{u}\right\|_2\left\|(\gamma_{h,0}^{u},\theta_{h,0}^{u})-(\gamma_{0}^{u},\theta_{0}^{u})\right\|_0\vspace{0.2cm}\\	
	&\leq &M\left(\dg^2+\dfrac{\dg^2}{\omega^2}\right)\left\|(\gamma_{h,0}^{u},\theta_{h,0}^{u})-(\gamma_{0}^{u},\theta_{0}^{u})\right\|_0.
	\end{array}
	$$
The same bound can be obtained for the second coordinate of $\mathcal{F}_h$. Thus	
	$$\left\|\mathcal{F}_h(\gamma_{h,0}^{u},\theta_{h,0}^{u})-\mathcal{F}_h(\gamma_{0}^{u},\theta_{0}^{u})\right\|_0 \leq M\dg^2\left\|(\gamma_{h,0}^{u},\theta_{h,0}^{u})-(\gamma_{0}^{u},\theta_{0}^{u})\right\|_0.$$
	Now, denote $\Delta_h^j=\pi_j\left(\mathcal{F}_h(\gamma_{h,0}^{u},\theta_{h,0}^{u})-\mathcal{F}_h(\gamma_{0}^{u},\theta_{0}^{u})\right)$, $j=1,2$, and $\Delta_h=(\Delta_h^1,\Delta_h^2)$. Then,
	$$
	\begin{array}{lcl}
	\left|\pi_1\left(\mathcal{G}_{\omega,h}(\gamma_{h,0}^{u},\theta_{h,0}^{u}) -\mathcal{G}_{\omega,h}(\gamma_{0}^{u},\theta_{0}^{u})\right)(v)\right|&=&\left|\displaystyle\int_{-\infty}^0e^{\omega is}\Delta_h^1(s+v)ds\right|.
	\end{array}
	$$	
	Since $\Delta_h\in\mathcal{X}_2^2$, we can  change the path of integration to  obtain
	$$
	\begin{array}{lcl}
	\left|\displaystyle\int_{-\infty}^0e^{\omega is}\Delta_h^1(s+v)ds\right|&=&\left|\displaystyle\int_{-\infty}^0e^{\omega ie^{-i\bg}\xi}\Delta_h^1(e^{-i\bg}\xi+v)e^{i\bg}d\xi\right|\vspace{0.2cm}\\
	&\leq&\displaystyle\int_{-\infty}^0e^{\omega\sin(\bg)\xi}|\Delta_h^1(e^{-i\bg}\xi+v)|d\xi\vspace{0.2cm}\\
	&\leq&\|\Delta_h\|_0\displaystyle\int_{-\infty}^0e^{\omega\sin(\bg)\xi}d\xi\vspace{0.2cm}\\
	&\leq&\dfrac{M}{\omega}\|\Delta_h\|_0.
	\end{array}
	$$
	The same argument holds for the second coordinate of $\mathcal{G}_{\omega,h}(\gamma_{h,0}^{u},\theta_{h,0}^{u}) -\mathcal{G}_{\omega,h}(\gamma_{0}^{u},\theta_{0}^{u})$.
\end{proof}

\begin{lemma}\label{Fhcompacto}
	Let $\mathcal{F}_0$ and $\mathcal{F}_h$ be given in \eqref{F0} and \eqref{Fh}, respectively, and consider the functions $(\gamma_{0}^{u},\theta_{0}^{u})$ obtained in Theorem \ref{fixpoint0}. Given $M_0>0$ fixed, there exist $\e_0$, $h_0>0$ and a constant $M>0$
	such that for $0\leq h\leq h_0$, $0<\e\leq \e_0$ and  $v\in D^u$ with $|h^{1/4}v|\leq M_0$, 
	$$\left|\pi_j\circ\mathcal{F}_h(\gamma_{0}^{u},\theta_{0}^{u})(v)- \pi_j\circ\mathcal{F}_0(\gamma_{0}^{u},\theta_{0}^{u})(v)\right|\leq \dfrac{M\dg\sqrt{h}}{\omega|v^2+2|},\quad  j=1,2.$$
	\end{lemma}
\begin{proof}
	Lemmas \ref{Fcompacto} and \ref{alphacompacto} imply
	$$
	\begin{array}{lcl}
	\left|\pi_1(\mathcal{F}_h(\gamma_{0}^{u},\theta_{0}^{u})(v)- \mathcal{F}_0(\gamma_{0}^{u},\theta_{0}^{u}))(v)\right| &\leq& |(Q^h)'(v)-(Q^0)'(v)|\vspace{0.2cm}\\
	&&+\omega\left|\gamma_{0}^{u}\right| \left|\eta_h(v, \gamma_{0}^{u},\theta_{0}^{u})- \eta_0(v, \gamma_{0}^{u},\theta_{0}^{u})\right|\vspace{0.2cm}\\
	&\leq& M\dfrac{\dg\sqrt{h}}{\omega|v^2+2|}.
	\end{array}
	$$
	The same holds for the second coordinate.	
\end{proof}

\begin{prop}\label{gamma00}
	Consider the functions $(\gamma_{0}^{u},\theta_{0}^{u})$ obtained in Proposition \ref{fixpoint0} and the operators $\mathcal{G}_{\omega,0}$ and $\mathcal{G}_{\omega,h}$  given in \eqref{fix0} and \eqref{fixh}, respectively. There exist $\e_0>$, $h_0>0$ and a constant $M>0$ such that, for $0<\e\leq\e_0$ and  $0<h\leq h_0$
	$$\left\|\mathcal{G}_{\omega,h}(\gamma_{0}^{u},\theta_{0}^{u}) -\mathcal{G}_{\omega,0}(\gamma_{0}^{u},\theta_{0}^{u})\right\|_0\leq \dfrac{M\dg\sqrt{h}}{\omega^2}.$$
	\end{prop}
\begin{proof}
	It follows from the proof of Proposition \ref{fixpointh} that the Lipschitz constant of $\mathcal{G}_{\omega,h}$ in a ball $\mathcal{B}_0(K\dg/\omega^2)$, for some $K>0$ fixed, satisfies $\mathrm{Lip}(\mathcal{G}_{\omega,h})\leq M\dg^2/\omega$. Moreover, $\|\mathcal{G}_{\omega,h}(0,0)\|_2\leq M\dg/\omega^2$ and $\|(\gamma_{0}^{u},\theta_{0}^{u})\|_2\leq M\dg/\omega^2$. Thus
	$$
	\|\mathcal{G}_{\omega,h}(\gamma_{0}^{u},\theta_{0}^{u})\|_2 \leq \|\mathcal{G}_{\omega,h}(\gamma_{0}^{u},\theta_{0}^{u})-\mathcal{G}_{\omega,h}(0,0)\|_2+\|\mathcal{G}_{\omega,h}(0,0)\|_2
	\leq M\dfrac{\dg}{\omega^2}.
	$$
	Moreover, $\|\mathcal{G}_{\omega,0}(\gamma_{0}^{u},\theta_{0}^{u})\|_2=\|(\gamma_{0}^{u},\theta_{0}^{u})\|_2\leq M\dg/\omega^2$.
	
	Let $v\in D^u$ and first assume that $|h^{1/4}v|\geq 1$, hence
	$$
	\begin{array}{lcl}
	\left|\pi_j(\mathcal{G}_{\omega,h}(\gamma_{0}^{u},\theta_{0}^{u})(v) -\mathcal{G}_{\omega,0}(\gamma_{0}^{u},\theta_{0}^{u})(v))\right|&\leq&\dfrac{\left\|\mathcal{G}_{\omega,h}(\gamma_{0}^{u},\theta_{0}^{u})\right\|_2}{|v^2+2|}+\dfrac{\left\|\mathcal{G}_{\omega,0}(\gamma_{0}^{u},\theta_{0}^{u})\right\|_2}{|v^2+2|}\vspace{0.2cm}\\
	&\leq& M\dfrac{\dg}{\omega^2||v|^2-2|}\vspace{0.2cm}\\	
	&\leq& M\dfrac{\dg}{\omega^2(1/\sqrt{h}-2)}\vspace{0.2cm}\\		
	&\leq& M\dfrac{\dg}{\omega^2}\sqrt{h},		
	\end{array}
	$$
	for $h>0$ sufficiently small, $j=1,2$.
	
	Now, assume that $|h^{1/4}v|<1$, and denote $\Delta_h^j=\pi_j(\mathcal{F}_{h}(\gamma_{0}^{u},\theta_{0}^{u})-\mathcal{F}_{0}(\gamma_{0}^{u},\theta_{0}^{u}))$, $j=1,2$.
	
	Consider the  path $s=e^{-i\beta}\xi$ (since $\Delta_h\in\mathcal{X}_2^2$) and let $\xi_0(v)\in\R$ be such that $v_0(v)=v+e^{-i\bg}\xi_0(v)$ is the unique point of intersecion between the curve $\gamma(\xi)=v+e^{-i\bg}\xi$ and the circle $S_h$ of radius $h^{-1/4}$ centered at the origin.
	$$
	\begin{array}{lcl}
	\left|\pi_1(\mathcal{G}_{\omega,h}(\gamma_{0}^{u},\theta_{0}^{u}) -\mathcal{G}_{\omega,0}(\gamma_{0}^{u},\theta_{0}^{u}))(v)\right|&=&\left|\displaystyle\int_{-\infty}^0e^{\omega is}\Delta_h^1(s+v)ds\right|\vspace{0.2cm}\\
	&=&\left|\displaystyle\int_{-\infty}^0e^{-\omega ie^{-i\bg}\xi}\Delta_h^1(v+e^{-i\bg}\xi)e^{-i\bg}d\xi\right|\vspace{0.2cm}\\
	&\leq&\left|\displaystyle\int_{-\infty}^{\xi_0(v)}e^{-\omega ie^{-i\bg}\xi}\Delta_h^1(v+e^{-i\bg}\xi)e^{-i\bg}d\xi\right|\vspace{0.2cm}\\
	&&+\left|\displaystyle\int_{\xi_0(v)}^0e^{-\omega ie^{-i\bg}\xi}\Delta_h^1(v+e^{-i\bg}\xi)e^{-i\bg}d\xi\right|\vspace{0.2cm}\\
	\end{array}$$
	
	Notice that the points in the path $\gamma(\xi)=v+e^{-i\bg}\xi$ satisfy that $|\gamma(\xi)h^{1/4}|\geq 1$ for every $\xi\leq \xi_0(v)$ and $|\gamma(\xi)h^{1/4}|< 1$ for every $\xi_0(v)<\xi<0$. Also, let $v_0^*(v)=e^{i\beta}v_0(v)$, and notice that $\Ip(v_0^*(v))=\Ip(v)$ and $|h^{1/4}v_0^*(v)|=1$.
	
	\begin{figure}[H]	\label{path}
		\centering
		\begin{overpic}[width=7cm]{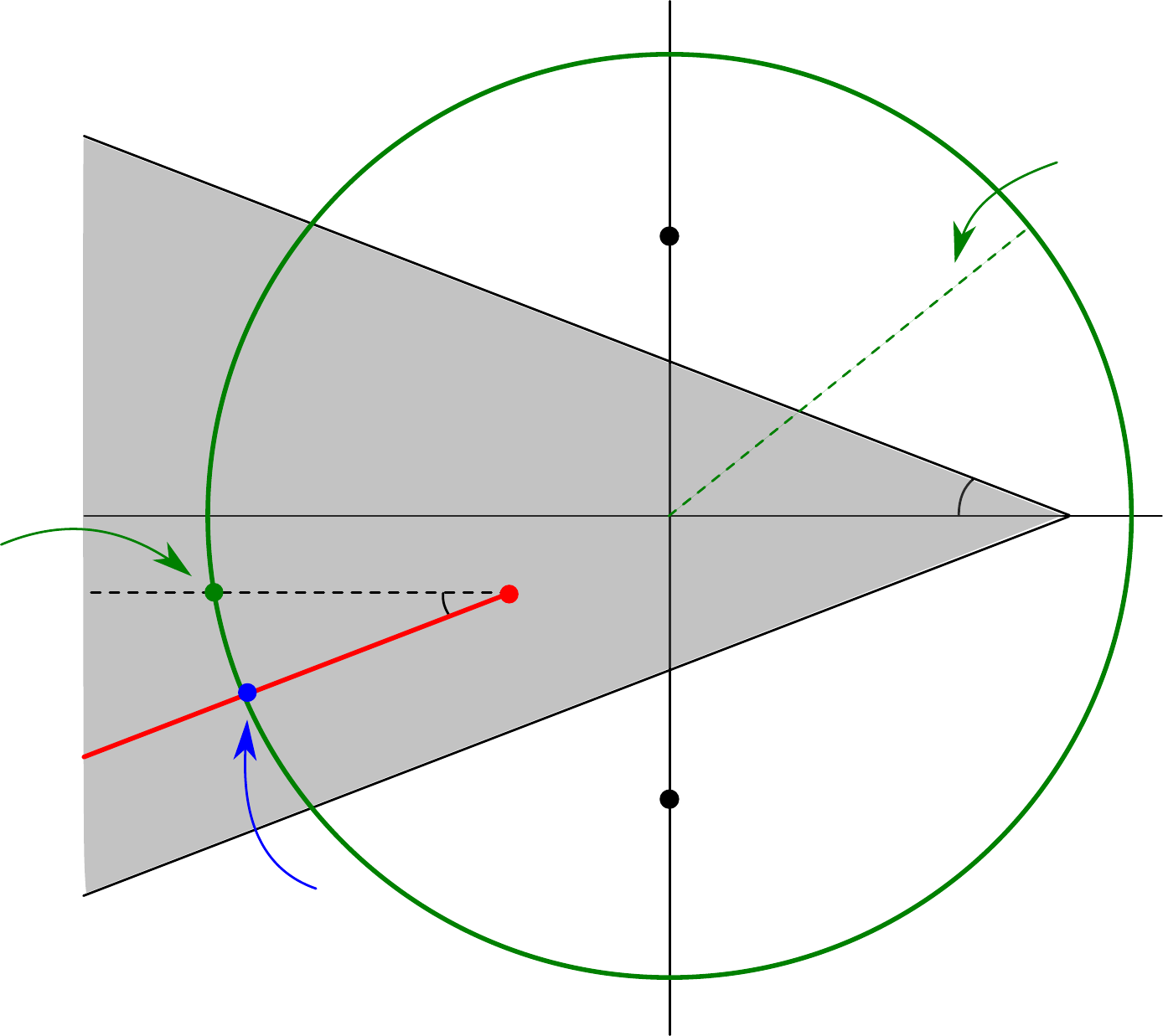}
			\put(91,75){{\footnotesize $h^{-1/4}$}}			
			\put(70,84){{\footnotesize $S_h$}}	
			\put(8,21){{\footnotesize $\gamma(\xi)$}}	
			\put(-10,40){{\footnotesize $v_0^*(v)$}}	
			\put(25.5,9){{\footnotesize $v_0(v)$}}		
			\put(46,36.5){{\footnotesize $v$}}			
			\put(35,35.5){{\scriptsize $\bg$}}			
			\put(78.5,45.5){{\footnotesize $\bg$}}																				
		\end{overpic}
		\bigskip
		\caption{ Definition of the points $v_0(v)$ and $v_0^*(v)$. }	
	\end{figure}
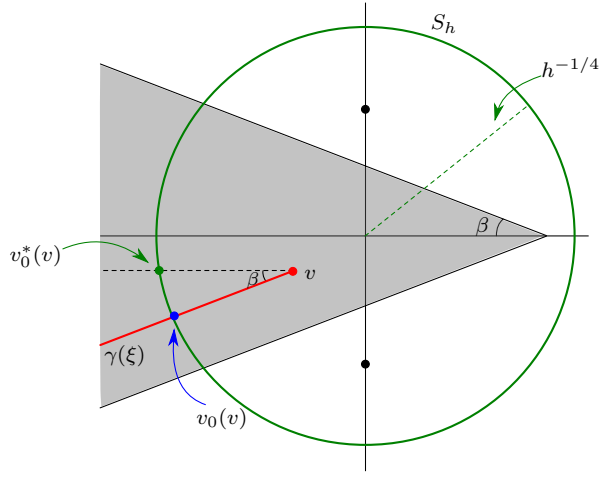 
	
	Thus the first integral satisfies that
	$$
	\begin{array}{lcl}
	\left|\displaystyle\int_{-\infty}^{\xi_0(v)}e^{-\omega ie^{-i\bg}\xi}\Delta_h^1(v+e^{-i\bg}\xi)e^{-i\bg}d\xi\right|&=& \left|\displaystyle\int_{-\infty}^{v_0^*(v)}e^{\omega i(v-r)}\Delta_h^1(r)dr\right|\vspace{0.2cm}\\
	&=& \left|e^{\omega i(v-v_0^*(v))}\displaystyle\int_{-\infty}^{v_0^*(v)}e^{\omega i(v_0^*(v)-r)}\Delta_h^1(r)dr\right|\vspace{0.2cm}\\
	&=& \left|\pi_1(\mathcal{G}_{\omega,h}(\gamma_{0}^{u},\theta_{0}^{u})(v_0^*(v)) -\mathcal{G}_{\omega,0}(\gamma_{0}^{u},\theta_{0}^{u})(v_0^*(v)))\right|\vspace{0.2cm}\\	
	&\leq&M\dfrac{\dg\sqrt{h}}{\omega^2}.
	\end{array}
	$$
	
	Now, since $|\gamma(\xi)h^{1/4}|< 1$ for every $\xi_0(v)<\xi<0$, we can use Lemma \ref{Fhcompacto} to see that the second integral satisfies
	$$
	\begin{array}{lcl}
	\left|\displaystyle\int_{\xi_0(v)}^0e^{-\omega ie^{-i\bg}\xi}\Delta_h^1(v+e^{-i\bg}\xi)e^{-i\bg}d\xi\right| &\leq&\displaystyle\int_{\xi_0(v)}^{0}e^{\omega\sin(\bg)\xi}|\Delta_h^1(v+e^{-i\bg}\xi)|d\xi\vspace{0.2cm}\\
	&\leq&\dfrac{M\dg\sqrt{h}}{\omega}\displaystyle\int_{-\infty}^{0}e^{\omega\sin(\bg)\xi}\dfrac{1}{|(v+e^{-i\bg}\xi)^2+2|}d\xi\vspace{0.2cm}\\	
	&\leq&\dfrac{M\dg\sqrt{h}}{\omega|v^2+2|}\displaystyle\int_{-\infty}^{0}e^{\omega\sin(\bg)\xi}d\xi\vspace{0.2cm}\\	
	&\leq&\dfrac{M\dg\sqrt{h}}{\omega^2|v^2+2|}.
	\end{array}
	$$
	
	The result follows from these bounds.	
\end{proof}

Now, define $\mathcal{E}(v)=( \gamma_{h,0}^{u}(v)- \gamma_{0}^{u}(v), \theta_{h,0}^{u}(v)- \theta_{0}^{u}(v))$  
and notice that 
$$
\left(\begin{array}{lcl}
\Gamma_{h,0}^{u}(0)- \Gamma_{0}^{u}(0)\vspace{0.2cm}\\
\Theta_{h,0}^{u}(0)- \Theta_{0}^{u}(0)
\end{array}\right)= \left(\begin{array}{lcl}
Q^{h}(0)- Q^{0}(0)\vspace{0.2cm}\\
-Q^{h}(0)+ Q^{0}(0)
\end{array}\right)+ \mathcal{E}(0)^{T}.
$$
Using  \eqref{first0} and \eqref{firsth}, we have $Q^{h}(0)=Q^{0}(0)=0$. Hence, to prove Proposition \ref{approxh0}, it is enough to bound  $\|\mathcal{E}\|_0$.  Since $(\gamma_{h,0}^{u},\theta_{h,0}^{u})$ and $(\gamma_{0}^{u},\theta_{0}^{u})$ are fixed points of $\mathcal{G}_{\omega,h}$ and $\mathcal{G}_{\omega,0}$, respectively, 
$$
\begin{array}{lcl}
\mathcal{E} &=& (\gamma_{h,0}^{u},\theta_{h,0}^{u}) -(\gamma_{0}^{u},\theta_{0}^{u}) \vspace{0.2cm}\\
&=& \mathcal{G}_{\omega,h}(\gamma_{h,0}^{u},\theta_{h,0}^{u}) -\mathcal{G}_{\omega,h}(\gamma_{0}^{u},\theta_{0}^{u})+ \mathcal{G}_{\omega,h}(\gamma_{0}^{u},\theta_{0}^{u}) -\mathcal{G}_{\omega,0}(\gamma_{0}^{u},\theta_{0}^{u}).
\end{array}
$$
It follows from Propositions \ref{gammah0} and \ref{gamma00} that
$$
\begin{array}{lcl}
\|\mathcal{E}\|_{0} &\leq& \|\mathcal{G}_{\omega,h}(\gamma_{h,0}^{u},\theta_{h,0}^{u}) -\mathcal{G}_{\omega,h}(\gamma_{0}^{u},\theta_{0}^{u}) \|_0+\| \mathcal{G}_{\omega,h}(\gamma_{0}^{u},\theta_{0}^{u}) -\mathcal{G}_{\omega,0}(\gamma_{0}^{u},\theta_{0}^{u}) \|_0\vspace{0.2cm}\\
&\leq& M\dg^2\|\mathcal{E} \|_0+ \dfrac{M\dg\sqrt{h}}{\omega^2}.
\end{array}
$$
Thus, for $\e_0$ sufficiently small, we have that $\|\mathcal{E}\|_0\leq2 \dfrac{M\dg\sqrt{h}}{\omega^2}$. This completes the proof.

\subsection{Approximation of $W_{\e}^{u}(\Lambda^{-}_{\kappa_1,\kappa_2})$ by $W_{\e}^{u}(p_0^{-})$}\label{approx3}
In this section, we obtain an approximation of $N_{\kappa_1,\kappa_2}^{u}$ by $N_{0,0}^{u}$, by approximating $N_{\kappa_1,\kappa_2}^{u}$ by $N^u_{\kappa_1,0}$ and $N^{u}_{\kappa_1,0}$ by $N_{0,0}^{u}$. 

Proceeding as for Proposition \ref{resta} and Lemma \ref{lipf}, one can obtain the next result.

\begin{prop}\label{restaf}
	Let $\Gamma_{\kappa_1,0}^u(v),$ $\Theta_{\kappa_1,0}^u(v)$ and $\Gamma_{\kappa_1,\kappa_2}^u(v,\tau),$ $\Theta_{\kappa_1,\kappa_2}^u(v,\tau)$ be given in \eqref{solh0} and \eqref{form311}, respectively. There exist  $\e_0>0$, $h_0>0$ and a constant $M>0$  such that, for $v\in D^{u}\cap\R$, $\tau\in\mathbb{T}$, $0\leq\e\leq\e_0$, $0\leq h\leq h_0$ $\kappa_1,\kappa_2\geq0$ with $\kappa_1+\kappa_2=h$,
	\begin{equation}
	\begin{array}{l}
	\left|\partial_\tau(\Gamma_{\kappa_1,\kappa_2}^{u}(v,\tau)-\Gamma_{\kappa_1,0}^{u}(v))\right|,\left|\Gamma_{\kappa_1,\kappa_2}^{u}(v,\tau)-\Gamma_{\kappa_1,0}^{u}(v)\right|\leq M\dfrac{\dg\sqrt{\kappa_2}}{\omega^{3/2}},\vspace{0.2cm}\\
	\left|\partial_\tau(\Theta_{\kappa_1,\kappa_2}^{u}(v,\tau)-\Theta_{\kappa_1,0}^{u}(v))\right|,\left|\Theta_{\kappa_1,\kappa_2}^{u}(v,\tau)-\Theta_{\kappa_1,0}^{u}(v)\right|\leq M\dfrac{\dg\sqrt{\kappa_2}}{\omega^{3/2}}.
	\end{array}
	\end{equation}
	\end{prop}


Notice that Proposition \ref{approxh0} allows us to approximate $N^u_{\kappa_1,0}$ by $N_{0,0}^{u}$, for $\kappa_1$ sufficiently small. Thus, we can combine this fact with Proposition \ref{restaf} to obtain the following proposition.

\begin{prop}\label{restaff}
	Let $\Gamma_{0}^u(v),$ $\Theta_{0}^u(v)$ and $\Gamma_{\kappa_1,\kappa_2}^u(v,\tau),$ $\Theta_{\kappa_1,\kappa_2}^u(v,\tau)$ be given in \eqref{eq1d} and \eqref{form311}, respectively. There exist  $\e_0>0$, $h_0>0$ and a constant $M>0$ such that, 	for  $v\in D^{u}\cap\R$, $\tau\in\mathbb{T}$, $0\leq\e\leq\e_0$, $0\leq h\leq h_0$ and $\kappa_1,\kappa_2\geq0$ with $\kappa_1+\kappa_2=h$,
	\begin{equation}
	\begin{array}{l}
	\left|\Gamma_{\kappa_1,\kappa_2}^{u}(v,\tau)-\Gamma_{0}^{u}(v)\right|,\left|\Theta_{\kappa_1,\kappa_2}^{u}(v,\tau)-\Theta_{0}^{u}(v)\right|\leq M\dfrac{\dg\sqrt{\kappa_2}}{\omega^{3/2}}+M\dfrac{\dg\sqrt{\kappa_1}}{\omega^2},\vspace{0.2cm}\\
	\left|\partial_\tau(\Gamma_{\kappa_1,\kappa_2}^{u}(v,\tau)-\Gamma_{0}^{u}(v))\right|,
	\left|\partial_\tau(\Theta_{\kappa_1,\kappa_2}^{u}(v,\tau)-\Theta_{0}^{u}(v))\right|\leq M\dfrac{\dg\sqrt{\kappa_2}}{\omega^{3/2}}.	
	\end{array}
	\end{equation}
	%
	%
	%
	%
	%
	%
\end{prop}

\end{document}